\documentclass[12pt]{amsart}
\usepackage{amssymb}
\usepackage{amsbsy}
\usepackage{amscd}
\usepackage[mathscr]{eucal}
\usepackage{verbatim}
\usepackage{version}
\usepackage{color}
\usepackage{youngtab}

\usepackage[all]{xy}
\usepackage{mathdots}
\makeatletter
\def\cal{\mathcal}
\def\Bbb{\mathbb}
\def\frak{\mathfrak}

\newenvironment{pf*}[1]{\proof[#1]}{\endproof}
\newcommand{\rom}{\textup}
\makeatother

\newenvironment{aenume}{%
  \begin{enumerate}%
  }{\end{enumerate}}
\makeatletter
\@ifclasslater{amsart}{1999/11/24}{}{
\renewcommand*\subjclass[2][1991]{%
  \def\@subjclass{#2}%
  \@ifundefined{subjclassname@#1}{%
    \ClassWarning{\@classname}{Unknown edition (#1) of Mathematics
      Subject Classification; using '1991'.}%
  }{%
    \@xp\let\@xp\subjclassname\csname subjclassname@#1\endcsname
  }%
}
\renewcommand{\subjclassname}{%
  \textup{1991} Mathematics Subject Classification}
\@xp\let\csname subjclassname@1991\endcsname \subjclassname
\@namedef{subjclassname@2000}{%
  \textup{2000} Mathematics Subject Classification}
}
\makeatother
\newenvironment{NB}{
\color{red}{\bf NB}. \footnotesize 
}{}
\excludeversion{NB}
\newenvironment{NB2}{
\color{blue}{\bf NB}. \footnotesize
}{}

\newtheorem{Theorem}[equation]{Theorem}
\newtheorem{Corollary}[equation]{Corollary}
\newtheorem{Lemma}[equation]{Lemma}
\newtheorem{Proposition}[equation]{Proposition}

\theoremstyle{definition}
\newtheorem{Definition}[equation]{Definition}

\theoremstyle{remark}
\newtheorem{Remark}[equation]{Remark}
\newtheorem{Claim}{Claim}

\numberwithin{equation}{section}

\newcommand{\thmref}[1]{Theorem~\ref{#1}}
\newcommand{\secref}[1]{\S\ref{#1}}
\newcommand{\lemref}[1]{Lemma~\ref{#1}}
\newcommand{\propref}[1]{Proposition~\ref{#1}}
\newcommand{\corref}[1]{Corollary~\ref{#1}}
\newcommand{\subsecref}[1]{\S\ref{#1}}

\newcommand{\defref}[1]{Definition~\ref{#1}}
\newcommand{\remref}[1]{Remark~\ref{#1}}
\newcommand{\C}{{\Bbb C}}
\newcommand{\Z}{{\Bbb Z}}

\newcommand{\R}{{\Bbb R}}
\newcommand{\proj}{{\Bbb P}}

\newcommand{\Spec}{\operatorname{Spec}\nolimits}

\newcommand{\End}{\operatorname{End}}
\newcommand{\Hom}{\operatorname{Hom}}
\newcommand{\Ext}{\operatorname{Ext}}
\newcommand{\Ker}{\operatorname{Ker}}
\newcommand{\Coker}{\operatorname{Coker}}
\newcommand{\Ima}{\operatorname{Im}}
\newcommand{\coker}{\operatorname{coker}}

\newcommand{\im}{\mathop{\text{\rm im}}\nolimits}

\newcommand{\rank}{\operatorname{rk}}

\newcommand{\ve}{\varepsilon}

\newcommand{\linf}{{\ell_\infty}}
\newcommand{\shfO}{\mathcal O}
\newcommand{\bp}{{{\widehat\proj}^2}}
\newcommand{\bM}{{\widehat M}}

\newcommand{\Quot}{\operatorname{Quot}}
\newcommand{\Supp}{\operatorname{Supp}}
\newcommand{\ch}{\operatorname{ch}}
\newcommand{\Wedge}{{\textstyle \bigwedge}}

\newcommand{\Coh}{\operatorname{Coh}}

\newcommand{\q}{\mathfrak q}

\newcommand{\hT}{\widetilde T}

\newcommand{\rk}{\mathop{{\rm rk}}}

\newcommand{\td}{\mathop{\text{\rm td}}}

\def\<{\langle}
\def\>{\rangle}

\newcommand{\codim}{\mathop{\text{\rm codim}}\nolimits}

\newcommand{\HilbX}[1]{X^{[#1]}}
\newcommand{\HilbbX}[1]{\bX^{[#1]}}
\newcommand{\bMz}{\bMm{0}}
\newcommand{\bMs}[1]{\bM_{{}^{#1}\zeta}^{\mathrm{s}}}
\newcommand{\bMm}[1]{\widehat{M}^{#1}}
\newcommand{\Mp}{M^p}
\newcommand{\Tor}{\operatorname{Tor}}
\newcommand{\bN}{\widehat{N}}

\makeatletter
\newcommand{\vechatom}{
    {\Vec{\omega}}
    \,\smash[b]{\hbox{\lower2\ex@\hbox{$\m@th\hat{\null}$}}}
}
\makeatother

\newcommand{\hf}{\hfil}
\newcommand{\hs}{
{\heartsuit}}
\newcommand{\sps}{
{\spadesuit}}
\newcommand{\pt}{\operatorname{pt}}

\newcommand{\bX}{{\widehat X}}
\newcommand{\Per}{\operatorname{Per}}
\newcommand{\bR}{\mathbf R}
\newcommand{\bD}{\mathbf D}
\newcommand{\bL}{\mathbf L}

\usepackage[bookmarks=false]{hyperref}

\begin{document}

\title[Perverse coherent sheaves on blow-up. II]
{Perverse coherent sheaves on blow-up. II.
\\ wall-crossing and Betti numbers formula
}
\author{Hiraku Nakajima}
\address{Department of Mathematics, Kyoto University, Kyoto 606-8502,
Japan}
\email{nakajima@math.kyoto-u.ac.jp}

\author{K\={o}ta Yoshioka}
\address{Department of Mathematics, Faculty of Science, Kobe University,
Kobe 657-8501, Japan}
\email{yoshioka@math.kobe-u.ac.jp}

\subjclass[2000]{Primary 14D21; Secondary 16G20}

\begin{abstract}
This is the second of series of papers studyig moduli spaces of a
certain class of coherent sheaves, which we call {\it stable perverse
  coherent sheaves}, on the blow-up $p\colon \bX\to X$ of a projective
surface $X$ at a point $0$.

The followings are main results of this paper:
\begin{aenume}
\item We describe the wall-crossing between moduli spaces caused by
  twisting of the line bundle $\shfO(C)$ associated with the
  exceptional divisor $C$.
\item We give the formula for virtual Hodge numbers of moduli spaces of
  stable perverse coherent sheaves.
\end{aenume}
Moreover we also give proofs of the followings which we observed in a
special case in \cite{perv}:
\begin{aenume}
\setcounter{enumi}{2}
\item The moduli space of stable perverse coherent
sheaves is isomorphic to the usual moduli space of stable coherent
sheaves on the {\it original surface\/} if the first Chern class is
orthogonal to $[C]$.
\item The moduli space becomes isomorphic to the usual moduli space of
  stable coherent sheaves on the {\it blow-up\/} after twisting by
  $\shfO(-mC)$ for sufficiently large $m$.
\end{aenume}
Therefore usual moduli spaces of stable sheaves on the blow-up and the
original surfaces are connected via wall-crossings.
\end{abstract}

\maketitle
\section*{Introduction}

This paper is formally a sequel, but is independent of our previous
paper \cite{perv} except in \subsecref{subsec:quiver}. All the rest do
not depend on results in \cite{perv}, though motivation to various
definitions come from \cite{perv}. The result in
\subsecref{subsec:quiver} is independent of other parts of the paper.
See also the comment below.

%
Let $p\colon \bX\to X$ be the blow-up of a projective surface $X$ at a
point $0$. Let $C$ be the exceptional divisor.
Let $\shfO_X(1)$ be an ample line bundle on $X$. A {\it stable
  perverse coherent sheaf\/} $E$ on $\widehat X$, with respect to
$\shfO_X(1)$, is
\begin{enumerate}
\item
$E$ is a coherent sheaf on $\widehat X$,
\item
$\Hom(E,{\cal O}_C(-1))=0$,
\item
$\Hom(\shfO_C,E) = 0$,
\item $p_*E$ is $\mu$-stable with respect to $\shfO_X(1)$.
\end{enumerate}
As was explained in \cite{perv}, this definition came from two
sources, a work by Bridgeland \cite{Br:4} and one by King
\cite{King}. In \cite{perv} the latter was explained in detail, and
the first will be explained in this paper.

For a given integer $m\in \Z$ and homological data $c\in H^*(\bX)$, we
will consider the moduli space $\bMm{m}(c)$ of coherent sheaves $E$
with $\ch(E) = c$ such that $E(-mC)$ is stable perverse coherent.
We assume that $(c_1,p^*\shfO_X(1))$ and $r$ are coprime, so the
$\mu$-stability and $\mu$-semistability are equivalent on $X$.
Then we construct varieties $\bMm{m,m+1}(c)$ connecting various
$\bMm{m}(c)$ by the diagram
\begin{equation}
  \label{eq:flip}
\cdots \setbox5=\hbox{$\bMm{m,m+1}(c)$}{\rule{0mm}{\ht5}\hspace*{-\wd5}}
\begin{aligned}[m]
  \xymatrix@R=.5pc{ &
\bMm{m}(c) \ar[rd]^{\xi_{m}} \ar[ld] & & 
\bMm{m+1}(c) \ar[rd]^{\xi_{m+1}} \ar[ld]_{\xi^+_{m}} & 
& 
\ar[ld]
\\
\setbox5=\hbox{$\bMm{m,m+1}(c)$}{\rule{0mm}{\ht5}\hspace*{\wd5}}
&& \bMm{m,m+1}(c) & & \bMm{m+1,m+2}(c) &
}
\end{aligned}
\cdots
\tag{$*$}
\end{equation}
The morphism $\xi^+_m$ is a kind of `flip' of $\xi_m$. (See
\propref{prop:ample} for the precise statement.)
This kind of the diagram appears often in the variation of GIT
quotients \cite{Th} and moduli spaces of sheaves 
(by Thaddeus, Ellingsrud-G\"ottsche, Friedman-Qin and others)
when we move ample line bundles. 

Furthermore, $\bMm{m,m+1}(c)$ will be constructed as the Brill-Noether
locus in the moduli space $M^X(p_*(c))$ of stable sheaves on $X$, and
the fibers of $\xi_m$, $\xi^+_m$ over $F\in M^X(p_*(c))$ are Grassmann
varieties consisting of subspaces
\(
   V \subset\Hom(\shfO_C(-m-1), p^* F)
\)
and
\allowbreak
\(
  U \allowbreak \subset\allowbreak \Hom(p^*F, \shfO_C(-m-1))
\)
of $\dim V = (c_1,[C]) + m$, $\dim U = (c_1,[C])+m+r$ respectively.
The dimensions of spaces of homomorphisms depend on the sheaf $F$, so
$\xi_m$, $\xi^+_m$ are {\it stratified\/} Grassmann bundles. This
looks similar to the picture observed in the context of quiver
varieties \cite{Na:1994} and exceptional bundles on K3
\cite{Yos,Mar} (see also \cite{Na:Missouri} for an exposition). But
there is a sharp distinction between the blowup case and these cases.
In the other cases, the spaces of homomorphisms (or extensions)
appearing in the fibers of $\xi_m$ and $\xi_m^+$ are {\it dual\/} to
each other, and $\dim U = \dim V$, so that two varieties are related
by the stratified Mukai flop.
However, our spaces $\Hom(\shfO_C(-m-1), p^* F)$ and $\Hom(p^*
F,\shfO_C(-m-1))$ have different dimensions, and $\dim U\neq \dim V$. (See
also \remref{rem:Grass} for another difference.)

Next we consider the formula for (virtual) Hodge numbers. %
\begin{NB}
May 7:

  We do not compute actual Betti numbers, but only virtual Hodge
  numbers. So I change it.
\end{NB}%
This study was not originally planned when we started this research
project, and is motivated by recent works on wall-crossings of
Donaldson-Thomas invariants \cite{BTL,KS}. It turns out to be a simple
application of techniques developed for the blow-up formula for
virtual Hodge polynomials in \cite[Th.~3.13]{NY2}.

Since the formula becomes complicated in higher rank cases, we
consider the rank $1$ case, where $\bMz(c)$ (resp.\ $\bMm{m}(c)$) is
the Hilbert scheme $\HilbX{N}$ (resp.\ $\HilbbX{N}$) of $N$ points in
$X$ (resp.\ $\bX$ for sufficiently large $m$ depending on $N$).
Then we have the formula for the generating function of Hodge
polynomials
\begin{equation}\label{eq:Betti}
   \sum_{N=0}^\infty P_{x,y}(\bMm{m}(c)) \q^N
   = \left(\sum_{N=0}^\infty P_{x,y}(\HilbX{N}) \q^N\right)
   \left(\prod_{d=1}^m \frac1{1 - (xy)^{2d} \q^d}\right),
\tag{$**$}
\end{equation}
where $c = 1 - N\pt$. When $m\to \infty$, the left hand side converges
to
\(
   \sum_{N=0}^\infty P_{x,y}(\HilbbX{N}) \q^N
\)
as we have just remarked. Then the above formula is compatible with
the {\it famous\/} G\"ottsche formula\footnote{We learned this naming
  from Atsushi Takahashi.} of Betti numbers of Hilbert
schemes of points of surfaces \cite{Go}. Thus factors of the
infinite product of the Dedekind $\eta$-function appear one by one
when we cross walls.

Moreover, in this rank $1$ case, $\bMm{1}(c)$ is isomorphic to the
nested Hilbert scheme of $N$ and $N+1$ points in $X$, where two
subschemes differ only at $0$. The above formula coincides with
Cheah's formula \cite[Theorem~3.3.3(5)]{Cheah} in this special case.
However our $\bMm{m}(c)$ for $m\ge 2$ seems new even in rank $1$ case.
In particular, they are different from incidence varieties used to
define Heisenberg generators in \cite[Chap.~8]{Lecture}.

In higher rank cases, we have the formula relating virtual Hodge
polynomials of $\bMm{m}(c)$ and $M^X(p_*(c))$. (See
\corref{cor:higherrk}.) In the limit $m\to\infty$, the formula
converges to the blow-up formula for virtual Hodge polynomials in
\cite[Th.~3.13]{NY2}. (See also \cite[Rem.~3.14]{NY2} for earlier
works.)

Similar wall-crossing formulae for Donaldson-Thomas invariants in
$3$-dimensional situation of \cite{Br:4} will be discussed elsewhere
(\cite{DT}).

Finally let us comment on the quiver description in our first paper
\cite{perv}. The most of materials in this paper can be worked out in
the language of quiver representations. In fact, the constructions of
moduli spaces and the diagram \eqref{eq:flip} are automatic in that
setup, and the assertion that the fibers are Grassmann varieties are
easy to prove.  The only missing is the isomorphism $\bMm{m}(c) \cong
\bMm{0}(ce^{-m[C]})$ induced by the tensor product by $\shfO(-mC)$. We
do not know how to construct the isomorphism explicitly in terms of
quivers. This, if it is possible, would be given by analog of
reflection functors, developed by the first-named author in the
context of quiver varieties \cite{Na:Reflect}.

The paper is organized as follows.
In \secref{sec:perverse} we study the category of perverse coherent
sheaves $\Per(\bX/X)$ which is the heart of the $t$-structure in the
derived category $\bD(\bX)$ of coherent sheaves on $\bX$, introduced
in more general setting in \cite{Br:4}. One of the main results in
this section is a simple criterion when a coherent sheaf $E$ is
perverse coherent (see \propref{prop:pervblowup}(1)).
In \secref{sec:moduli} we construct moduli spaces of perverse coherent
sheaves in the general context in \cite{Br:4}. One of key observations
is that though perverse coherent sheaves are objects in $\bD(\bX)$ in
general, they are genuine sheaves if we impose the stability and the
assumption on the dimension of their supports. This was already
observed in \cite{Br:4} in the case of perverse ideal sheaves.
Combined with the result in \secref{sec:perverse} we get the
conditions (1)$\sim$(4) in the blow-up case.
In \secref{sec:wall-crossing} we construct the diagram
\eqref{eq:flip}. Our tools here are Brill-Noether loci and moduli
spaces of coherent systems, which had been used in different settings
as we mentioned above. 
In \secref{sec:incidence} we show that $\bMz(c)$ is an incidence
variety in the product of two moduli spaces
$M^X(p_*(c))\times M^X(p_*(c)+n\pt)$ ($n = (c_1,[C])$).
In \secref{sec:Betti} we give the formula for virtual Hodge numbers of
$\bMm{m}(c)$. The proof goes like that of \cite[Th.~3.13]{NY2}. We
observe that the formula is universal, i.e.\ is independent of the
surface $X$, and is enough to compute it in the moduli of framed
sheaves. Then we can use a torus action to deduce it from a
combinatorial study of fixed points. The combinatorics involves Young
diagrams and removable boxes, which is closely related to one
appearing in the Pieri formula (but only for the multiplication by
$e_1$~!) for Macdonald polynomials \cite[\S VI.6]{Mac}.

\subsection*{Acknowledgements}
The first named author is supported by the Grant-in-aid for Scientific
Research (No.\ 19340006), JSPS. A part of this work was done while
the first named author was visiting the Institute for Advanced Study
with supports by the Ministry of Education, Japan and the Friends of
the Institute.
We are grateful to Y.~Soibelman for sending us a preliminary version
of \cite{KS}.

\subsection*{Notations}

$\bD(X)$ denotes the unbounded derived category of coherent sheaves on
a variety $X$. The full subcategory of complexes with bounded
cohomology sheaves is denoted by $\bD^b(X)$.

We consider a blowup $p\colon \bX\to X$ of a smooth projective surface
$X$ at a point $0\in X$. But occasionaly we consider a general
situation where $p\colon Y\to X$ is a birational morphism of
projective varieties such that $\bR p_*(\shfO_Y) = \shfO_X$ and $\dim
p^{-1}(x)\le 1$ for any $x\in X$.

When we write $\shfO$ without indicating the variety, it means the
structure sheaf of $\shfO_{\bX}$.

Let $C = p^{-1}(0) \subset \bX$ denote the exceptional divisor. Let
$\shfO(C)$ the line bundle associated with $C$, and $\shfO(mC)$ its
$m^{\mathrm{th}}$ tensor product $\shfO(C)^{\otimes m}$ when $m > 0$,
$\left(\shfO(C)^{\otimes -m}\right)^\vee$ if $m < 0$, and $\shfO$ if
$m=0$.

The structure sheaf of the exceptional divisor $C$ is denoted by
$\shfO_C$. If we twist it by the line bundle $\shfO_{\proj^1}(n)$ over
$C \cong\proj^1$, we denote the resulted sheaf by $\shfO_C(n)$. Since
$C$ has the self-intersection number $-1$, we have
$\shfO_C\otimes \shfO(C) = \shfO_C(-1)$.

For $c\in H^*(\bX)$, its degree $0$, $2$, $4$-parts are denoted by
$r$, $c_1$, $\ch_2$ respectively. If we want to specify $c$, we denote
by $r(c)$, $c_1(c)$, $\ch_2(c)$.

We also use the following notations often:
\begin{itemize}
\item $\rk E$ is the rank of a coherent sheaf $E$.
\item $e := \ch(\shfO_C(-1))$.
\item $\pt$ is a single point in $X$ or $\bX$. Its Poincar\'e dual
in $H^4(X)$ or $H^4(\bX)$ is also denoted by the same notation.
\item $\chi(E,F) := \sum_{i=-\infty}^{\infty} (-1)^i \dim \Ext^i(E,F)
= \sum_{i=-\infty}^{\infty} (-1)^i \dim \Hom(E,F[i])$.
\item $h^0(E, F) := \dim \Hom(E,F)$.
\item $\chi(E) := \chi(\shfO_{\bX}, E)$.
\item $h^0(E) := h^0(\shfO_{\bX}, E)$.
\end{itemize}

\begin{NB}
Serre duality:
\begin{equation*}
\begin{split}
  & \Ext^i(E,\shfO_C(-1)) \cong \Ext^{2-i}(\shfO_C(-1),E\otimes K)^\vee
  \cong \Ext^{2-i}(K^\vee\otimes\shfO_C(-1),E)^\vee
  \cong \Ext^{2-i}(\shfO_C,E)^\vee,
\\
  & \Ext^i(\shfO_C(-1),E) \cong \Ext^{2-i}(E, \shfO_C(-1)\otimes K)^\vee
  \cong \Ext^{2-i}(E, \shfO_C(-2))^\vee,
\end{split}
\end{equation*}
as $K|_C \cong \shfO_C(-1)$.

Riemann-Roch:
\begin{gather*}
\ch(\shfO_C(-1)) = [C] - \frac12 \pt,
\\
\td{\bX} = 1 + \frac12 c_1(\bX) + \frac1{12}(c_1(\bX)^2 + c_2(\bX).
\end{gather*}
Therefore
\begin{equation*}
\begin{split}
  & \chi(E,\shfO_C(-1)) = \int_{\bX} \ch(E)^\vee \ch(\shfO_C(-1))\td \bX
  = - (c_1(E), [C]),
\\
  & \chi(\shfO_C(-1),E) = \int_{\bX} \ch(E)\ch(\shfO_C(-1))^\vee \td \bX
  = - (c_1(E), [C]) - \rk E,
\end{split}
\end{equation*}
as $c_1(\bX)\cdot [C] = \pt$.
\end{NB}

\section{Perverse coherent sheaves on blow-up}
\label{sec:perverse}

\subsection{General situation}\label{subsec:general}

Let $p\colon Y\to X$ be a birational morphism of projective varieties
such that $\bR p_*\shfO_Y = \shfO_X$ and $\dim p^{-1}(x) \le 1$ for
any $x\in X$. This is the assumption considered to define {\it
  perverse coherent sheaves\/} in \cite{Br:4}.
\begin{NB}
I move the following definition here.  
\end{NB}%
We set $Z:=\{ x \in X \mid \dim p^{-1}(x)=1 \}$. Then $p^{-1}(Z)$ is
the exceptional locus of $p$.
The example we have in mind is the blowup of a projective surface $X$
at a smooth point $0\in X$, but we review the arguments in
\cite{Br:4} for the completeness in this subsection.

\begin{Definition}[\protect{\cite[3.2]{Br:4}}]\label{def:perverse}
  Let $\Per(Y/X)$ be the full subcategory of ${\bf D}(Y)$
  consisting of objects $E \in {\bf D}(Y)$ satisfying the following
  conditions:
\begin{enumerate}
\item
$H^i(E)=0$ for $i \ne -1,0$,
\item
$R^0 p_*(H^{-1}(E))=0$ and $R^1 p_*(H^0(E))=0$,
\item
$\Hom(H^0(E),K)=0$ for any sheaf $K$ on $Y$ with
$\bR p_*(K)=0$.%
\begin{NB}
I change the notation from $c$ to $K$ to avoid the confusion with the
Chern character.  
\end{NB}
\end{enumerate}
An object $E \in \Per(Y/X)$ is called a {\it perverse coherent sheaf}.
\end{Definition}

By \cite[\S\S2,3]{Br:4} $\Per(Y/X)$ is the heart of a $t$-structure on
${\bD}(Y)$, and in particular, is an abelian category. This will be
reviewed below. An object $E\in\Per(Y/X)$ satisfies $H^i(\bR p_* (E))
= 0$ for $i\neq 0$. Thus $\bR p_*(E)\in\Coh(X)$.
\begin{NB}
In \cite{Kota} it seems that $\bR$ is missing.
\end{NB}

\begin{Lemma}[cf.\ \protect{\cite[5.1]{Br:4}}]\label{lem:Bridgeland}
\textup{(1)}
For a coherent sheaf $F$ on $X$, we have an exact sequence
$$
0 \to R^1 p_*(L^{-1} p^* (F)) \to F \to p_* p^*(F)
\to 0.
$$
Moreover we have $p^*(F)\in\Per(Y/X)$.
\begin{NB}
I add this statement.  
\end{NB}%
Furthermore, $F \cong p_* p^*(F)$ if $F$ is torsion free.

\textup{(2)}
Let $E$ be a coherent sheaf on $Y$. For a natural homomorphism
$\phi\colon p^* p_*(E) \to E$, we have
\textup{(i)} $\bR p_*(\Ker\phi) = 0$, 
\textup{(ii)} $p_*(\Ima\phi)\to p_*(E)$ is isomorphic, 
\textup{(iii)} $p_*(\Coker\phi) = 0$, 
\textup{(iv)} $R^1 p_*(\Ima\phi) = 0$,%
\begin{NB}
I add this statement.  
\end{NB}
$R^1 p_*(E) \cong R^1 p_*(\Coker\phi)$.

\textup{(3)} A coherent sheaf $E$ belongs to $\Per(Y/X)$ if and only if 
$\phi\colon p^* p_*(E) \to E$ is surjective.

\textup{(4)} For a coherent sheaf $F$ on $X$, we have
$\Ext^1(p^*(F),K)=0$ for all $K \in \Coh(Y)$ with $\bR p_*(K)=0$.
\end{Lemma}

\begin{proof}
(1) The first assertion is a consequence of the projection formula
  $\bR p_*(\bL p^*(F)) = F$ and the spectral sequence
\begin{equation*}
    R^p p_*(L^q p^*(F)) \Rightarrow H^{p+q}(\bR p_*(\bL p^*(F)).
\end{equation*}
We also get $R^1 p_*(p^*(F)) = 0$ at the same time.
Now we have $\Hom(p^*(F),K) = \Hom(F,p_*(K)) = 0$ for $K\in\Coh(Y)$
with $\bR p_*(K) = 0$. Therefore $p^*(F)$ is perverse coherent.
\begin{NB}
I add the following proof.  
\end{NB}%
For the last assertion we note that $R^1 p_*(L^{-1} p^*(F))$ is
supported on $p^{-1}(Z)$, and hence is torsion.

(2) We have exact sequences
\begin{gather*}
  0 \to p_*(\Ker\phi) \to p_*(p^*(p_*(E))) \to p_*(\Ima\phi) 
  \to R^1 p_*(\Ker\phi) \to 0,
\\
\xymatrix@R=.8pc{
  0 \ar[r] & p_*(\Ima\phi) \ar[r] & p_*(E) \ar[r] & p_*(\Coker \phi) \ar[lld]
\\
           & R^1 p_*(\Ima\phi) \ar[r] & R^1p_*(E) \ar[r] 
           & R^1 p_*(\Coker\phi) \ar[r] & 0,
}
\end{gather*}
where we have used $R^1 p_*(p^*(p_*(E))) = 0$ from (1) in the first
exact sequence.
Since the composition $p_*(E) \to p_*(p^*(p_*(E))) \to p_*(E)$ is the
identity, (1)
\begin{NB}
applied to $F = p_*(E)$  
\end{NB}%
implies that both homomorphisms are isomorphisms. Therefore
$p_*(\Ima\phi) \to p_*(E)$ is also an isomorphism.
We have $\bR p_*(\Ker\phi) = 0$ from the first exact sequence.

Since $R^1 p_*(\Ima\phi) = 0$ follows from $R^1 p_*(p^*(p_*(E))) = 0$,
the second exact sequence gives $p_*(\Coker\phi) = 0$ and
$R^1p_*(E) \cong R^1p_*(\Coker\phi)$.

(3) Suppose $E\in\Coh(Y)$ and $\phi\colon p^* p_*(E) \to E$ is
surjective. We have $R^1 p_*(E) = 0$ from (2)(iv). We also have
\(
   0 \to \Hom(E,K) \to \Hom(p^*(p_*(E)), K)
\)
and
\(
   \Hom(p^*(p_*(E)), K) = \Hom(p_*(E),p_*(K)) = 0
\)
for a sheaf $K$ with $\bR p_*(K) = 0$. Therefore $E\in\Per(Y/X)$.

Conversely suppose $E\in\Per(Y/X)\cap\Coh(Y)$. By (2)(iii),(iv) we
have $\bR p_*(\Coker\phi) = 0$. By \defref{def:perverse}(3) we
have $0 = \Hom(E,\Coker\phi)$, i.e.\ $\Coker\phi = 0$.

(4) We consider a distinguished triangle
\(
  \bL^{< 0}p^* F \to \bL p^* F \to p^* F \to \bL^{< 0}p^* F[1].
\)
We apply the functor $\Hom(\bullet,K)$ to get an exact sequence
\[
  \Hom(\bL^{< 0}p^* F[1],K[1]) \to \Hom(p^*F,K[1]) \to \Hom(\bL p^* F,K[1]).
\]
We have
\begin{gather*}
  \Hom(\bL p^* F,K[1])=\Hom(F,\bR p_* (K[1]))=
  \Hom(F, \bR p_*(K)[1]) = 0,
\\
  \Hom(\bL^{< 0}p^* F[1],K[1]) = \Hom(\bL^{< 0}p^* F, K)=0,
\end{gather*}
where the latter follows from the degree reason. We therefore have
$\Hom(p^* F,K[1])=\Ext^1(p^* F,K)=0$.
\end{proof}

\begin{NB}
For $F$ on $X$, $F \to p_*(p^* F)$ is surjective.
Indeed let 
$\cdots \to W \overset{\phi}{\to} V \to F \to 0$ be a locally free resolution
of $F$.
Then we have an exact sequence
$p^*(W) \overset{p^*(\phi)}{\to} p^*(V) \to p^*(F) \to 0$.
Since $p^*(W) \to \Ima p^*(\phi)$ is surjective,
$R^1 p_*(\Ima p^*(\phi))=0$.
Hence we have an exact sequence
$0 \to p_*(\Ima p^*(\phi)) \to 
p_*(p^*(V)) \to p_*(p^*(F)) \to 0$,
which implies that $F \to p_*(p^*(F))$ is surjective.

About (4) : $\Ext^1(p^*(F),K)=0$ for all $K$ with $\bR p_*(K)=0$.
We give other proofs.
(I) By the exact sequence
$\Hom(\Ima p^*(\phi),K) \to 
\Ext^1(p^*(F),K) \to \Ext^1(p^*(V),K)$
and the surjection
$p^*(W) \to \Ima p^*(\phi)$,
it is sufficient to prove that
$\Ext^1(p^*(V),K)=\Hom(p^*(W),K)=0$.
But these follows from ${\bR} p_*(K)=0$. 

\begin{NB2}
I modified the argument according to Kota's message on Apr.\ 8.
\end{NB2}
(II)
Suppose that there is an extension
\begin{equation}
0 \to K \to E \to p^*(F) \to 0.
\end{equation}
Then $p_*(E) \to p_* p^*(F)$ is an isomorphism since $\bR p_*(K) = 0$.
Note the composition of $p^*(F) \to p^*(p_*(p^* F)) \to p^*(F)$ is an
identity and $p^*(F)\to p^*(p_*(p^* F))$ is surjective by
(1). Therefore $p^*(F) \cong p^*(p_*(p^* F))$. Combined with above, we
have $p^*(p_*(E))\cong p^*(F)$. Composite with $p^*p_*(E) \to E$, we
have a splitting $p^*(F) \to E$.
\end{NB}

Let 
\begin{gather*}
 \mathcal C := \{ K\in \Coh(Y) \mid \bR p_*(K) = 0\},
\\ 
  \mathcal T := \{ E\in \Coh(Y) \mid R^1 p_*(E) = 0, \;
  \text{$\Hom(E,K) = 0$ for all $K\in\mathcal C$}\},
\\
  \mathcal F :=  \{ E\in \Coh(Y) \mid p_*(E) = 0 \}.
\end{gather*}
From the above definition, we have 
\(
   \Per(Y/X) = \{ E\in \bD(Y) \mid 
   \text{$H^i(E) = 0$ for $i\neq 0,-1$},
   \linebreak[1]
   H^{-1}(E) \in\mathcal F,
   \linebreak[1]
   H^0(E) \in \mathcal T \}.
\)

Then the definition of $\Per(Y/X)$ is an example of a general
construction in \cite[\S2]{HRS}:

\begin{Lemma}
$(\mathcal T,\mathcal F)$ is a torsion pair on $\Coh(Y)$ in the sense
of \cite[\S2]{HRS}.
\end{Lemma}

\begin{proof}
We check two assertions: (i) $\Hom(T,F) = 0$ for $T\in\mathcal T$,
$F\in\mathcal F$, (ii) for any $E\in\Coh(Y)$, there exists an exact
sequence $0\to T\to E\to F\to 0$ with $T\in\mathcal T$, $F\in\mathcal
F$.

(i) By \lemref{lem:Bridgeland}(3), we have
$p^*p_*(T) \twoheadrightarrow T$. Therefore
$\Hom(T,F) \subset \Hom(p^* p_*(T),F) = \Hom(p_*(T),p_*(F)) = 0$
for $T\in\mathcal T$, $F\in\mathcal F$.

(ii) For $E\in\Coh Y$, let us consider the exact sequence
$0\to\Ima\phi\to E\to \Coker \phi \to 0$ for $\phi$ as in
\lemref{lem:Bridgeland}(2). We have $R^1 p_*(\Ima\phi) = 0$ and
$p_*(\Coker\phi) = 0$ by (2)(iii),(iv). We also have $\Hom(\Ima\phi,K)
\subset \Hom(p^* p_*(E),K) = \Hom(p_*(E),p_*(K)) = 0$ for
$K\in\mathcal C$. Therefore $\Ima\phi\in\mathcal T$,
$\Coker\phi\in\mathcal F$.
\end{proof}

An exact sequence $0 \to A\to B\to C\to 0$ in $\Per(Y/X)$ is a
distinguished triangle $A\to B\to C\to A[1]$ in $\bD(Y)$ such that all
$A$, $B$, $C\in \Per(Y/X)$. Hence it induces an exact sequence
\[
   0\to H^{-1}(A) \to H^{-1}(B)\to H^{-1}(C) 
    \to H^0(A) \to H^{0}(B)\to H^{0}(C) \to 0
\]
in $\Coh(Y)$.
From a general theory, if $A\to B\to C\to A[1]$ is a distinguished
triangle in $\bD(Y)$ such that all
$A$, $C\in \Per(Y/X)$, then $B\in\Per(Y/X)$.
We have $\Ext^i_{\Per(Y/X)}(A,B) = \Hom_{\bD(Y)}(A,B[i])$ for $A,
B\in\Per(Y/X)$, $i=0,1$ (\cite[Cor.~2.2(c)]{HRS}).
It was proved that $\bD^b(Y) \cong \bD^b(\Per(Y/X))$ in \cite{VB}, but
we will not use it in this paper. 

\begin{Remark}
  Let $E\in\Per(Y/X)\cap\Coh(Y)$. By \lemref{lem:Bridgeland}(3) we
  have the exact sequence $0\to \Ker\phi\to p^*p_*(E)\to E\to 0$ in
  the category $\Coh(Y)$.  This gives a distinguished triangle
  $\Ker\phi\to p^*p_*(E)\to E \to \Ker\phi[1]$. Now notice that $p^*
  p_*(E)$, $E$, $\Ker\phi[1]\in \Per(Y/X)$. Therefore we have $0\to
  p^*p_*(E) \to E \to \Ker\phi[1]\to 0$ in the category of
  $\Per(Y/X)$.
\end{Remark}

\begin{Lemma}\label{lem:perv}
Let $E$, $F$ be objects in $\Per(Y/X)$, and hence
$\bR p_*(E)$, $\bR p_*(F)\in\Coh(Y)$.

\textup{(1)} Assume that $H^i(E) = 0$ for $i\neq 0$ and $p_*(E) =
0$. Then $E = 0$.

\textup{(2)} a homomorphism $\xi\colon E\to F$ is injective in
$\Per(Y/X)$ if and ony if $H^{-1}(E)\to H^{-1}(F)$ is injective in
$\Coh(Y)$ and $\bR p_*(E)\to \bR p_*(F)$ is injective in $\Coh(X)$.
\end{Lemma}

\begin{proof}
(1) Since $\bR p_*(E)=0$ from the assumption and the
\defref{def:perverse}(2), we have
$\Hom(E,E) = 0$ by \defref{def:perverse}(3). Thus $E=0$.

(2) We first assume that $E \to F$ is injective in $\Per(Y/X)$. Since
$\Per(Y/X)$ is an abelian category, we have an exact sequence
\begin{equation*}
0 \to E \to F \to G \to 0,\quad G := \Coker\xi \in \Per(Y/X).
\end{equation*}
Hence $H^{-1}(E) \to H^{-1}(F)$ is injective in $\Coh(Y)$
and we have an exact sequence in $\Coh(X)$:
\begin{equation*}
0 \to \bR p_*(E) \to \bR p_*(F) \to \bR p_*(G) \to 0,
\end{equation*}
as $\bR p_*(E)$, $\bR p_*(F)$, $\bR p_*(G)\in\Coh(X)$.
\begin{NB}
More precisely, we have an exact triangle
$\bR p_*(E) \to \bR p_*(F) \to \bR p_*(G) \to \bR p_*(E)[1]$.
Hence we get an exact sequence
$0 \to H^0(\bR p_*(E)) \to H^0(\bR p_*(F)) \to H^0(\bR p_*(G)) \to 0$
as $\bR p_*(E)$, $\bR p_*(F)$, $\bR p_*(G)\in\Coh(X)$.
\end{NB}

Conversely, we assume that 
$H^{-1}(E) \to H^{-1}(F)$ is injective in $\Coh(Y)$ and
$\bR p_*(E) \to \bR p_*(F)$ is injective in $\Coh(X)$. 
Let $K \in \Per(Y/X)$ be the kernel of $\xi$ in $\Per(Y/X)$.
Then $H^{-1}(K)=0$ and we have an exact sequence
\begin{equation*}
0 \to \bR p_*(K) \to \bR p_*(E) \to \bR p_*(F).
\end{equation*}
Hence $\bR p_*(K)=0$.
\begin{NB}
  More precisely, consider exact sequences
\(
  0 \to K \to E \to \Ima\xi \to 0,
\)
\(
  0 \to \Ima\xi \to F \to \Coker\xi \to 0
\)
in $\Per(Y/X)$. We have
\begin{gather*}
  0 \to \bR p_*(K) \to \bR p_*(E) \to \bR p_*(\Ima\xi) \to 0,
\\
  0 \to \bR p_*(\Ima\xi) \to \bR p_*(F) \to \bR p_*(\Coker\xi) \to 0.
\end{gather*}
From the assumption $\bR p_*(E)\to\bR p_*(F)$ is injective, we have
$\bR p_*(E) \to \bR p_*(\Ima\xi)$ is injective and $\bR p_*(K) = 0$.
\end{NB}%
By (1), we get $K=0$.
\end{proof}

\begin{Lemma}\label{lem:subsheaf}
  Let $E\in\Coh(Y)$ and let $F$ be a subsheaf of $p_*(E)$. Then $F\to
  p_*(p^*(F))$ is an isomorphism.
\end{Lemma}

\begin{proof}
Consider the composite of $F\to p_*(p^*F) \to p_*(p^*(p_*(E))) \to
p_*(E)$. It is equal to the given inclusion $F\hookrightarrow
p_*(E)$. Hence $F\to p_*(p^*F)$ is injective.
On the other hand, $F\to p_*(p^*(F))$ is surjective by
\lemref{lem:Bridgeland}(1). So $F\to p_*(p^*(F))$ is an isomorphism.
\end{proof}

\subsection{Blow-up case}

Suppose $p\colon \bX\to X$ is a blow-up of a projective surface $X$ at
a smooth point $0\in X$.

We first determine the sheaves $K$ appearing the condition~(3) in
\defref{def:perverse}.

\begin{NB}
  In \cite{Kota} it was proved that $K$ is semi-stable in the sense of
  Simpson, under the assumption of (2). Since the semi-stability is
  defined later (in a different definition), so I do not organize how
  the paper should be written, so I omit it.
\end{NB}

\begin{Lemma}\label{lem:T}
  Let $K$ be a sheaf on $\bX$.

\textup{(1)} If $p_*(K) = 0$, then there is a filtration
\begin{equation*}
   K = F^0 \supset F^1 \supset \cdots \supset F^{s-1} \supset F^s = 0
\end{equation*}
such that $F^k/F^{k+1}\cong \shfO_C(-1-a_k)$ for $a_k\ge 0$. 
In particular, if $\Hom(K,\shfO_C(-1)) = 0$, then $K=0$.

\begin{NB}
Is it possible to take $a_k$ to be decreasing ?  
\end{NB}

\textup{(2)} If $\bR p_*(K) = 0$, then $K = \shfO_C(-1)^{\oplus s}$.
\end{Lemma}

\begin{Remark}
The filtration in (1) can be considered as a kind of
Harder-Narashimhan filtration. This will be clear in a different proof
given in the next subsection.
\end{Remark}

\begin{proof}
(1)  We may assume $K\neq 0$. Since $p_*(K) = 0$, $K$ is of pure
dimension $1$, and hence $c_1(K) = s[C]$ with $s > 0$. Then 
\[
  \chi(K,\shfO_C(-1)) 
  = \int_{\bX} \ch(K)^\vee  \ch(\shfO_C(-1))\td \bX
  = - (c_1(K), [C]) = s > 0.
\]
Let $\C_0$ be the skyscraper at $0$. Since 
\[
  \Ext^2(K,\shfO_C(-1)) \cong \Hom(\shfO_C,K)^\vee 
  \cong \Hom(p^*(\C_0),K)^\vee \cong \Hom(\C_{0}, p_*(K))^\vee = 0
\]
from the assumption, we must have $\Hom(K,\shfO_C(-1))\neq 0$. Take a
non-zero homomorphism $\phi\colon K\to \shfO_C(-1)$. Then we have
$p_*(\Ker\phi) = 0$ and $\Ima\phi\cong\shfO_C(-1-a)$ with $a\ge 0$.
Applying this procedure to $\Ker\phi$, we get the assertion.

(2) We first note that $\chi(K) = 0$.
\begin{NB}
In fact, $\bR p_*(K) = 0$ if and only if $p_*(K) = 0$, $\chi(K) = 0$.
\end{NB}%
Let $E$ be a subsheaf of
$K$. Then $H^0(\bX,E) = 0$, which implies that $\chi(E) \le 0$.%
\begin{NB}
This means $K$ is `semi-stable' with respect to $\chi$.
\end{NB}
Applying this to $E := \Ker\phi$ in the proof of (1), we get
$\chi(\Ima\phi) \ge 0$. Then $a$ in (1) must be $0$, i.e.\ $\Ima\phi \cong
\shfO_C(-1)$.%
\begin{NB}
I add one step to \cite{Kota}.
\end{NB}
We also have $\chi(\Ker\phi) = 0$. Repeating this
argument, we conclude all $F^k/F^{k+1}\cong \shfO_C(-1)$. Since
$\Ext^1(\shfO_C(-1), \shfO_C(-1)) = 0$, we get the assertion.
\end{proof}

\begin{NB}
Let $C$ be a rational curve with $N_{C/X}={\cal O}_C(-1)^{\oplus n}$.
Let $K$ be a purely 1-dimensional sheaf supported on $C$.
If $K$ is stable in the sense of Simpson, then
$K$ is isomorphic to ${\shfO}_C(a)$ for some $a\in\Z$.

It is sufficient to prove that $K$ is an ${\cal O}_C$-module.
If $K$ is not an ${\cal O}_C$-module, then
we have a non-zero homomorphism
$K \otimes I_C^k/I_C^{k+1} \to K$ for some $k>0$.

We suppose $(K|C)/(\text{torsion})=\bigoplus_i {\cal O}_C(a_i)$ with
$a_1 \leq a_2 \leq \cdots \leq a_s$.

Then $\mu(K) \le\mu({\cal O}_C(a_1))$ as we have a nonzero
homomorphism $K\to \shfO_C(a_1)$.
On the other hand, we also have a nonzero homomorphism
$\shfO_C(a_1+k)\to K$ from the above remark. Therefore
$\mu({\cal O}_C(a_1+k)) \le \mu(K)$.
This is impossible. Hence $K$ is an ${\cal O}_C$-module.

If $X$ is a surface, then we also have another argument:
Since $-(c_1(K),C)=\chi({\cal O}_C(n),K)$,
$\Hom({\cal O}_C(n),K) \ne 0$ or $\Hom(K,{\cal O}_C(n-1)) \ne 0$.
We set $n_0:=\max\{n|\Hom({\cal O}_C(n),K) \ne 0\}$.
Then $\Hom({\cal O}_C(n_0),K) \ne 0$ and
$\Hom(K,{\cal O}_C(n_0)) \ne 0$,
which means that $K \cong {\cal O}_C(n_0)$.
\end{NB}

\begin{Proposition}\label{prop:pervblowup}
  \textup{(1)} A coherent sheaf $E$ on $\bX$ belongs to $\Per(\bX/X)$
  if and only if $\Hom(E,\shfO_C(-1)) = 0$.

  \textup{(2)} Let $E\in \Per(\bX/X) \cap \Coh(\bX)$ and $\phi\colon
  p^*(p_*(E))\to E$ be the natural homomorphism. Then $\Ker\phi \cong
  \Ext^1(E,\shfO_C(-1))^\vee\otimes\shfO_C(-1)$, and the exact
  sequence $0 \to \Ker\phi \to p^*(p_*(E)) \to E \to 0$ obtained from
  \lemref{lem:Bridgeland}(3) is the universal extension of $E$ with
  respect to ${\cal O}_{C}(-1)$.
\end{Proposition}

\begin{proof}
(1) From \defref{def:perverse}(3) $E\in\Per(\bX/X)$ satisfies
$\Hom(E,\shfO_C(-1)) = 0$. For the converse, suppose
$E\in\Per(\bX/X)\cap \Coh(\bX)$ satisfies $\Hom(E,\shfO_C(-1)) =
0$. By \lemref{lem:Bridgeland}(3) it is enough to show that
$\phi\colon p^*p_*(E) \to E$ is surjective. By
\lemref{lem:Bridgeland}(2)(iii) we have $p_*(\Coker\phi) = 0$.
Since $\Hom(\Coker\phi,\shfO_C(-1))\subset \Hom(E,\shfO_C(-1)) = 0$
from the assumption, we have $\Coker\phi = 0$ from \lemref{lem:T}(1).

(2) Consider $\phi\colon p^* p_*(E)\to E$. By
\lemref{lem:Bridgeland}(2),(3) this is surjective and $\Ker\phi$ satisfies
$\bR p_*(\Ker\phi) = 0$. By \lemref{lem:T}(2), $\Ker\phi =
\shfO_C(-1)^{\oplus s}$ for some $s\in\Z_{\ge 0}$. By
\lemref{lem:Bridgeland}(4) we have $\Ext^1(E,\shfO_C(-1))\cong
\Hom(\Ker\phi,\shfO_C(-1))$.
\end{proof}

\begin{Lemma}\label{lem:R^1}
  If $E\in\Per(\bX/X)\cap\Coh(\bX)$, then $R^1 p_*(E(C)) = 0$.
\end{Lemma}

\begin{proof}
  By the exact sequence
\(
  0 \to \shfO_{\bX} \to \shfO_{\bX}(C) \to \shfO_C(-1) \to 0,
\)
we have $\bR p_*(\shfO_{\bX}(C)) = \shfO_X$.
From the projection formula we have
\begin{equation*}
  \bR p_*(\shfO_{\bX}(C)\otimes \bL p^*(p_*(E))) 
  \cong \bR p_*(\shfO_{\bX}(C)) \otimes p_*(E) 
  \cong p_*(E).
\end{equation*}
The spectral sequence as in the proof of \lemref{lem:Bridgeland}(2)
implies $R^1 p_*(\shfO_{\bX}(C)\otimes p^*(p_*(E))) = 0$.
As $p^* p_*(E)\to E$ is surjective by the assumption, we have the
conclusion.
\end{proof}

\begin{Lemma}\label{lem:torsionfree}
Let $E\in\Coh(\bX)$. Then $p_*(E)$ is torsion free at $0$
\begin{NB}
What is the right terminology ?  
\end{NB}%
if and only if $\Hom(\shfO_C,E) = 0$.
\end{Lemma}

\begin{proof}
  We have
  \begin{equation*}
    \Hom(\shfO_C,E) \cong \Hom(p^* \C_0, E) \cong \Hom(\C_0, p_*(E)).
  \end{equation*}
Now the assertion is clear.
\end{proof}

\subsection{Perverse coherent sheaves and representations of a quiver}
\label{subsec:quiver}

This subsection is a detour. We look at the definition of the perverse
coherent sheaves in view of \cite{perv}. The result of this subsection
will not be used later.

Let $X = \mathbb C^2$ and let $\bX$ be the blowup of $X$ at the origin
$0$. As a by-product of the main result of \cite{perv}, we have an
equivalence between the derived category $\bD^b_c(\Coh \bX)$ of
complexes of coherent sheaves whose homologies have proper supports
and the derived category of finite dimensional modules of the quiver
\(
\xymatrix{
  \bullet \ar@<-1ex>[r]_{d} & \ar@<-1ex>[l]_{B_1,B_2} \ar[l]
  \bullet
}
\)
with relation $B_1 d B_2 = B_2 d B_1$.

\begin{Proposition}
The abelian category $\{ E \in \Per(\bX/X) \mid \text{$E$ has
a proper support \/}\}$ is equivalent to the abelian category of finite
dimensional representations of the above quiver 
with relation $B_1 d B_2 = B_2 d B_1$.
\end{Proposition}

\begin{proof}
Let us first recall how we constructed the equivalence between derived
categories in \cite{perv}. Let us number the left (resp.\ right)
vertex as $0$ (resp.\ $1$). We consider line bundles
$L_0 := \shfO$ and $L_1 := \shfO(C)$, and homomorphisms between them
$s = z_1/z = z_2/w \colon L_0\to L_1$, $z$, $w\colon L_1\to L_0$,
where $\bX = \{ (z_1,z_2,[z:w])  \in\C^2\times\proj^1\mid z_1 w = z_2
z \}$. For an object $E\in\bD^b_c(\Coh\bX)$, we define
\begin{equation*}
   V_k = \bR(\mathrm{pt})_{*}(E\otimes L_k), \quad (k=0,1)
\end{equation*}
where $\mathrm{pt}$ is a projection of $\bX$ to a point. Then the
homomorphisms $s$, $z$, $w$ give a structure of a quiver
representation.
Conversely given a complex of quiver representations, we define a
double complex of coherent sheaves on $\bX$ as
\begin{equation*}
\begin{CD}
  \mathcal A := 
\begin{matrix}
   V_0 \otimes L_1
\\
   \oplus
\\ V_1\otimes L_0
\end{matrix}
@>{\alpha}>>
\mathcal B := 
\begin{matrix}
   \C^2\otimes V_0\otimes L_0 \\ \oplus \\ \C^2\otimes V_1\otimes L_0
\end{matrix}
@>{\beta}>>
  \mathcal C := 
\begin{matrix}
  V_0\otimes L_0
\\
  \oplus
\\
  V_1\otimes L_1
\end{matrix},
\end{CD}
\end{equation*}
with $\alpha$, $\beta$ as in \cite[(1.2)]{perv}. 
We assign the
degree by $\deg \mathcal A = -2$, $\deg \mathcal B = -1$,
$\deg\mathcal C = 0$.%
\begin{NB}
Recall that the data $C_m = 
\xymatrix{
  \C^{m} & \ar@<-.5ex>[l]_{B_1,B_2} \ar@<.5ex>[l]
  \C^{m+1}
}
$ corresponding to $\shfO_C(-m-1)$ is given by $\Ker\beta/\Ima\alpha$.
But $\shfO_C(-m-1)[1]\in\Per(\bX/X)$.
\end{NB}
Then the associated
total complex is an object in $\bD^b_c(\Coh(\bX))$.
(For the reader familiar with \cite{perv}: We consider the $W = 0$
case. So objects have proper supports, and hence implicitly have the framing.) 

Let us start the proof of this proposition.
Suppose $E$ is a perverse coherent sheaf on $\bX$ with a proper
support. The corresponding representation satisfies
\begin{equation*}
   H^i(V_0) = H^{i}(\bX,E), \qquad H^i(V_1) = H^{i}(\bX,E(C)).
\end{equation*}
We have a spectral sequence 
\(
   H^i(\bX, H^j(E)) \Longrightarrow H^{i+j}(\bX,E).
\)
Therefore $H^i(V_0) = 0$ for $i\neq 0$ follow from the following
vanishing results, which are direct consequence of the definition of
perverse coherent sheaves:
\begin{gather*}
  H^0(\bX, H^{-1}(E)) = H^0(\C^2, R^0p_*(H^{-1}(E))) = 0,
\\
  H^1(\bX,H^{0}(E)) = H^0(\C^2, R^1 p_*(H^{0}(E))) = 0.
\end{gather*}

Next consider $V_1$. We consider the exact sequence of vector bundles
\[
   0\to \shfO_{\bX}(C) \xrightarrow{\left[
       \begin{smallmatrix}
          z \\ w
       \end{smallmatrix}\right]}
   \shfO_{\bX}^{\oplus 2}
   \xrightarrow{\left[
       \begin{smallmatrix}
          -w & z
       \end{smallmatrix}\right]}
     \shfO_{\bX}(-C) \to 0.
\]
This exact sequence is preserved under $\bullet\otimes H^{-1}(E)$ as
$\shfO_{\bX}(-C)$ is locally-free. Therefore we get an exact sequence
$0\to p_*(H^{-1}(E(C))) \to p_*(H^{-1}(E))^{\oplus 2}$. But the right
hand side vanish by the assumption. This implies $H^{-1}(V_1) = 0$ as
above.
We have $H^0(E)\in\Per(\bX/X)$, and hence
we have $R^1 p_*(H^0(E(C))) = 0$ from \lemref{lem:R^1}. This gives
$H^1(V_1) = 0$. Vanishing of other cohomology groups is trivial.

For the converse we check that the object $E
= \{ \mathcal A\to\mathcal B\to\mathcal C \} \in\bD^b_c(\Coh \bX)$ 
corresponding to a quiver representation
satisfies the conditions in \defref{def:perverse}.
The condition (1) follows from the injectivity of $\alpha$ as argued 
in \cite[\S5.2]{perv}.
The condition (2) follows from $H^i(V_0) = 0$ for $i\neq 0$ by the
above discussion.
Let us consider the condition (3). Note that $H^0(E) =
\Coker\beta$. We also know that $\shfO_C(-1)[1]$ corresponds
to the representation
$C_0 = 
\xymatrix{
  0 \ar@<-1ex>[r] & \ar@<-1ex>[l] \ar[l] 
  \C
}$
by \cite[Prop.~5.3]{perv}. Let us write the corresponding complex as
\(
  \{ \mathcal A' \to \mathcal B' \to \mathcal C' \}.
\)
Then we apply the proof of \cite[Prop.~5.12]{perv} to show that
(i) $\Hom(H^0(E),\shfO_C(-1))$ is isomorphic to the space of homomorphisms
between monads, where the complex
\(
  \{ \mathcal A' \to \mathcal B' \to \mathcal C' \}
\)
is shifted by $-1$, and
(ii) this space of homomorphisms between monads vanishes.
\end{proof}

\begin{NB}
  For the framed perverse coherent sheaves, we used $V_0 =
  H^1(\bX,E)$, $V_1 = H^1(\bX,E(C))$. It seems that the mismatch of
  the degree comes from the switching from ideal sheaves and structure
  sheaves. For an ideal sheaf $I_Z\subset \shfO_{\bX}$, we should use
  $H^1$, while we should use $H^0$ for $\shfO_Z = \shfO_{\bX}/I_Z$.
\end{NB}

Let us give another proof of \lemref{lem:T} based on the above result.

(1) If $p_*(K) = 0$, $E := K[1]$ is a perverse coherent sheaf. Let us
consider the corresponding complex as above.  Since $K$ is a sheaf, we
have $H^0(E) = 0$, which means $\beta$ is surjective.
By \cite[Lem.~5.1]{perv} this is equivalent to 
$\codim T_1 > \codim T_0$ or $(T_0,T_1) = (V_0,V_1)$ for
any subrepresentation $T = (T_0,T_1)$ of $E$.
Taking the Harder-Narashimhan filtration with respect to the slope
\(
   \theta(S_0,S_1) = (- \dim S_0 + \dim S_1) / (\dim S_0 + \dim S_1),
\)
we get a filtration
\(
   K = Y^0 \supset Y^1 \supset \cdots \supset Y^{N-1} \supset Y^N = 0
\)
such that $Y^k/Y^{k+1}$ is $\theta$-semistable and
$\theta(Y^0/Y^1) < \theta(Y^1/Y^2) < \cdots <
\theta(Y^{N-1}/Y^N)$.
Taking $(T_0,T_1) = Y^1$, $\codim T_1 > \codim T_0$ means
$\theta(Y^0/Y^1) > 0$. Therefore all $\theta(Y^k/Y^{k+1}) >
0$. Looking at the classification of all stable representations in
\cite[Rem.~2.17]{perv}, we find that $Y^k/Y^{k+1}$ is (a direct sum
of) $C_m = \shfO_C(-m-1)$ for $m\ge 0$.

\begin{NB}
  $\theta(\shfO_C(-m-1)) = 1/(2m+1)$. Therefore $\theta(Y^k/Y^{k+1})$
  is increasing means that the corresponding $m = m_k$ is decreasing.
\end{NB}

(2) If we further have $\bR p_*(K) = 0$, we have $V_0 = 0$ for the
representation corresponding to $E = K[1]$. Then assertion is obvious.

\begin{NB}
This proof for (2) seems to work also in the $3$-dimensional standard
flop case.

Now a geometric proof from Kota's e-mail on Apr. 8:

Let $I_C$ be the ideal sheaf of the exceptional curve $C$.
Then $I_C/I_C^2 \cong {\cal O}_C(-1) \otimes {\Bbb C}^2$.
Hence $I^n_C/I^{n+1}_C \cong {\cal O}_C(n) \otimes S^n{\Bbb C}^2$.

Let $F$ be the torsion free quotient of
$E\otimes {\cal O}_X/I_C$. Since $R^1 p_* F=0$, we have
$F \cong \oplus_i {\cal O}_C(a_i)$, $a_i \geq -1$.

For a sufficiently large $n$,
the surjective homomorphism
$\phi:E \to E \otimes {\cal O}_X/I_C^n$ has a 0-dimensional kernel.
\begin{NB2}
Let $H$ be a $p$-ample divisor on $X$.
Then for the kernel $G$ of
$\phi$, $\chi(E \otimes {\cal O}_H)=
\chi(E\otimes {\cal O}_X/I_C^n \otimes {\cal O}_H)
+\chi(G \otimes {\cal O}_H)$,
$\chi(G \otimes {\cal O}_H) \geq 0$ and
$\chi(G \otimes {\cal O}_H)=0$ iff $G$ is of 0-dimension.
Hence for a sufficiently large $n$, we have the claim.
\end{NB2}
Since $p_* E=0$, $E$ is purely 1-dimensional.
Thus $E \cong E \otimes {\cal O}_X/I_C^n$ for $n \gg 0$.
\begin{NB2}
This may be trivial.
\end{NB2}
Assume that $\phi_m:E \to E \otimes {\cal O}_X/I_C^m$ is isomorphic and
$\phi_{m-1}:E \to E \otimes {\cal O}_X/I_C^{m-1}$
is not isomorphic. Then $\ker \phi_{m-1}$ is a purely-1-dimensional
${\cal O}_C$-module generated by
$E \otimes I^{m-1}_C/I^{m}_C$.
Hence if $m-1>0$, then we see that $p_*(\ker \phi_{m-1}) \ne 0$,
which is a contradiction.
Therefore $m=1$. Thus $E$ is a purely 1-dimensional
${\cal O}_C$-module with ${\bf R}p_* E=0$, which implies that
$E \cong {\cal O}_C(-1)^{\oplus l}$.
\end{NB}

\section{Moduli spaces of semistable perverse coherent sheaves}
\label{sec:moduli}

\begin{NB}

For a polynomial $a(x) \in {\Bbb Q}[x]$,
$a(x) \geq 0$ if $a(m) \geq 0$ for $m \gg 0$.  
Then $f(m,n)=\sum_{i=0}^d a_i(m)n^i \geq 0$ for $n \gg m \gg 0$
if and only if $(a_d(x),a_{d-1}(x),\dots,a_0(x)) \geq 
(0,0,\dots,0)$ with respect to the lexicographic order.

Then for a polynomial $f(x,y)=\sum_{i=0}^d a_i(x)y^i$
such that $\deg a_d(x)=d$,
there is an integer $m_0$ such that 
for $m \geq m_0$,
$f(m+ln,n) > 0$ for $n \gg l \gg m \geq m_0$
if and only if $a_d(m) > 0$. 

\end{NB}

We return to the general situation considered in
\subsecref{subsec:general}, i.e.\
$p\colon Y\to X$ is a birational morphism of projective varieties
such that ${\bR} p_*(\shfO_Y)=\shfO_X$ and
$\dim p^{-1}(x) \leq 1$ for all $x \in X$.


\subsection{Stability}
Let ${\cal O}_X(1)$ be an ample line bundle on $X$.  We also denote
the pull-back $p^*({\cal O}_X(1))$ by ${\cal O}_Y(1)$.
\begin{NB}
Let $G$ be a vector bundle on $X$ and
we set $\widetilde{G}:=\pi^*(G)$.
\end{NB}

Let $M$ be a line bundle on $Y$.
For $E^{\bullet}\in\bD^b(Y)$, 
we define $a_i(E^{\bullet},M) \in {\Bbb Z}$ by
the coefficient of the Hilbert polynomial of 
$E^{\bullet}$ with respect to $M$: 
\begin{equation*}
\chi(E^{\bullet} \otimes M^{\otimes m})=
\sum_{i=0}^{\dim Y} a_i(E^{\bullet},M) \binom{m+i}{i}.
\end{equation*}
We use the same notation for $E\in\bD^b(X)$ and a line bundle $M$ on
$X$ instead of $Y$.%
\begin{NB}
I suppose $a_i$ will be used only for $X$ and $Y$.  
\end{NB}

\begin{NB}
We set $a_i^{\widetilde{G}}(E^{\bullet},M):=
a_i(\widetilde{G}^{\vee} \otimes E^{\bullet},M)$.
Thus we have
\begin{equation*}
\chi(\widetilde{G},E^{\bullet} \otimes M^{\otimes m})=
\chi(\widetilde{G}^{\vee} \otimes E^{\bullet} \otimes M^{\otimes m})=
\sum_{i=0}^d a_i^{\widetilde{G}}(E^{\bullet},M) \binom{m+i}{i}.
\end{equation*}
\end{NB}
If $E^{\bullet}$ is a coherent sheaf of dimension $d$ and $M$ is
ample, then $a_i(E^{\bullet},M)=0$, $i>d$ and $a_d(E^{\bullet},M)>0$.
For a perverse coherent sheaf $E^{\bullet} \in \Per(X/Y)$, we denote
$a_i(E^{\bullet},{\cal O}_X(1))=
a_i(\pi_*(E^{\bullet}),{\cal O}_Y(1))$ by $a_i(E^{\bullet})$ if there
is no fear of confusion.

Recall that we say a coherent sheaf $E$ of dimension $d$ on $X$ is
{\it \textup(semi\textup)stable\/} if 
\begin{equation*}
   a_d(E) \chi(F(m)) (\le) a_d(F) \chi(E(m)) \quad \text{for $m\gg 0$}
\end{equation*}
for any proper subsheaf $0\neq F\subsetneq E$.
Here we adapt the convention for the short-hand notation in \cite{HL}.
The above means two assertions: semistable if we have `$\le$', and
stable if we have `$<$'. If $E$ is semistable, then it is pure of
dimension $d$: if $0\neq F\subsetneq E$ is of dimension $< d$, then the
right hand side is $0$, while the left hand side is positive. Under
the assumption that $E$ is pure, the above inequality is equivalent to
$\chi(F(m))/a_d(F) (\le) \chi(E(m))/a_d(E)$, as we automatically has
$a_d(F) > 0$.

We say $E$ is {\it $\mu$-\textup(semi\textup)stable} if it is purely
$d$-dimensional and
\begin{equation*}
   \frac{a_{d-1}(F)}{a_d(F)} (\le) \frac{a_{d-1}(E)}{a_d(E)}
\end{equation*}
for any subsheaf $0\neq F\subset E$ with $a_d(F) < a_d(E)$.%
\begin{NB}
The inequality, after replaced by $a_{d-1}(F) a_d(E) (\le) a_{d-1}(E)
a_d(F)$, does not imply $E$ to be pure of dimension $d$. It just implies
that the torsion subsheaf of $E$ has codimension at least two.

It seems that $\mu$-stability was defined only for $\dim E = \dim X$,
and the purity was not assumed in \cite{HL}. But as we will only use
torsion free sheaves on downstairs, we put the above definition.
\end{NB}

If $E$ is a $d$-dimensional coherent sheaf on $X$, then we have the
following implications:
\begin{equation*}
   \text{$E$ is $\mu$-stable} \Longrightarrow
   \text{$E$ is stable} \Longrightarrow
   \text{$E$ is semistable} \Longrightarrow
   \text{$E$ is $\mu$-semistable}.
\end{equation*}

Now we return to the situation $p\colon Y\to X$. Let $L_0$ an ample
line bundle on $Y$. We set $L_l=L_0(l)$.

We consider the following conditions on an object
$E^\bullet\in\Per(Y/X)$:
\begin{subequations}\label{eq:support}
  \begin{align}
    & \dim p_*(E^{\bullet}) >\dim Z,   
\\
    & \chi(E^{\bullet}(m) \otimes L_l^{\otimes n}) >0
    \quad\text{for $n \gg 0$}.
  \end{align}
\end{subequations}

\begin{Definition}\label{def:stable}
Let $E^{\bullet} \in \Per(Y/X)$ be an object satisfying
\eqref{eq:support}.
Then $E^{\bullet}$ is {\it \textup(semi\textup)stable\/} if 
for any proper subobject $F^{\bullet} \in \Per(Y/X)$ of $E^{\bullet}$,
we have
\begin{equation}\label{eq:G-pervstable}
\chi(F^{\bullet}(m)) (\leq)
\frac{\chi(F^{\bullet}(m)\otimes L_l^{\otimes n})}
{\chi(E^{\bullet}(m) \otimes L_l^{\otimes n})} 
{\chi(E^{\bullet}(m))}
\end{equation}
for all $n \gg l \gg m \gg 0$.
\end{Definition}

\begin{NB}
  May 7 : Added
\end{NB}
\begin{Remark}
  The above definition is suitable for perverse coherent sheaves
  satisfying \eqref{eq:support}. On the other hand, it is also natural
  to expect that $\shfO_C(-m-1)[1]$ (in case of $p\colon \bX\to X$) is
  stable in some definition in view of \cite{perv}.
\end{Remark}

Note that for $E^{\bullet} \in \Per(Y/X)$ satisfying
\eqref{eq:support} with $d := \dim p_*(E^\bullet)$, $E$ is
(semi)stable if and only if
$H^{-1}(E^{\bullet})=0$ and
for any subsheaf $F$ of $E:=H^0(E^{\bullet})$
in $\Coh(Y)$ with $F \in \Per(Y/X)$,
\begin{gather}
\chi(F(m))
<
\frac{a_d(F,{\cal O}_Y(1))} 
{a_d(E,{\cal O}_Y(1))}
{\chi(E(m))}
\tag{1}
\intertext{or}
 \left\{
   \begin{aligned}[m]
\chi(F(m))
&
=
\frac{a_d(F,{\cal O}_Y(1))}
{a_d(E,{\cal O}_Y(1))} {\chi(E(m))},
\\
a_d(F,{\cal O}_Y(1))
& 
<
\frac{a_d(F,L_l)}
{a_d(E,L_l)}{a_d(E,{\cal O}_Y(1))}
\end{aligned}\right.
\tag{2}
\\
\intertext{or}
\left\{
\begin{aligned}[m]
\chi(F(m))
& =
\frac{a_d(F,{\cal O}_Y(1))}
{a_d(E,{\cal O}_Y(1))} {\chi(E(m))},
\\
a_d(F,{\cal O}_Y(1))
&=
\frac{a_d(F,L_l)}
{a_d(E,L_l)}{a_d(E,{\cal O}_Y(1))},\\
{\chi(F(m)\otimes L_l^{\otimes n})}
& (\geq)
\frac{a_d(F,{\cal O}_Y(1))}
{a_d(E,{\cal O}_Y(1))} 
\chi(E(m)\otimes L_l^{\otimes n})
\quad\text{for $n \gg l \gg m \gg 0$}.
\end{aligned}\right.
\tag{3}
\end{gather}
\begin{NB}
It seems that (\ref{eq:support}a) is used here:

For example, suppose $p\colon \bX\to X$ and $E = \shfO_C$. Then
$\chi(E(m)\otimes L_l^{\otimes n}) = (c_1(\shfO_Y(m+nl)) + n
c_1(L),[C]) = n(c_1(L),[C])$. Therefore the assumption is certainly
necessary.

If $D\subset X$ is another curve, and $E = \shfO_C + \shfO_D$.
Then 
$\chi(E(m)\otimes L_l^{\otimes n}) = (c_1(\shfO_Y(m+nl)) + n
c_1(L),[C]+[D]) \sim nl (c_1(L),[D])$.
\end{NB}

\begin{Remark}
If $d=\dim Y$, then $a_d$ is essentially the rank, so we have
\begin{equation*}
{a_d(F,{\cal O}_Y(1))}
=
\frac{a_d(F,L_l)}
{a_d(E,L_l)} a_d(E,{\cal O}_Y(1)).
\end{equation*}
Therefore the second case (2) does not occur.  
\begin{NB}
If ${a_d(F,L_l)} = 0$,
then $F = 0$, so $a_d(F,\shfO_Y(1)) = 0$.
(In that case, the numerator of the left hand side is also $0$.)
\end{NB}
\end{Remark}

\begin{Lemma}\label{lem:basics}
Let $E^{\bullet} \in \Per(Y/X)$ be an object satisfying \eqref{eq:support}.
Then

  \textup{(1)} If $E^{\bullet}$ is semistable, then
  $E^{\bullet} \in \Coh(Y)$.

\textup{(2)} If $E^{\bullet}$ is semistable, then
  $p_*(E^{\bullet})$ is semistable.

\textup{(3)} Suppose further $E^\bullet = E \in\Coh(Y)$. Then
$E$ is \textup(semi\textup)stable if and only if
\eqref{eq:G-pervstable} holds for any proper subsheaf $F^\bullet = F$
of $E$ in $\Coh(Y)$ which is also in $\Per(Y/X)$.

\textup{(4)} Suppose further $E^\bullet = E\in\Coh(Y)$ and
$p_*(E)$ is stable. Then $E$ is stable.
\end{Lemma}

\begin{NB}
\begin{Lemma}\label{lem:sheafisenough}
Let $E^{\bullet} \in \Per(X/Y)$ be an object satisfying \eqref{eq:support}.
Then

\textup{(1)}
If $E^{\bullet}$ is semistable, then
$E^{\bullet} \in \Coh(X)$.

\textup{(2)}
For $E \in \Per(Y/X) \cap \Coh(Y)$ and a subobject $F^{\bullet} \to E$
in $\Per(Y/X)$, we set $F':=\Ima(H^0(F^\bullet) \to H^0(E))$.
Then 
\begin{equation*}
\chi(\widetilde{G},F^{\bullet}(m))
\leq 
\frac{\chi(\widetilde{G},F^{\bullet}(m)\otimes L_l^{\otimes n})}
{\chi(\widetilde{G},F'(m)\otimes L_l^{\otimes n})}
\chi(\widetilde{G},F'(m)).
\end{equation*}
\end{Lemma}
\end{NB}

\begin{proof}
(1) First note that we have
$\chi(E^\bullet(m)\otimes L_l^{\otimes n}) > 0$ for $n \gg 0$ by
(\ref{eq:support}b). On the 
other hand, as $0\neq \bR p_*(E^\bullet)\in\Coh(X)$ from
(\ref{eq:support}a) and the perversity, we have $\chi(E^\bullet(m)) > 0$ for $m\gg 0$.
\begin{NB}
  Apr. 26 : I add the above argument to clarify we used the assumptions.
\end{NB}

Next note that $E^{\bullet}$ contains a subobject $H^{-1}(E^{\bullet})[1]$.
Assume that $H^{-1}(E^{\bullet})[1] \ne 0$.
Then 
\(
  \chi(H^{-1}(E^{\bullet})[1](m)) 
  = \chi(R^1 p_*(H^{-1}(E^{\bullet}))(m)) \geq 0
\)
and
\(  
  \chi(H^{-1}(E^{\bullet})[1](m)\otimes L_l^{\otimes n})
  = - \chi(H^{-1}(E^{\bullet})(m)\otimes L_l^{\otimes n})
 < 0,
\)
which means that $E^{\bullet}$ is not semistable.

(2) Suppose $F_1$ is a subsheaf of $p_*(E)$. 
Note that we can apply \lemref{lem:subsheaf} to $F_1\subset p_*(E)$
thanks to (1). Therefore $p_*(p^*(F_1)) = F_1$.
Let $\alpha$ be the composition of $p^*(F_1) \to p^* p_*(E)\to E$. We
have $p^*(F_1)$, $E\in \Per(Y/X)$ by \lemref{lem:Bridgeland}(1) and
the assumption respectively. We have $H^{-1}(p^*(F_1)) = H^{-1}(E) =
0$. We also have that $F_1 = p_*(p^*(F_1)) \to p_*(E)$ is injective by
the assumption. Therefore the homomorphism $\alpha$ is
injective in the category $\Per(Y/X)$ by \lemref{lem:perv}(2).
Therefore we have the inequality \eqref{eq:G-pervstable} for $F^\bullet :=
p^*(F_1)$. It means that
\begin{equation*}
  \chi(F_1(m)) \le
    \frac{\chi(p^*(F_1)(m)\otimes L_l^{\otimes n})}
    {\chi(E^\bullet(m)\otimes L_l^{\otimes n})}
   \chi(p_*(E)(m)).
\end{equation*}
From the note after \defref{def:stable}, we must have
\begin{equation*}
  \chi(F_1(m)) \le
    \frac{a_d(p^*(F_1), \shfO_Y(1))}
    {a_d(E^\bullet, \shfO_Y(1))}
   \chi(p_*(E)(m)).
\end{equation*}
Since 
$a_d(E^\bullet,\shfO_Y(1)) = a_d(p_*(E^\bullet),\shfO_X(1))$,
$a_d(p^*(F_1), \shfO_Y(1)) = a_d(F_1, \shfO_X(1))$,
this inequality says $p_*(E^\bullet)$ is semistable.

(3) The `only if' part is clear: if $F\subset E$ is a subsheaf, then
$E/F$ is also in $\Per(Y/X)$ obviously from the definition. Therefore
$0\to F\to E\to E/F\to 0$ is also exact in $\Per(Y/X)$, hence $F$ is a
subobject of $E$ in $\Per(Y/X)$.

Let us show the `if' part for the semistability. Suppose $F^\bullet$
is a subobject of $E$, and let $E/F^\bullet$ be the quotient in
$\Per(Y/X)$. Then we have an exact sequence in $\Coh(Y)$:
\begin{equation*}
\xymatrix@R=.8pc{
  0 \ar[r] & H^{-1}(F^\bullet)\ar[r] & H^{-1}(E) = 0 \ar[r] 
  & H^{-1}(E/F^\bullet) \ar[lld]
\\
           & H^0(F^\bullet) \ar[r] & E \ar[r] 
           & H^0(E/F^\bullet) \ar[r] & 0,
}
\end{equation*}
Therefore we have $H^{-1}(F^\bullet) = 0$ and an exact sequence in $\Coh(Y)$:
\begin{equation*}
   0 \to H^{-1}(E/F^\bullet) \to H^0(F^\bullet) \to F'\to 0,
\end{equation*}
where $F' = \Ima(H^0(F^\bullet) \to E)$. Note that $F'\in\Per(Y/X)$.
Take the direct image with respect to $p$ to get
\begin{equation*}
   0 \to p_*(H^0(F^\bullet)) \to p_*(F') 
     \to R^1 p_*(H^{-1}(E/F^\bullet)) \to 0.
\end{equation*}
Therefore we have
\begin{equation}\label{eq:ineq}
\begin{split}
   \chi(F^\bullet(m)\otimes L_l^{\otimes n})
   &  
   \begin{aligned}[t]
   & =   \chi(F'(m)\otimes L_l^{\otimes n})
   + \chi(H^{-1}(E/F^\bullet)(m)\otimes L_l^{\otimes n}))
\\
   & \ge \chi(F'(m)\otimes L_l^{\otimes n}),
\end{aligned}
\\
   \chi(p_*(F')(m)) 
   & 
   \begin{aligned}[t]
   & = \chi(p_*(H^0(F^\bullet))(m)) + \chi(R^1 p_*(H^{-1}(E/F^\bullet))(m))
\\
   & \ge \chi(F^\bullet(m)).
\end{aligned}
\end{split}
\end{equation}
So the inequality \eqref{eq:G-pervstable} for $F'$ implies one for
$F^\bullet$.

Finally show the `if' part for the stability. Assume that the strict
inequality in \eqref{eq:G-pervstable} holds for a proper subsheaf
$F$. Suppose that the equality holds for a subobject $F^\bullet\subset
E$ in $\Per(Y/X)$. From the above discussion, $F' = E$ follows from
the assumption. Therefore $H^0(E/F^\bullet) = 0$.
Moreover the inequalities in \eqref{eq:ineq} must be the equalities,
so we must have $H^{-1}(E/F^\bullet) = 0$. Therefore $E = F^\bullet$.

(4) To test the stability of $E$, it is enough to check the inequality
for a subsheaf $F\subset E$ such that $F\in\Per(Y/X)$ by (3).
We have $p_*(F)\subset p_*(E)$. We may assume $p_*(F)\neq 0$ by
\lemref{lem:perv}(1). If $p_*(F)\neq p_*(E)$, then the stability of
$p_*(E)$ implies the strict inequality for
\eqref{eq:G-pervstable}. 
Here we have used (\ref{eq:support}a) so that the leading coefficient
of $\chi(E(m)\otimes L_l^{\otimes n})$ is $a_d(E,\shfO_Y) =
a_d(p_*(E),\shfO_X)$.

Therefore we may assume $p_*(F) = p_*(E)$. Let $C$ be the cokernel of
$F\to E$ in $\Coh(Y)$. Then $C$ is perverse coherent and $\bR p_*(C) =
0$. Therefore $C = 0$ by \lemref{lem:perv}(1). Hence $F = E$.
\end{proof}

If $E^\bullet = E\in \Per(Y/X)\cap \Coh(Y)$, we have the following
implications:
\begin{equation*}
   \text{$p_*(E)$ is stable} \Longrightarrow
   \text{$E$ is stable} \Longrightarrow
   \text{$E$ is semistable} \Longrightarrow
   \text{$p_*(E)$ is semistable}.
\end{equation*}


\begin{Lemma}\label{lem:surface}
Suppose that $Y$ is a nonsingular surface, $\dim
p_*(E^{\bullet})=2$, and $E^\bullet = E\in \Coh(Y)\cap\Per(Y/X)$
satisfies the condition \eqref{eq:support}. Then
$E$ is semistable if and only if the followings hold\textup:
\begin{aenume}
\item $p_*(E^{\bullet})$ is semistable,
\item
\begin{equation*}
{(c_1(F),c_1(L_0))}
\geq
\frac{{\rk F}}{\rk E} (c_1(E),c_1(L_0))
\end{equation*}
for any subsheaf $F\subset E$ such that $F\in \Per(X/Y)$ and
\(
{\chi(F(m))} =
\allowbreak
{{\rk F}\chi(E^{\bullet}(m))}/
{\rk E^{\bullet}} .
\)
\end{aenume}

Moreover,
$E$ is stable if and only if the followings hold\textup:
\begin{aenume}
\setcounter{enumi}{2}
\item $p_*(E)$ is semistable,
\item $p_*(E)$ is stable, or the strict inequality in above
  \textup{b)} holds.
\end{aenume}
\end{Lemma}

This is a consequence of the note after \defref{def:stable} and
\lemref{lem:basics}.

\begin{NB}
\subsubsection{Filtration}

\begin{Lemma}
Let $F_1$ be a subsheaf of $\pi_*(E)$ such that
$\dim F_1 \leq s$ and
$\pi_*(E)/F_1$ does not contain an $s$-dimensional subsheaf.
We define a subsheaf $E_1$ of $E$ as the image of
$\pi^*(F_1) \to \pi^* (\pi_*(E)) \to E$.   
Then $\pi_*(E_1)=F_1$, $H^{-1}(E/E_1)=0$ and
$\pi_*(E/E_1)=\pi_*(E)/F_1$.
\end{Lemma}

\begin{proof}
Since $E_1$ is a subsheaf of $E$,
$H^{-1}(E/E_1)=0$ and
$\pi_*(E_1) \subset \pi_*(E)$.
We have a homomorphism $\xi:F_1 \to \pi_*(\pi^*(F_1))
\to \pi_*(E_1)$.
We note that $F_1 \overset{\xi}{\to} \pi_*(E_1) \to 
\pi_*(E)$ is the natural inclusion.
Since $\Supp(\pi_*(E_1)/F_1) \subset \Supp(F_1)$
and $\pi_*(E)/F_1$ does not contain an $s$-dimensional subsheaf,
we get the claim.
\end{proof}

\begin{Lemma}
Let $E$ be a perverse coherent sheaf such that 
$H^i(E)=0$, $i \ne 0$ and $\pi_*(E)$ is 
torsion free.
Then we have a filtration 
\begin{equation}
0 \subset F_1 \subset F_2 \subset \cdots \subset F_s=E
\end{equation}
in $\Per(Y/W)$
such that
\begin{enumerate}
\item
$F_i/F_{i-1}$ is a semi-stable perverse coherent sheaf
with respect to $\widehat{H}$ and
\item
\begin{equation}
0 \subset \pi_*(F_1) \subset \pi_*(F_2) \subset \cdots \subset 
\pi_*(F_s)=\pi_*(E)
\end{equation}
is the Harder-Narasimhan filtration of $\pi_*(E)$
\end{enumerate}
\end{Lemma}

\begin{proof}
Let $0 \subset G_1 \subset G_2 \subset \cdots \subset G_s=\pi_*(E)$
be the Harder-Narasimhan filtration of $\pi_*(E)$.
Then there is a subsheaf $F_1 \subset E$ such that
$\pi_*(F_1)=G_1$ and $H^{-1}(E/F_1)=0$.
Since $F_1$ is a semi-stable perverse coherent sheaf and
$\pi_*(E/F_1)=\pi_*(E)/G_1$, we can inductively
construct a required filtration of
$E$.  
\end{proof}

\end{NB}

\subsection{Construction of moduli spaces}

Thanks to the discussion in the previous subsection, we can work
entirely in the category of coherent sheaves.

For a perverse coherent sheaf $E$, let $h(x,y)$ be the polynomial such
that $\chi(E(m) \otimes L_0^{\otimes n})=h(m,n)$.  Then $\chi(E(m)
\otimes L_l^{\otimes n})= \chi(E(m+ln) \otimes L_0^{\otimes
  n})=h(m+ln,n)$.  We call $h$ the Hilbert polynomial of the perverse
coherent sheaf $E$.

The following is the main result in this subsection.

\begin{Theorem}\label{thm:twisted}
Let ${\cal X} \to S$ be a flat family of projective schemes 
and 
$p\colon{\cal Y} \to {\cal X}$ a family of birational maps over $S$
such that $\dim p^{-1}(x) \leq 1$ for all $x \in {\cal X}$
and
${\bR}p_*({\cal O}_{\cal Y})={\cal O}_{\cal X}$.
Let ${\cal O}_{{\cal X}}(1)$ be a relatively ample line bundle
on ${\cal X}/S$ and
${\cal L}_0$ a relatively ample line bundle on ${\cal Y}/S$. 
Then there is a coarse moduli scheme 
$\overline{M}_{{\cal Y}/{\cal X}/S}^p({h})$
parametrizing $S$-equivalence classes of
semistable perverse coherent sheaves $E$ on ${\cal X}_s$, $s \in S$
with the Hilbert polynomial $h$.
Moreover, $\overline{M}_{{\cal Y}/{\cal X}/S}^p({h})$
is a projective scheme over $S$.
There is an open subscheme $\Mp_{\mathcal Y/\mathcal X/S}(h)\subset
\overline{M}_{{\cal Y}/{\cal X}/S}^p({h})$ parametrizing isomorphism
classes of stable perverse coherent sheaves.
\begin{NB}
Apr. 29: I add the `stable' moduli space.  
\end{NB}
\end{Theorem}

For simplicity, we treat the absolute case $p\colon Y \to X$.%
\begin{NB}
We set $h(x):=h(x,0)$. 
\end{NB}

Our construction of the moduli space of semistable perverse coherent
sheaves is a modification of that of usual moduli spaces by Simpson
\cite{S:1} (see also \cite{HL,Mar:4}). This idea was already appeared
in \cite{Br:4}.
However we modify the arguments in many places, so we need to recall
almost all steps of the usual proof.

\begin{Definition}\label{def:typelambda}
Let $\lambda$ be a nonnegative
\begin{NB}
Apr. 27:  I add the above.
\end{NB}%
rational number.

(1)
Let $E$ be a coherent sheaf of dimension $d$ on $X$. Then
$E$ is {\it of type $\lambda$}
(with respect to the semi-stability),
if the following two conditions hold:
\begin{aenume}
\item
$E$ is of pure dimension $d$,
\item For all subsheaf $F$ of $E$ we have
\begin{equation*}
{a_{d-1}(F)} \leq 
\frac{{a_d(F)}}{a_d(E)}a_{d-1}(E)
+\lambda.
\end{equation*}
\end{aenume}
Note that this is equivalent to the $\mu$-stability if $\lambda=0$.

(2)
For a perverse coherent sheaf $E$ on $Y$ with $E\in\Coh(Y)$,%
\begin{NB}
Apr. 26:  I add the condition `$E\in\Coh(Y)$' 
\end{NB}
$E$ is {\it of type $\lambda$}, if $p_*(E)$ is of type $\lambda$.
\end{Definition}

Since the set of type $\lambda$ coherent sheaves on $X$ is bounded
(see e.g., \cite[3.3.7]{HL})
and $p^*(p_*(E)) \to E$ is surjective for a perverse 
coherent sheaf of type $\lambda$, we get the following.

\begin{Lemma}
The set of type $\lambda$ perverse coherent sheaves on $X$ with a fixed 
Hilbert polynomial is bounded.
\end{Lemma}

From Langer's important result \cite[Cor.~3.4]{Langer:1} (see also
\cite[3.3.1]{HL}%
\begin{NB}
Apr. 27: I add the reference.\ 
\end{NB}%
), we have the following estimate for the dimension of
sections for $F$ on $X$ of type $\lambda$:
\begin{equation}\label{eq:section}
 \frac{h^0(F)}{a_d(F)} \leq 
\frac{1}{d!} \left[
 \frac{a_{d-1}(F)}{a_d(F)}+\lambda
 +c \right]_+^d,
\end{equation}
where 
$c$ depends only on $(X,{\cal O}_X(1))$, $d$, $a_d(F)$
and $[x]_+:=\max\{x,0\}$.

\begin{Definition}
Let ${\cal U} \equiv \mathcal U(\lambda,h)$ be the set of pairs $(E'
\subset E)$ such that $E$ is a perverse coherent sheaf of type
$\lambda$ with the Hilbert polynomial $h$, $E' \in \Per(X/Y)$, and 
$E'':=E/E'$ satisfies
\begin{equation}\label{eq:a3}
\chi(E(m))\frac{a_d(E'')}{a_d(E)} \geq 
{\chi(E''(m))} \quad\text{for $m \gg 0$}.
\end{equation}
\end{Definition}
The inequality means $p_*(E')$ {\it destabilizes\/} $p_*(E)$ in a weak
sense (i.e.\ `$=$' is allowed).
\begin{NB}
Apr. 26: I add this explanation.  
\end{NB}

Since the set of $E$ is bounded,
by Grothendieck's boundedness theorem, the set $\mathcal U$ of such pairs
$(E'\subset E)$ is also bounded.
Hence there is an integer $m({\lambda})$ which depends on $h$
and $\lambda$ such that if $m \geq m({\lambda})$ and $(E'\subset
E)\in\mathcal U$,
\begin{subequations}\label{eq:flat}
\begin{align}
& \text{$H^0(E'(m)) \otimes \shfO_Y \to E'(m)$ is surjective and}
\\
& H^i(E'(m))=0 \quad\text{for $i>0$},
\end{align}
\end{subequations}
\begin{NB}
\begin{enumerate}
 \item[($\flat 1$)]
  $H^0(E'(m)) \otimes \shfO_Y \to E'(m)$ is surjective and
 \item[($\flat 2$)]
  $H^i(E'(m))=0$, $i>0$.
\end{enumerate}
\end{NB}%
and for $F\in \Coh(X)$ of type $\lambda$ 
\begin{NB}
I also impose the following condition.
\end{NB}
\begin{subequations}\label{eq:flat'}
\begin{align}
& \text{$H^0(F(m)) \otimes \shfO_X \to F(m)$ is surjective and}
\\
& H^i(F(m))=0 \quad\text{for $i>0$}.
\end{align}
\end{subequations}
In particular, we apply \eqref{eq:flat} to $(E = E)$ to have that the
above two conditions hold for $E$.

Furthermore, since the set of Hilbert polynomials of $E'$ is finite,
we may assume that $m(\lambda)$ satisfies also the followings: for all $m
\geq m(\lambda)$, we can choose sufficiently large integers $l(m)$ and
$n(m) \gg l(m)$ such that
\begin{subequations}\label{eq:sharp}
\begin{gather}
\begin{aligned}[m]
& \frac{\chi(E(m))}
{\chi(E(m) \otimes L_{l(m)}^{\otimes n(m)})} \geq 
\frac{\chi(E'(m))}
{\chi(E'(m)\otimes L_{l(m)}^{\otimes n(m)})} 
\\
\Longleftrightarrow\; &
\frac{\chi(E(m))}
{\chi(E(m) \otimes L_l^{\otimes n})} \geq 
\frac{\chi(E'(m))}
{\chi(E'(m)\otimes L_l^{\otimes n})} 
\quad\text{ for all $n \gg l \gg m$},
\end{aligned}
\\
\chi(E'(m) \otimes L_{l(m)}^{\otimes n(m)})=
h^0(E'(m) \otimes L_{l(m)}^{\otimes n(m)})
\end{gather}
\end{subequations}
hold for $(E' \subset E) \in {\cal U}$.
\begin{NB}
From now on, we write $L_{l(m)}$ as $L$.
\end{NB}

For $m\ge m(0)$ let $V_m$ be a vector space of dimension $h(m,0)$.
Let ${\frak Q}:=\Quot_{V_m \otimes \shfO_Y/Y}^{h[m]}$ 
be the quot-scheme 
parametrizing all quotients
$V_m\otimes\shfO_Y \twoheadrightarrow F$ (in $\Coh(X)$) with
the Hilbert polynomial $h[m]$, 
where $h[m](x,y)=h(m+x,y)$.
Let $V_m \otimes \shfO_{{\frak Q}\times Y} \twoheadrightarrow 
\widetilde{E}(m)$
be the universal quotient sheaf on ${\frak Q} \times Y$.
Let ${\frak Q}^{ss}$ be the open subscheme of ${\frak Q}$ consisting of
quotients $f\colon V_m \otimes \shfO_Y\to E(m)$
such that 
\begin{enumerate}
\item
the canonical map
$V_m \to H^0(E(m))$ is an isomorphism and 
\item
$E$ is 
a semi-stable sheaf.
\end{enumerate}
Note that $E(m)$ is automatically in $\Per(Y/X)$. Other conditions
clearly define the {\it open\/} subscheme.

By the above discussion all $E$ appearing as $(E'\subset E)\in\mathcal
U$ together with a choice of basis of $H^0(E(m))$ gives a closed point
in $\mathcal Q$ if $m\ge m(\lambda)$. In particular, we can construct
the moduli scheme as a quotient of $\mathcal Q^{ss}$.
\begin{NB}
Apr. 26 : I add the above explanation.  
\end{NB}

In order to take the quotient via the GIT, we use a Grassmann
embedding of $\mathcal Q$ as follows.
Let $n \gg l \gg m$. Set $W:=H^0(L_l^{\otimes n})$.
Let ${\frak G}(l,n):=\operatorname{Gr}(V_m \otimes W,h(m+ln,n))$
be the Grassmannian parametrizing $h(m+ln,n)$-dimensional quotient
spaces of $V_m \otimes W$.
For a quotient $(V_m \otimes \shfO_Y \to E(m)) \in {\frak Q}$
its kernel $F$ satisfies
\begin{enumerate}
\item
$H^0(F\otimes L_l^{\otimes n}) \otimes \shfO_Y
\to F\otimes L_l^{\otimes n}$ is surjective and
\item
$H^i(F\otimes L_l^{\otimes n})=0$, $i>0$
\end{enumerate}
for sufficiently large $n$.
Hence we get a quotient vector space   
$V_m \otimes W \to H^0(E(m) \otimes L_l^{\otimes n})$.
\begin{NB}
From the long exact sequence and the vanishing (2).
\end{NB}%
Thus we get a morphism ${\frak Q} \to {\frak G}(l,n)$,
which is a closed immersion.
This embedding depends on the choice of $n \gg l \gg m$.
%
%
%
We have a natural action of $SL(V_m)$ on ${\frak G}(l,n)$. 
Let ${\cal L}:={\cal O}_{{\frak G}(l,n)}(1)$ be the 
tautological line bundle on ${\frak G}(l,n)$.
Then ${\cal L}$ has an $SL(V_m)$-linearization.
We consider the GIT semi-stability with respect to 
${\cal L}$. The following is well-known (cf.\ \cite[4.4.5]{HL})
\begin{Proposition}\label{prop:ss}
Let $\alpha\colon V_m \otimes W \twoheadrightarrow A$
be a quotient corresponding to a point of ${\frak G}(l,n)$.
Then it is GIT \textup(semi\textup)stable with respect to ${\cal L}$
if and only if 
\begin{equation*}
  \frac{\dim \left[\alpha(V' \otimes W)\right]}{\dim V'}
   (\geq) 
  \frac{\dim \left[\alpha(V_m \otimes W)\right]}{\dim V_m}
\end{equation*}
for all non-zero proper subspaces $V'$ of $V_m$.
\end{Proposition}

\begin{NB}
I move the following lemma to here (before \propref{prop:sub}).

And I also change the right hand side from $1/2$ to $1/3$, as later we
replace this inequality by an inequality before limit.
\end{NB}

We prepare several estimates in order to compare the semistability of $E$
and that of the corresponding point in the Grassmann variety.

\begin{Lemma}\label{lem:L}
  Let $E$ be a $d$-dimensional sheaf with the Hilbert polynomial $h$
  and $E'$ a subsheaf of $E$. Then if we take a sufficiently large
  $l \gg m$ depending on $h$, $m$, we have
\begin{equation*}
\left|\frac{a_d(E',L_l)}{a_d(E,L_l)}
-\frac{a_d(E')}{a_d(E)}\right|
\leq  
\frac{1}{3 h(m,0) a_d(E)!}.
\end{equation*}
\end{Lemma}

In particular, we may suppose that $l(m)$ in the above \eqref{eq:sharp}
satisfies this condition.

\begin{proof}
We have
$a_d(E) \geq a_d(E') \geq 0$ and
$a_d(E/E',L_l) \geq l^d a_d(E/E')$.
Since $a_d(E/E',L_l)=
a_d(E,L_l)-a_d(E',L_l)$
and $a_d(E/E')=a_d(E)-a_d(E')$, we have
$a_d(E,L_l)-l^d a_d(E) \geq 
a_d(E',L_l)- l^d a_d(E') \geq 0$.
Hence we have 
\begin{equation*}
\begin{split}
\left|\frac{a_d(E',L_l)}{a_d(E,L_l)}
-\frac{a_d(E')}{a_d(E)}
\right|
& \leq
\left|\frac{a_d(E',L_l)-l^d a_d(E')}
{a_d(E,L_l)} \right|
+\left|\frac{ l^d a_d(E')}{a_d(E,L_l)}-
\frac{a_d(E')}{a_d(E)}\right|\\
& \leq
2\left|1-\frac{ l^d a_d(E)}{a_d(E,L_l)}\right|.
\end{split}
\end{equation*}%
\begin{NB}
  The second term is
  \begin{equation*}
     \frac{a_d(E')}{a_d(E)} \left| \frac{l^d a_d(E)}{a_d(E,L_l)} - 1 \right|.
  \end{equation*}
\end{NB}
If we take a sufficiently large $l$ depending on $h$, this can be made
smaller than an arbitrary given number.
\end{proof}

\begin{NB}
I move the definition of $\mathcal F$ to here, as it will be also used
later.
\end{NB}

We consider a set of pairs
\begin{equation*}
  {\cal F}:=\{ (V_m\otimes\shfO_Y\twoheadrightarrow E(m))\in {\frak Q}\}
  \times \{ V' \subset V_m \}.
\end{equation*}
Let $\alpha\colon V_m\otimes W\twoheadrightarrow H^0(E(m)\otimes
L_l^{\otimes n})$ be the corresponding point in $\mathfrak G(l,n)$.
\begin{NB}
Original definition:
\begin{equation*}
    {\cal F}:=\{E' \subset \widetilde{E}_q(m)|
E'=\im(V' \otimes \widetilde{G} \to \widetilde{E}_q(m)), q \in {\frak Q}, 
V' \subset V_m \}.
\end{equation*}
\end{NB}%
We set $E'(m):=\Ima(V' \otimes \shfO_Y\to {E}(m))$.
Since ${\cal F}$ is a bounded set, $E'(m)$ satisfies
\begin{subequations}\label{eq:vanishonY}
  \begin{align}
    & \alpha(V' \otimes W)=H^0(E'(m)\otimes L_l^{\otimes n})
\\
    & H^i(E'(m)\otimes L_l^{\otimes n})=0 \quad\text{for $i>0$}
  \end{align}
\end{subequations}
for a sufficiently large $n$ which depends on $m$ and $l$.
\begin{NB}
  We have natural homomorphisms 
\[
  V'\otimes W \to H^0(E'(m)\otimes L_l^{\otimes n})
  \subset H^0(\widetilde{E}_q(m)\otimes L_l^{\otimes n}),
\]
whose composition is ghe restriction of $\alpha$ to $V'\otimes W$. If
$n$ is sufficiently large, then the first arrow is surjective, as
$H^1(\Ker[V'\otimes \shfO_Y\to E']\otimes L_l^{\otimes n}) =
0$. Therefore we have $H^0(E'(m)\otimes L_l^{\otimes n})
= \alpha(V'\otimes W)$.
\end{NB}%
Then we have
\(
  \dim [\alpha(V' \otimes W)] = \chi(E'(m)\otimes L_l^{\otimes n}).
\)

\begin{Lemma}\label{lem:compare}
\textup{(1)} $\dim V' \le h^0(E'(m))$.

\textup{(2)} Take $l\gg m$ as in \lemref{lem:T}.
If we take $n$ sufficiently large depending on $h$, $m$, $l$, we have
\begin{equation*}
 \left| \frac{\dim [\alpha(V'\otimes W)]}{\dim[\alpha(V_m\otimes W)]}
 -\frac{a_d(E')}{a_d({E})} \right|
 <\frac{1}{2 \dim V_m\, a_d({E}) !}.
\end{equation*}
\end{Lemma}

\begin{proof}
  (1) We have a natural homomorphism $V' \to H^0(E'(m))$. If we
  compose an injective homomorphism $H^0(E'(m))\to H^0(E(m)) = V_m$,
  it becomes equal to the given inclusion $V'\subset V_m$, so it is
  injective. The assertion follows.

(2) We have
\begin{equation*}
 \left| \frac{\chi(E'(m)\otimes L_l^{\otimes n})}
{\chi({E}(m)\otimes L_l^{\otimes n})}
 -\frac{a_d(E')}{a_d({E})} \right|
 <\frac{1}{2 \dim V_m\, a_d({E}) !}
\end{equation*}
for this sufficiently large $n$ by \lemref{lem:L}. Thus the assertion
follows from the above conditions (1),(2).
\end{proof}

We replace $n(m) \gg l(m) \gg m$ in \eqref{eq:sharp} if necessary so
that they also satisfy the assertion in this lemma.

\begin{Proposition}\label{prop:sub}
There is an integer $m_1 (\geq m(0))$ such that for all $m \geq m_1$,
${\frak Q}^{ss}$ is contained in ${\frak G}(l,n)^{ss}$,
where $l=l(m), n=n(m)$.
\end{Proposition}

\begin{NB}
Apr. 27: I have slightly change the organization of the proof. I
collect the discussion under the equality in \eqref{eq:pgs} to one
part. Also the first part of the proof, which we do not assume the
semistability, is moved to above.
\end{NB}

\begin{proof}
We first take $m\ge m(0)$.

Suppose $E\in\mathfrak Q^{ss}$, i.e.\ $E$ is semistable, and
take $V'\subset V_m$.
From \lemref{lem:compare}(1),(2)
\begin{NB}
  and $h^0(E(m)) = \dim V_m$
\end{NB}%
we have
\begin{equation}\label{eq:ss_check}
\begin{split}
    &\dim V_m \dim [\alpha(V' \otimes W)]-
  \dim V' \dim [\alpha(V_m \otimes W)] 
\\
  \ge \; & h^0({E}(m))\,\dim [\alpha(V' \otimes W)]
   - h^0(E'(m))\, \dim [\alpha(V_m \otimes W)] 
\\
  \ge\; & \left(h^0(E(m)) \frac{a_d(E')}{a_d(E)}- h^0(E'(m))
    -\frac{1}{2 a_d(E)!}
  \right) \dim [\alpha(V_m \otimes W)].
\end{split}
\end{equation}

Since $p_*(E)$ is semistable,
in the same way as in \cite[4.4.1]{HL},
\begin{NB}
Originally it was refered to \cite[sect.~4]{Mar:4}.
\end{NB}%
we see that
there is an integer $m_3$ which depends on $h$
such that for $m \geq m_3$ and
a subsheaf $E'$ of
$E$,
\begin{equation}\label{eq:pgs}
 \frac{h^0(E'(m))}{a_d(E')} \leq 
 \frac{h^0(E(m))}{a_d(E)}
\end{equation}
and the equality holds, if and only if 
\begin{equation}\label{eq:equal}
 \frac{\chi(E'(m+m'))}{a_d(E')}
=\frac{\chi(E(m+m'))}{a_d(E)}
\end{equation}
for all $m'$. More precisely we apply the argument in \cite[4.4.1]{HL}
to $p_*(E')\subset p_*(E)$.

We take $m_1:=\max \{m_3,m(0) \}$ so that both
(\ref{eq:ss_check},\ref{eq:pgs}) hold for $m\ge m_1$.

If the inequality in \eqref{eq:pgs} is strict, the last expression of
\eqref{eq:ss_check} is positive. Therefore $\alpha$ is stable.

So we may assume the equality in \eqref{eq:pgs} holds. Then $p_*(E')$
is also semistable by \eqref{eq:equal}, and we may assume (\ref{eq:flat}a,b)
holds for $E'$. So we have $h^0(E'(m)) =\chi(E'(m))$.
Therefore the middle expression of \eqref{eq:ss_check} is equal to
\begin{equation*}
  \chi(E(m)) \chi(E'(m)\otimes L_l^{\otimes n})
  - \chi(E'(m)) \chi(E(m)\otimes L_l^{\otimes n}).
\end{equation*}
This is nonnegative by the semistability of $E$.
Therefore our claim holds. 
\end{proof}

\begin{Proposition}\label{prop:proper}
There is an integer $m_2$ such that for all $m \geq m_2$,
${\frak Q}^{ss}$ is a closed subscheme of 
${\frak G}(l,n)^{ss}$,
where $n = n(m)$, $l = l(m)$.
\end{Proposition}

We choose an $m(\ge m_1)$ so that $h(m)/a_d(h)>1$.
We shall prove that 
${\frak Q}^{ss} \to {\frak G}(l,n)^{ss}$ is proper.
Let $(R,{\frak m})$ be a discrete valuation ring and its maximal
ideal, and $K$ the quotient field of $R$.
We set $T:=\Spec(R)$ and $U:=\Spec(K)$.
Let $U \to {\frak Q}^{ss}$ be a morphism such that $U \to {\frak Q}^{ss} \to 
{\frak G}(l,n)^{ss}$ is extended to
a morphism $T \to {\frak G}(l,n)^{ss}$.
Since ${\frak Q}$ is a closed subscheme of ${\frak G}(l,n)$,
there is a morphism $T \to {\frak Q}$, {\it i.e.,}
there is a flat family of quotients:
\begin{equation*}
 V_m \otimes \shfO_Y \otimes {\cal O}_{T} \to 
 {\cal E}(m) \to 0.
\end{equation*}
Let $\alpha\colon V_m \otimes W \otimes R \to %
{p_{T*}}(
{\cal E}(m)\otimes L^{\otimes n})$ 
be the quotient of
$V_m \otimes W \otimes R$ corresponding to
the morphism $T \to {\frak G}(l,n)^{ss}$.
We set $E:={\cal E} \otimes R/{\frak m}$.

\begin{Claim}\label{claim:1}
$V_m \to H^0(E(m))$ is injective.
\end{Claim}

\begin{proof}
We set $V':=\Ker(V_m \to H^0(E(m)))$.
Then $\alpha(V' \otimes W)=0$.
\begin{NB}
  $\alpha$ is given by the multiplications of sections
  in $V_m$ and $W$.
\end{NB}%
Hence we get

\begin{equation*}
 \begin{split}
  0 \leq & \dim V_m \dim [\alpha(V' \otimes W)]-
  \dim V' \dim [\alpha(V_m \otimes W)]\\
  =& -\dim V' \dim [\alpha(V_m \otimes W)] \leq 0.
   \end{split}
\end{equation*}
Therefore $V'=0$.
\end{proof}

\begin{Claim}\label{claim:2}
There is a rational number $\lambda$ which depends on $h$ 
such that $E$ is of type $\lambda$.
\end{Claim}

\begin{proof}
By \cite[Lem.~1.17]{S:1} (see also \cite[4.4.2]{HL})
there is a purely $d$-dimensional sheaf $F$ with the
Hilbert polynomial $h(x,0)$ and a homomorphism $p_*(E) \to F$ 
whose kernel is a coherent sheaf of dimension less than $d$.
Note that the assumption in \cite[Lem.~1.17]{S:1} that $p_*(E)$ can be
deformed to a pure sheaf is satisfied by our definition of $E$.
We shall first check that $F$ is of type $\lambda$.
We need to check the inequality in \defref{def:typelambda}(1b) for the
maximal destabilizing subsheaf of $F$. Let $F \to F''$ be the
corresponding quotient, which is semistable.
\begin{NB}
Apr. 27:  I expand the explanation.
\end{NB}%
We set $E':=\Ker(p_*(E) \to F'')$ and $E'':=\Ima(p_*(E) \to F'')$.
Since $F''$ is semistable, \eqref{eq:section} gives
\begin{equation}\label{eq:h0est}
\frac{1}{d!} \left[
 m + \frac{a_{d-1}(F'')}{a_d(F'')}+ c \right]_+^d
 \ge
 \frac{h^0(F''(m)))}{a_d(F'')}
 \ge \frac{h^0(E''(m))}{a_d(E'')},
\end{equation}
where we have used $h^0(E''(m)) \leq h^0(F''(m))$ and
$a_d(E'')=a_d(F'')$ in the second inequality.
\begin{NB}
  $p_*(E)$ and $F$ are isomorphic in dimension $d$.
\end{NB}

\begin{NB}
Let $p_*(E) \to E''$ be a purely $d$-dimensional 
quotient of $p_*(E)$.
Let $E'$ be the kernel of $p_*(E) \to E''$.
\end{NB}
We note that $V_m \to H^0(p_*(E)(m))$ is injective by
Claim~\ref{claim:1}.
\begin{NB}
Apr. 27: I add `by Claim~1'.  
\end{NB}%
We set $V':=V_m \cap H^0(E'(m))$.
Then 
\begin{equation}
  \label{eq:lower}
  h^0(E''(m)) \geq \dim V_m -\dim V'.
\end{equation}
\begin{NB}
We have
\(
  0 \to H^0(E'(m)) \to H^0(E(m)) \to H^0(E''(m)).
\)
Therefore
\(
  H^0(E''(m)) \supset H^0(E(m))/H^0(E'(m))
  \supset V_m/V'.
\)
\end{NB}%
Let $E_1(m)$ be the image of $V'\otimes\shfO_Y \to E(m)$.
\begin{NB}
Original:
a subsheaf of $E(m)$ generated by $V'$.
\end{NB}%
Then $E_1$ comes from $(E, V')\in {\cal F}$. (We have denoted the
corresponding sheaf by $E'$ above, but we change the notation as it is
already used for a different sheaf.)
\begin{NB}
Apr.27:  I add a trivial explanation.
\end{NB}%
Since $E''$ is purely $d$-dimensional and
$(p_*(E_1)+E')/E'$ 
\begin{NB}
  This $+$ is taken in $E$.
\end{NB}%
is supported on $Z$, 
\begin{NB}
  Outside $Z$, $p_*(E_1)$ is clearly contained in $E'$.
\end{NB}%
we have $p_*(E_1) \subset E'$.
\begin{NB}
  $(p_*(E_1)+E')/E'\subset E/E' = E''$ is a lower dimensional sheaf.
\end{NB}%
Therefore 
\begin{equation}\label{eq:est}
   a_d(E_1) \le a_d(E') = a_d(E) - a_d(E'').
\end{equation}
\begin{NB}
Apr. 27: I add an inequality.
\end{NB}

We write 
\(
\varepsilon:=
1/{2 h(m,0)\, a_d(E)!}
\)
%
%
the constant appearing in \lemref{lem:compare}(2) for brevity. Then
\begin{NB}
Apr. 27: I comment out the following since it has been already
discussed above.

Since ${\cal F}$ is a bounded set, for 
$l=l(m)$ and $n=n(m)$
we may assume that 
$\alpha(V' \otimes W)=\Hom(G,E_1 \otimes L_l^{\otimes n})$,
$\Ext^i(G,E_1 \otimes L_l^{\otimes n})=0$, $i>0$ and 
\begin{equation}\label{eq:a1}
 \left|\frac{ \dim \alpha(V' \otimes W)}{\dim \alpha(V_m \otimes W)}
  -\frac{a_d(E_1)}{a_d(h)} \right|<
 \varepsilon.
\end{equation}
Since $\pi_*(E_1) \subset E'$,
\end{NB}%
{\allowdisplaybreaks
\begin{alignat*}{2}
   \frac{h^0(E''(m))}{a_d(E'')} & \geq
   \frac{\dim V_m-\dim V'}{a_d(E'')}
   &\quad & (\text{by \eqref{eq:lower}})
\\
  & \ge \frac{\dim V_m}{a_d(E'')}\left(1
    - \frac{\dim [\alpha(V'\otimes W)]}{\dim [\alpha(V_m\otimes W)]}
    \right) & \quad & (\text{by the semistability of $\alpha$})
\\
  & \ge \frac{\dim V_m}{a_d(E'')}\left(
    1 - \frac{a_d(E_1)}{a_d(E)} - \ve \right) &
  \quad & (\text{by \lemref{lem:compare}(2)})
\\
  & \ge \frac{\dim V_m}{a_d(E'')}\left(
    \frac{a_d(E'')}{a_d(E)} - \ve \right) &
  \quad & (\text{by \eqref{eq:est}})
\\
  & \geq \dim V_m\left(\frac{1}{a_d(E)}-\varepsilon \right) &
  \quad & (\text{as $a_d(E'') \ge 1$}).
\end{alignat*}
There} is a rational number $\lambda_1$ 
and an integer $m_4 \geq \lambda_1-a_{d-1}(E)/a_d(E)$ which depend on
$h(x,0)$ such that
\begin{equation*}
 \dim V_m \left(\frac{1}{a_d(E)}-\varepsilon \right) 
 = \frac{\dim V_m}{a_d(E)}-\frac{1}{2a_d(E)!} \geq  
 \frac{1}{d!}\left(m+\frac{a_{d-1}(E)}{a_d(E)}-\lambda_1 \right)^d
\end{equation*}
 for $m \geq m_4$.
\begin{NB}
The middle expression is
\begin{equation*}
  \binom{m+d}{d} + \sum_{i=0}^{d-1} \frac{a_i(E)}{a_d(E)}
  \binom{m+i}{i} - \frac1{2a_d(E)!}
  = 
  \frac1{d!} m^d + \text{lower order}.
\end{equation*}
\end{NB}%
Combining this with the above inequality and \eqref{eq:h0est}, we get
\begin{equation*}
\frac{1}{d!}\left(m+\frac{a_{d-1}(F'')}{a_d(F'')}+c \right)
\geq 0
\end{equation*}
and
\begin{equation}
 \frac{a_{d-1}(E)}{a_d(E)}-\lambda_1 \leq 
 \frac{a_{d-1}(F'')}{a_d(F'')}+c
\end{equation}
for $m\ge m_4$.
Hence $F$ is of type $\lambda := (\lambda_1 + c)a_d(E)$.
\begin{NB}
\begin{equation*}
\begin{split}
 & \frac{a_{d-1}(E)}{a_d(E)} \leq 
 \frac{a_{d-1}(F'')}{a_d(F'')} + \lambda_1 + c
\\
\Longleftrightarrow\; &
   \frac{a_{d-1}(F')}{a_d(F')} \leq \frac{a_{d-1}(E)}{a_d(E)} 
   + (\lambda_1 + c) \frac{a_d(F'')}{a_d(F')}
\end{split}
\end{equation*}
\end{NB}%

We set $m_2 := \max\{ m_4, m(\lambda)\}$ and take $m\ge m_2$.
We consider $V_m = H^0(p_*(E)(m)) \to H^0(F(m))$ and let
$V'$ be the kernel.
Then $J:=\Ima(V' \otimes \shfO_Y \to E(m))$, restricted to
$Y\setminus p^{-1}(Z)$,
\begin{NB}
  I add the condition.
\end{NB}%
is of dimension less than $d$.
Hence we get $a_d(J)=0$.
By \lemref{lem:compare}(2) (applied to $E' := J$) and
\propref{prop:ss}, we have $V' = 0$.
Thus $H^0(\pi_*(E)(m)) \to H^0(F(m))$ is injective. But both have
dimension equal to $h(m,0)$, and hence they are isomorphic.
Since $H^0(F(m)) \otimes \shfO_X \to F(m)$ is surjective,
$p_*(E) \to F$ must be surjective. As they have the same Hilbert
polynomials, they are isomorphic.
Therefore $p_*(E)$ is of pure dimension $d$, of type $\lambda$ and 
$V_m \to H^0(E(m))$ is an isomorphism.
Thus we complete the proof of Claim~\ref{claim:2}. 
\end{proof}

\begin{proof}[Proof of \propref{prop:proper}]
Finally we need to show that $E$ is semistable. Then it gives the
lifting $T \to {\frak Q}^{ss}$ and 
finish the proof that ${\frak Q}^{ss} \to {\frak G}(l,n)^{ss}$ is proper.

Assume that there is an exact sequence
\begin{equation*}
0 \to E_1 \to E \to E_2 \to 0
\end{equation*}
such that $E_1 \in \Per(X/Y)$ and $E_1$ destabilizes the
semistability. Then \eqref{eq:a3} is satisfied, so
$(E_1\subset E)\in\mathcal U$, so $E_1$ satisfies \eqref{eq:flat}. 
Since $\alpha\in\mathfrak G(l,n)$ corresponding to $E$ is semistable,
we have the inequality in \propref{prop:ss} for
\(
  V' := H^0(E_1(m))\subset H^0(E) = V_m.
\)
But by \eqref{eq:vanishonY} the inequality is equivalent to
\begin{equation*}
 \frac{\chi(E_1(m))}
{\chi(E_1(m)\otimes L_{l(m)}^{\otimes n(m)})} \leq 
 \frac{\chi(E(m))}
{\chi(E(m)\otimes L_{l(m)}^{\otimes n(m)})},
\end{equation*}
%
%
which means that $E_1$ is not a destabilizing subsheaf.
Therefore $E$ is semistable.
\end{proof}

\begin{NB}
Let $(R,{\frak m})$ be a discrete valuation ring $R$ and the
maximal ideal ${\frak m}$.
Let $K$ be the fractional field and $k$ the residue field.
Let ${\cal E}$ be a $R$-flat family of perverse coherent sheaves on $X$
such that $\pi_*({\cal E}) \otimes_R K$
is pure.
Then $\pi_*({\cal E})$ is a
$R$-flat family of coherent sheaves on $Y$.
\begin{Lemma}\label{lem:valuative}\cite[Lem. 1.17]{S:1}
There is an $R$-flat family
of coherent sheaves ${\cal F}$ on $Y$ and a homomorphism
$\psi:\pi_*({\cal E}) \to {\cal F}$ on $Y$ such that
${\cal F} \otimes_R k$ is pure,
$\psi_K$ is an isomorphism and
$\psi_k$ is an isomorphic at generic points of $\Supp({\cal F} \otimes_R k)$.
\end{Lemma}
\end{NB}

By standard arguments, we see that 
$SL(V_m)s, s \in {\frak Q}^{ss}$ is a closed orbit
if and only if 
the corresponding semistable perverse coherent sheaf $E$
is isomorphic to
$\bigoplus_i E_i$,
where $E_i$ are stable perverse coherent sheaves. This completes the proof of
\thmref{thm:twisted}.

\section{Wall-crossing}\label{sec:wall-crossing}

\begin{NB}
Kota assumed $-K_X$ is ample. For the smoothness of the moduli spaces,
this assumption is reasonable, but most of results should be
independent of this assumption. Please check.
\end{NB}

Hereafter we only consider the case when $p\colon Y = \bX\to X$ is the
blow-up of a point $0$ in a nonsingular projective surface $X$. Let
$\shfO_X(1)$ be an ample line bundle on $X$ and let $\bMm{m}(c)$ be
the moduli space of objects $E$ such that $E(-mC)$ is stable perverse
coherent with Chern character $c\in H^*(\bX)$. We say
$E$ is {\it $m$-stable\/} if this stability condition is satisfied.
When $m=0$ this was denoted by $\Mp_{Y/X/\C}(c)$ in
\thmref{thm:twisted}.%
\begin{NB}
Apr. 29:  I add the definition of the stable moduli space in
  \thmref{thm:twisted}.
\end{NB}


We assume that $r(c) > 0$ and
$\operatorname{gcd}\left((c_1,p^*\shfO_X(1)), r(c)\right) = 1$, then
$\mu$-stability and $\mu$-semistability (and hence also
(semi)stability) are equivalent. Then $E$ is stable perverse coherent
if and only if $E\in\Coh(\bX)\cap\Per(\bX/X)$ and $p_*(E)$ is
$\mu$-stable by \lemref{lem:basics}. In particular, we have
$\overline{M}^p_{Y/X/\C}(c) = \Mp_{Y/X/\C}(c)$ in the notation in
\thmref{thm:twisted}. This assumption is {\it essential\/} to compare
moduli spaces on $\bX$ and $X$. See \lemref{lem:surface} that the
relation of stabilities is delicate if we do not assume the
condition.%
\begin{NB}
Apr. 29:  I add an explanation. But we will never use the notaion
$\Mp_{Y/X/\C}(c)$.
\end{NB}

In case of framed sheaves on $\widehat{\proj}^2 =
\widehat{\C}^2\cup\linf$, moduli spaces corresponding to $\bMm{m}(c)$
for various $m$ are constructed by GIT quotients of the {\it common\/}
variety with respect to various choices of polarizations in the quiver
description. From a general construction by Thaddeus \cite{Th}, we can
construct a diagram \eqref{eq:flip} in Introduction, which induces a
flip $\bMm{m}(c)\dashrightarrow \bMm{m+1}(c)$ under some mild
assumptions.
Unfortunately our spaces $\bMm{m}(c)$ and $\bMm{m+1}(c)$ are not
quotients of a {\it common\/} space. Therefore we must construct the
space $\bMm{m,m+1}(c)$ and the diagram {\it by hand}. This will be
done in this section. We also study the fibers of $\xi_m$ and
$\xi_m^+$. Under a condition (= \cite[(4.4)]{Th}) Thaddeus showed that
fibers are weighted projective spaces. This condition (even in the
framed case) is {\it not\/} satisfied, but we will show that the
fibers are Grassmanns.

We have an isomorphism
\(
   \bMm{m}(c) \cong \bMm{0}(c e^{-m[C]})
\)
given by $E\mapsto E(-mC)$, twisting by the line
bundle $\shfO_{\bX}(-mC)$.
Therefore we only need to consider the case $m = 0$. But we also use
$\bMm{m}(c)$ to simplify the notation, and make the change of moduli
spaces apparent.

\subsection{A distinguished chamber -- torsion free sheaves on
  blow-down}\label{subsec:discham}
As is explained above, we restrict ourselves to the case $m=0$ in this
subsection.

By the definition of $\bMz(c)$ we have a morphism 
\begin{equation}\label{eq:xi}
\begin{matrix}
\xi\colon & \bMz(c)& \to &M^X(p_*(c))\\
          & E & \mapsto & p_*(E),
\end{matrix}
\end{equation}
where $M^X(p_*(c))$ is the moduli space of $\mu$-stable sheaves on $X$.

Here $p_*(c)$ is defined so that it is compatible with the
Riemann-Roch formula. So it is twisted from the usual
push-forward homomorphism as
\(
   p_*(c)= p^{\mathrm{usual}}_*(c\, \td \bX) (\td X)^{-1} .
\)
In particular, we have $p_*(e) = p_*(\ch(\shfO_C(-1))) = 0$.
This convention will be used throughout this paper.

\begin{Lemma}\label{lem:ext1}
Let $E \in \bMz(c)$. Then we have
$\Hom(E,{\cal O}_C(-1))=\Ext^2(E,{\cal O}_C(-1))=0$ and
$\Hom({\cal O}_C,E)=\Ext^2({\cal O}_C,E)=0$. 
In particular, $\chi(E,{\cal O}_C(-1))=\chi({\cal O}_C,E)
= - (c_1(E),[C]) \leq 0$. \textup(cf.\ \cite[Lem.~7.3]{perv}\textup).
\end{Lemma}
  
\begin{proof}
By the Serre duality, we have
$\Ext^2(E,{\cal O}_C(-1))=\Hom({\cal O}_C,E)^{\vee}$
and
$\Ext^2({\cal O}_C,E)=\Hom(E,{\cal O}_C(-1))^{\vee}$.
Then the assertions follow from the definition of stable perverse
coherent sheaves.
\end{proof}

\begin{NB}
$\chi(E,{\cal O}_C(-1))=\chi(E^{\vee}_{|C}(-1))=-\deg(E_{|C})=-(c_1(E),[C])$.
\end{NB}

We first consider the case $(c_1,[C]) = 0$.

\begin{Proposition}[\protect{cf. \cite[Prop.~7.4]{perv}}]\label{prop:blowdown}
  The morphism $\xi\colon \bMz(c)\to M^X(p_*(c))$ is an isomorphism if
  $(c_1,[C]) = 0$.
\end{Proposition}

\begin{proof}
  We have $\dim \Ext^1(E,\shfO_C(-1)) = \chi(E,\shfO_C(-1)) = -
  (c_1(E),[C]) = 0$. Therefore we have $E = p^* p_*(E)$ by
  \propref{prop:pervblowup}(2).
\end{proof}

\begin{NB}
Apr. 16: I make a new subsection and move some part of proofs here.  

Apr. 22: I again include this to the previou subsection.
\end{NB}


Besides the morphism $\xi\colon \bMz(c)\to M^X(p_*(c))$, we have
another natural morphism:
\begin{Lemma}\label{lem:othermor}
  We have a morphism 
\begin{equation*}
\begin{matrix}
\eta\colon & \bMz(c)& \to &M^X(p_*(c) + n\pt))\\
          & E & \mapsto & p_*(E(C)),
\end{matrix}
\end{equation*}
where $n = (c_1,[C])$.
\end{Lemma}

\begin{proof}
From \lemref{lem:R^1} the direct image sheaf $p_*(E(C))$ has the
Chern character\linebreak[2] 
$p_*(\ch(E)e^{[C]}) = p_*(ce^{[C]}) = p_*(c) + n\pt$.
Therefore it is enough to show that $p_*(E(C))$ is $\mu$-stable.

As $\Hom(\shfO_C, E(C)) = \Hom(\shfO_C(1),E) = 0$ from
$\Hom(\shfO_C,E) = 0$,
\begin{NB}
  Consider $\shfO_C^{\oplus 2} \twoheadrightarrow \shfO_C(1)$.
\end{NB}%
$p_*(E(C))$ is torsion free by \lemref{lem:torsionfree}. 

Consider $p_*(E) \to p_*(E(C))$. This is an isomorphism outside the
point $0$. Therefore the kernel is $0$ since $p_*(E)$ is torsion free
by the assumption. 
Since $p_*(E)$ is $\mu$-stable and $p_*(E(C))/p_*(E)$ is
$0$-dimensional, $p_*(E(C))$ is also $\mu$-stable.
\begin{NB}
A precise argument:

Let $S\subset p_*(E(C))$ be a nonzero subsheaf. If
$S\cap p_*(E) = 0$, then $S$ is torsion, and we have a
contradiction with the torsion freeness of $p_*(E(C))$. We have
$\mu(S\cap p_*(E)) < \mu(p_*(E))$ by the $\mu$-stability
of $p_*(E)$. This is equivalent to $\mu(S) < \mu(p_*(E(C)))$ as
$p_*(E)$ and $p_*(E(C))$ differ only at $0$, and have the same values of
$\mu$. Therefore $p_*(E(C))$ is $\mu$-stable.
\end{NB}
\end{proof}

\begin{NB}
Apr. 22:
Since we find that the diagram is {\it not \/} a correct one, so I
comment out this part.

Tensoring the line bundle $\shfO(C)$, we identify
$M^p(c e^{-[C]}) \cong \bMm{1}(c)$. Therefore the above $\eta$ can be
considered as a morphism $\eta\colon \bMm{1}(c)\to M^X(p_*(c))$. Hence
we have the diagram
\begin{equation}
  \label{eq:flip2}
  \xymatrix@R=.5pc{
\bMm{0}(c) \ar[rd]^{\xi} & & \bMm{1}(c) \ar[ld]_{\eta}
\\
& M^X(p_*(c)) &
}
\end{equation}
We define $\bMm{0,1}(c)$ be the image of $\eta$. 
\begin{NB2}
In \cite{Kota} $\bMm{0,1}(c)$ was the image of $\xi$. But as
$\bMm{0}(c)$ is possibly empty set, while $\bMm{1}(c)$ is not, I
changed the definition. Note that
$\bMm{0}(c)\neq\emptyset\Longrightarrow \bMm{1}(c)\neq\emptyset$ from
the BN theory.  
\end{NB2}
This will become the $m = 0^{\mathrm{th}}$ part of the diagram
\eqref{eq:flip} in Introduction.  However it is not even clear that
$\xi$ is mapped to $\bMm{0,1}(c)$ at this moment.
\end{NB}

\subsection{The morphism to the Uhlenbeck compactification downstairs}

Let $M_0^X(p_*(c))$ be the Uhlenbeck compactifiction, that is
$\bigsqcup M_{\text{\it lf}}^X(p_*(c) + m\pt)\times S^m X$, where
$M_{\text{\it lf}}^X(p_*(c) + m\pt)$ is the moduli space of $\mu$-stable
{\it locally free\/} sheaves on $X$. Then J.~Li \cite{Li} defined a
scheme structure which is projective, and there is a projective
morphism $\pi\colon M^X(p_*(c))_{\text{red}}\to M_0^X(p_*(c))$ sending
$E$ to $(E^{\vee\vee}, \Supp(E^{\vee\vee}/E))$.
In \cite[F.11]{NY2} the authors defined a projective morphism
$\widehat\pi$ from the moduli space of torsion-free sheaves on $\bX$
to $M_0^X(p_*(c))$. One of essential ingredients of the construction
was a morphism to $M^X(p_*(c e^{-m[C]}))$ for sufficiently large $m$.
Since we have the natural morphism $\bMm{m}(c)\to M^X(c e^{-m[C]})$ by
the construcion in the previous subsection, we can apply the same
method to define a projective morphism
\begin{equation}
  \label{eq:hatpi}
  \begin{matrix}
  \widehat\pi\colon & \bMm{m}(c)_{\text{red}} & \to & M_0^X(p_*(c)) \\
  & E & \mapsto & \left(p_*(E)^{\vee\vee}, \Supp(p_*(E)^{\vee\vee}/p_*(E))
  + \Supp(R^1 p_*(E))\right).
  \end{matrix}
\end{equation}

\begin{NB}
  We need to explain the `blowup' formula is universal somewhere. (May
  not be here.)

  It will be explained in \secref{sec:Betti}.
\end{NB}

\subsection{Smoothness}

\begin{Lemma}
Let $E\in\bMm{m}(c)$.
  We have an injective homomorphism
  \begin{equation*}
     \Hom(E, E\otimes K_{\bX}) \hookrightarrow
     \Hom(p_*(E)^{\vee\vee}, p_*(E)^{\vee\vee} \otimes K_{X}).
  \end{equation*}
\end{Lemma}

\begin{proof}
Since we have $\Hom(E, E\otimes K_{\bX})\cong
\Hom(E(-mC),E(-mC)\otimes K_{\bX})$ and
$p_*(E)^{\vee\vee} \cong p_*(E(-mC))^{\vee\vee}$, we may assume $m=0$.

If $E$ is perverse coherent, \lemref{lem:Bridgeland}(3) implies that
the natural homomorphism $p^* p_*(E)\to E$ induces an injection
\[
  \Hom(E, E\otimes K_{\bX}) \hookrightarrow\Hom(p^* p_*(E), E\otimes K_{\bX})
   \cong \Hom(p_*(E), p_*(E(C))\otimes K_{X})).
\]
We compose it with the induced homomorphism
from $p_*(E(C))\hookrightarrow p_*(E(C))^{\vee\vee} =
p_*(E)^{\vee\vee}$ to replace the right most term by
\(
   \Hom(p_*(E), p_*(E)^{\vee\vee} \otimes K_{X}).
\)
Let us consider the exact sequence $0\to p_*(E)\to
p_*(E)^{\vee\vee}\to Q\to 0$. Since $p_*(E)^{\vee\vee}$ is torsion
free, we have $\Hom(Q, p_*(E)^{\vee\vee}\otimes K_X) = 0$. We have
$\Ext^1(Q, p_*(E)^{\vee\vee}\otimes K_X)\cong
\Ext^1(p_*(E)^{\vee\vee\vee}, Q)^\vee = 0$ as $Q$ is
$0$-dimensional. Therefore we have
$$
  \Hom(p_*(E)^{\vee\vee}, p_*(E)^{\vee\vee} \otimes K_{X}) \cong
  \Hom(p_*(E), p_*(E)^{\vee\vee} \otimes K_{X})
  \eqno{\qedhere}
$$
\end{proof}

\begin{NB}
I think that the following is more obvious than the above lemma.
\end{NB}
\begin{Corollary}
  If $(\shfO_X(1),K_X) < 0$, then $\bMm{m}(c)$ is nonsingular of dimension
\(
   2r \Delta(c) - (r^2 - 1) \chi(\shfO_X) + h^1(\shfO_X),
\)
where
\(
  \Delta(c) := \int_{\bX} c_2 - (r-1)/(2r) c_1^2.
\)
\end{Corollary}

In general, the number
$2r \Delta(c) - (r^2 - 1) \chi(\shfO_X) + h^1(\shfO_X)$ is called
the {\it expected dimension\/} of $\bMm{m}(c)$, and denoted by
$\exp\dim \bMm{m}(c)$. If any irreducible component of $\bMm{m}(c)$
has dimension equal to $\exp\dim\bMm{m}(c)$, then we say
$\bMm{m}(c)$ has the expected dimension.
By the results of Donaldson, Zuo, Gieseker-Li, O'Grady (see
\cite[\S9]{HL}) there exists a constant $\Delta_0$ depending only on
$X$, $\shfO_X(1)$ and $r(c)$ such that $M^X(c)$ is irreducible,
normal, locally of complete intersection, and of expected dimension
for $\Delta(c)\ge \Delta_0$. The argument is applicable to
$\bMm{m}(c)$.
\begin{Proposition}\label{prop:genericsmooth}
  There exists a constant $\Delta_0$ such that $\bMm{m}(c)$ is
  irreducible, normal and of expected dimension if $\Delta(c)\ge \Delta_0$.
\end{Proposition}

\subsection{Evaluation homomorphisms}

This subsection is the technical heart of this paper. Starting from a
stable perverse coherent sheaf and a vector subspace in the space of
homomorphisms from $\shfO_C(-1)$ or to $\shfO_C$, we construct a new
stable perverse coherent sheaf. It will become a key to analyze the
change of stability conditions.

\begin{Lemma}\label{lem:key(-2)}
\textup{(1)}
Let $E\in\Per(\bX/X)\cap\Coh(\bX)$ such that $p_*(E)$ is torsion free,
and let $V \subset \Hom({\cal O}_{C}(-1),E)$ be a subspace.
Then the evaluation homomorphism induces an exact sequence
\textup(in the category $\Coh(\bX)$\textup)
\begin{equation}
  \label{eq:E'}
  0 \to V \otimes {\cal O}_{C}(-1) \xrightarrow{\operatorname{ev}}
  E \to E' := \Coker(\operatorname{ev}) \to 0,
\end{equation}
and $\Coker(\operatorname{ev}) \in \Per(\bX/X)(\cap\Coh(\bX)$\textup).

\textup{(2)} Let $E'\in\Per(\bX/X)\cap\Coh(\bX)$ and let
$V' \subset \Ext^1(E',{\cal O}_{C}(-1))$ be a subspace. Then
the associated extension \textup(in $\Coh(\bX)$\textup)
\[
0 \to (V')^{\vee} \otimes {\cal O}_{C}(-1) \to E \to E' \to 0
\]
defines $E \in \Per(\bX/X)(\cap \Coh(\bX)$\textup).

\textup{(3)} Let $F\in\Per(\bX/X)\cap\Coh(\bX)$ and let $U \subset
\Hom(F,{\cal O}_C)$ be a subspace. Then the evaluation homomorphism
induces an exact sequence \textup(in $\Coh(\bX)$\textup)
\begin{equation}
  0 \to F' := \Ker(\operatorname{ev}) \to F 
  \xrightarrow{\operatorname{ev}} U^\vee\otimes\shfO_C \to 0,
\end{equation}
and $F'\in \Per(\bX/X)(\cap\Coh(\bX)$\textup).

\textup{(4)} Let $F'\in\Per(\bX/X)\cap\Coh(\bX)$ and let
$U' \subset \Ext^1({\cal O}_C,F')$ be a subspace.
The associated extension \textup(in $\Coh(\bX)$\textup)
\[
0 \to F' \to F \to U' \otimes {\cal O}_C \to 0
\]
defines $F \in \Per(\bX/X)(\cap\Coh(\bX)$\textup) satisfying
$\Hom(\shfO_C,F) \cong \Hom(\shfO_C,F')$.
\begin{NB}
  Apr.\ 16 : corrected $=0$ in the previous version.
\end{NB}
\end{Lemma}

In (1) we have an exact sequence in $\Per(\bX/X)$:
\begin{equation*}
   0\to E \to E' \to V\otimes\shfO_C(-1)[1]\to 0.
\end{equation*}
This corresponds to the inclusion
$V\subset \Hom(\shfO_C(-1),E) \cong \Ext^1(\shfO_C(-1)[1],E)$. This
makes sense without the assumption that $p_*(E)$ is torsion free, but
$E'$ may not be a sheaf in general.
\begin{NB}
We only have
\begin{equation*}
   0 \to H^{-1}(E') \to V\otimes\shfO_C(-1) \to E \to H^0(E') \to 0.
\end{equation*}
For example, let $E = \C_p$ where $p$ is a point on
$C$. Then $\Hom(\shfO_C(-1),E) \cong \C$ and $\operatorname{ev}\colon
\shfO_C(-1)\to \C_p$ is not injective. We have
$H^{-1}(E') = \shfO_C(-2)$, $H^0(E') = 0$ in this case.
\end{NB}%
Similarly in (2) we have an exact sequence in $\Per(\bX/X)$:
\begin{equation*}
   0 \to E \to E' \to (V')^\vee\otimes \shfO_C(-1)[1] \to 0
\end{equation*}
corresponding to the inclusion
$V'\subset \Ext^1(E',\shfO_C(-1)) = \Hom(E',\shfO_C(-1)[1])$.
In (3), (4), the natural exact sequences in $\Coh(\bX)$ are also
exact sequences in $\Per(\bX)$.

In the following there are two ways of proofs to prove the
assertion. One is working in the category $\Coh(\bX)$ and check
the condition in \propref{prop:pervblowup}(1) to show that sheaves are
perverse coherent. The other is working in the category
$\Per(\bX/X)$ and check the condition $H^{-1}(\ ) = 0$ to show that
objects are, in fact, sheaves. We will give proofs of (2),(3) in the
first way, and ones of (1),(4) in the second way. We leave other
proofs as an exercise for a reader.

\begin{proof}[Proof of \lemref{lem:key(-2)}]
(1)
\begin{NB}
Proof in $\Coh(\bX)$:

Since $\bR p_*({\cal O}_{C}(-1))=0$, we have
$p_*(\Ker(\operatorname{ev}))=0$,
$p_*(\Ima(\operatorname{ev}))\cong R^1p_*(\Ker(\operatorname{ev}))$,
and
$R^1 p_*(\Ima(\operatorname{ev})) =0$
from the long exact sequence associated with
\(
  0\to\Ker(\operatorname{ev})\to V\otimes\shfO_C(-1)
  \to \Ima(\operatorname{ev})\to 0.
\)
Applying $p_*$ to
\(
  0\to\Ima(\operatorname{ev})\to E,
\)
we get an exact sequence $0\to p_*(\Ima(\operatorname{ev}))\to p_*(E)$.
As $p_*(E)$ is torsion free, we have
$p_*(\Ima(\operatorname{ev}))=0$. Hence we also get
$R^1p_*(\Ker(\operatorname{ev})) = 0$.
By \lemref{lem:T}, we have $\Ker(\operatorname{ev}) \cong {\cal
  O}_{C}(-1)^{\oplus r}$.
Since $\operatorname{ev}$ indices an injective homomorphism $V \to
\Hom({\cal O}_{C}(-1),E)$, we must have $r=0$, i.e.\
$\operatorname{ev}$ is injective.
From the exact sequence $E\to E' := \Coker(\operatorname{ev}) \to 0$, we
have%
\begin{NB2}
$R^1 p_*(E') = 0$ and
\end{NB2}
$\Hom(E',\shfO_C(-1)) = 0$. Therefore $E'\in\Per(\bX/X)$.%
\end{NB}%
Let $0\to E\to E'\to V\otimes\shfO_C(-1)[1]\to 0$ be the extension in
$\Per(\bX/X)$ corresponding to $V\subset \Ext^1(\shfO_C(-1)[1],E)$.
Then we have $\bR p_*(E) \cong \bR p_*(E')$. Applying $\bR p_*(\ )$
to an exact sequence $0\to H^{-1}(E')[1] \to E' \to H^0(E') \to 0$ in
$\Per(\bX/X)$, we get an injective homomorphism $R^1 p_*(H^{-1}(E'))
\to \bR p_*(E') \cong \bR p_*(E) \cong p_*(E)$. But $R^1
p_*(H^{-1}(E'))$ is a torsion, so we have $R^1 p_*(H^{-1}(E')) = 0$
from our assumption that $p_*(E)$ is torsion free.
Therefore $H^{-1}(E')\cong \shfO_C(-1)^{\oplus s}$ by
\lemref{lem:T}(2).

From the first exact sequence, we get a long exact sequence
\begin{equation*}
  0\to \Hom(\shfO_C(-1)[1],E)  \to \Hom(\shfO_C(-1)[1],E')
   \to V \to \Ext^1(\shfO_C(-1)[1],E')
\end{equation*}
The first term is $0$ as it is $\Ext^{-1}(\shfO_C(-1),E)$. The right
most arrow is injective by our choice. Therfore
$\Hom(\shfO_C(-1)[1],E') = 0$. Therefore $s$ must be $0$. This shows
$E'$ is a sheaf.

(2) 
\begin{NB}
From the exact sequence and $\bR p_*(\shfO_C(-1)) = 0$, we have
$R^1 p_*(E)=0$. 
\end{NB}%
Applying $\Hom(\bullet,\shfO_C(-1))$ to the given exact sequence, we get
\begin{equation*}
   0 \to \Hom(E,\shfO_C(-1)) \to V' \to \Ext^1(E',\shfO_C(-1)).
\end{equation*}
By the construction the right most arrow is injective. Hence
$\Hom(E,\shfO_C(-1)) = 0$.
Therefore $E \in \Per(\bX/X)$.

\begin{NB}
Proof in $\Per(\bX/X)$:

As $V'\subset \Ext^1(E',\shfO_C(-1)) = \Hom(E',\shfO_C(-1)[1])$, we
have a natural homomorphism
$\operatorname{ev}\colon E'\to (V')^\vee\otimes \shfO_C(-1)[1]$.
Consider $\Coker(\operatorname{ev})$ in the category $\Per(\bX/X)$. We
have
$H^0(\Coker(\operatorname{ev})) = 0$ as $H^0((V')^\vee\otimes
\shfO_C(-1)[1]) = 0$. We also have an exact sequence
\( 
   0\to \bR p_*(\Ima(\operatorname{ev}))
    \to \bR p_*((V')^\vee\otimes \shfO_C(-1)[1])
    \to \bR p_*(\Coker(\operatorname{ev})) \to 0
\)
in $\Coh(X)$. Therefore $\bR p_*(\Coker(\operatorname{ev})) = 0$.
Hence $\Coker(\operatorname{ev}) =
H^{-1}(\Coker(\operatorname{ev}))[-1] = \shfO_C(-1)^{\oplus s}[-1]$ by
\lemref{lem:T}. 
We have an exact sequence
\[
  0 \to \Hom(\Coker(\operatorname{ev}), \shfO_C(-1)[1])
    \to V' \to \Hom(E',\shfO_C(-1)[1]).
\]
Since the right most arrow is injective from the construction, we have 
$\Hom(\Coker(\operatorname{ev}), \shfO_C(-1)[1]) = 0$, i.e.\ $s = 0$.
Therefore $\operatorname{ev}$ is surjective in the category
$\Per(\bX/X)$. Therefore $E := \Ker(\operatorname{ev})$ gives an exact
sequence
\(
   0 \to E \to E' \to (V')^\vee\otimes \shfO_C(-1)[1] \to 0
\)
in $\Per(\bX/X)$. We have $H^{-1}(E) = 0$ as $H^{-1}(E') = 0$. So
$E\in\Coh(\bX)$.
\end{NB}

(3) We consider the following exact sequences in $\Coh(\bX)$:
\begin{equation}
\label{eq:twoexact}
\begin{gathered}
   0 \to \Ker(\operatorname{ev}) \to F \to \Ima(\operatorname{ev})\to 0,
\\
   0 \to \Ima(\operatorname{ev}) \to  U^\vee \otimes\shfO_C\to 
   \Coker(\operatorname{ev})\to 0.
\end{gathered}
\end{equation}
Applying $\bR p_*(\shfO_{\bX}(C)\otimes\bullet)$ to the second exact
sequence, we have 
\( 
 \bR p_*(\Ima(\operatorname{ev})\otimes\shfO_{\bX}(C)) = 0,
\)
\( 
 \bR p_*(\Coker(\operatorname{ev})\otimes\shfO_{\bX}(C)) = 0.
\)
By \lemref{lem:T}(2), we have
\(
  \Ima(\operatorname{ev}) \cong \shfO_C^{\oplus a},
\)
\(
  \Coker(\operatorname{ev}) \cong \shfO_C^{\oplus b}
\)
for some $a,b\in\Z_{\ge 0}$.

Applying $\Hom(\bullet,\shfO_C)$ to \eqref{eq:twoexact}, we get
\begin{gather*}
   0 \to \Hom(\Ima(\operatorname{ev}),\shfO_C)
     \to \Hom(F,\shfO_C),
\\
   0 \to \Hom(\Coker(\operatorname{ev}),\shfO_C)
     \to U
     \to \Hom(\Ima(\operatorname{ev}),\shfO_C).
\end{gather*}
As the composition of
\(
   U \to \Hom(\Ima(\operatorname{ev}),\shfO_C)
      \to \Hom(F,\shfO_C)
\)
is the natural inclusion, the left homomorphism is injective. So
$\Hom(\Coker(\operatorname{ev}),\shfO_C) = 0$. But as we already
observed $\Coker(\operatorname{ev}) = \shfO_C^{\oplus b}$, this means
$\Coker(\operatorname{ev}) = 0$. Hence $\operatorname{ev}$ is
surjective. 

Applying $\bR\Hom(\bullet,\shfO_C(-1))$ to the first exact sequence
of \eqref{eq:twoexact}, we get
\begin{equation*}
   0 \to \Hom(\Ker(\operatorname{ev}),\shfO_C(-1))
     \to \Ext^1(\Ima(\operatorname{ev}),\shfO_C(-1)).
\end{equation*}
But as $\Ima(\operatorname{ev})\cong U^\vee\otimes\shfO_C$, the
latter space is $0$, hence $\Hom(\Ker(\operatorname{ev}),\shfO_C(-1))
= 0$. Therefore $\Ker(\operatorname{ev})\in\Per(\bX/X)$.%
\begin{NB}
From $0\to \shfO_{\bX}(-C) \to \shfO_{\bX} \to \shfO_C \to 0$,
we have an exact sequence
\begin{equation*}
   \Hom(\shfO_{\bX}(-C),\shfO_C(-1))
   \to \Ext^1(\shfO_C,\shfO_C(-1))
   \to \Ext^1(\shfO_{\bX},\shfO_C(-1)).
\end{equation*}
We have 
$\Hom(\shfO_{\bX}(-C),\shfO_C(-1)) =
\Hom(\shfO_{\bX},\shfO_C(-1)\otimes\shfO_{\bX}(C))
= \Hom(\shfO_{\bX},\shfO_C(-2)) = 0$.
We also have
\(
   \Ext^1(\shfO_{\bX},\shfO_C(-1)) \cong
   H^1(\bX,\shfO_C(-1)) = 0.
\)
Therefore $\Ext^1(\shfO,\shfO_C(-1)) = 0$.
\end{NB}%
\begin{NB}
Proof in $\Per(\bX/X)$:

Consider exact sequences in $\Per(\bX/X)$:
\begin{gather*}
   0 \to \Ker(\operatorname{ev})\to F
     \to \Ima(\operatorname{ev}) \to 0.
\\
   0 \to \Ima(\operatorname{ev})
   \to \shfO_C\otimes U^\vee\to \Coker(\operatorname{ev})\to 0,
\end{gather*}
We have $H^{-1}(\Ima(\operatorname{ev})) = 0$ from
$H^{-1}(\shfO_C\otimes U^\vee) = 0$. So
$\Ima(\operatorname{ev})\in\Coh(\bX)$. Taking $\bR p_*(\bullet)$ of
the first exact sequence, we get
\begin{equation*}
   0 \to p_*(\Ima(\operatorname{ev}))
     \to \C_0 \otimes U^\vee
     \to \bR p_*(\Coker(\operatorname{ev})) \to 0.
\end{equation*}
Therefore $p_*(\Ima(\operatorname{ev}))\cong \C_0\otimes W$ for some
vector space $W\subset U^\vee$. We consider the commutative diagram
\begin{equation*}
  \begin{CD}
     p^* p_*(\Ima(\operatorname{ev})) @>>> p^* p_*(\shfO_C)\otimes U^\vee
\\
     @VVV @VVV
\\
     \Ima(\operatorname{ev}) @>>> \shfO_C\otimes U^\vee.
  \end{CD}
\end{equation*}
We have $p^*(\C_0) = \shfO_C$, hence the right vertical arrow is an
identity. From the above discussion, we have
$p^* p_*(\Ima(\operatorname{ev})) = \shfO_C\otimes W$, hence
the upper horizontal arrow is injective. The left vertical arrow is
surjective by \lemref{lem:Bridgeland}(3). The commutativity of the
diagram implies that the left vertical arrow is an identity. Therefore
$\Ima(\operatorname{ev}) = \shfO_C\otimes W$. We have
$\Coker(\operatorname{ev}) \cong \shfO_C\otimes U^\vee/W$ from the
second exact sequence.

Applying $\Hom(\bullet,\shfO_C)$ to the two exact sequences, we
have
\begin{gather*}
  0\to \Hom(\Ima(\operatorname{ev}),\shfO_C) 
   \to \Hom(F,\shfO_C),
\\
  0\to \Hom(\Coker(\operatorname{ev}),\shfO_C) 
   \to U
   \to \Hom(\Ima(\operatorname{ev},\shfO_C).
\end{gather*}
As the composition of
\(
U\to \Hom(\Ima(\operatorname{ev},\shfO_C) \to
\Hom(F,\shfO_C)
\)
is the natural inclusion, the left arrow is injective. Therefore
$\Hom(\Coker(\operatorname{ev}),\shfO_C) = 0$. But
as $\Coker(\operatorname{ev})\cong \shfO_C\otimes U^\vee/W$, we
conclude $\Coker(\operatorname{ev}) = 0$, i.e.\ $\operatorname{ev}$ is
surjective. Now $H^{-1}(\Ker(\operatorname{ev})) = 0$ as
$H^{-1}(F) = 0$. Therefore $\Ker(\operatorname{ev})\in\Coh(\bX)$.
\end{NB}

(4) %
\begin{NB}
Proof in $\Coh(\bX)$:

\begin{NB2}
Applying $\bR p_*(\bullet)$ to the given exact sequence and using $R^1
p_*(F') = 0 = R^1 p_*(\shfO_C)$, we have $R^1 p_*(F) = 0$.
\end{NB2}
Applying $\Hom(\bullet,\shfO_C(-1))$ to the given exact sequence, we
have
\[
   0\to (U')^\vee\otimes \Hom(\shfO_C,\shfO_C(-1)) 
   \to \Hom(F,\shfO_C(-1)) \to \Hom(F',\shfO_C(-1)).
\]
We have $\Hom(F',\shfO_C(-1)) = 0$ by the assumption
$F'\in\Per(\bX/X)$. We also have $\Hom(\shfO_C,\shfO_C(-1)) =
0$. Therefore $\Hom(F,\shfO_C(-1)) = 0$, and hence $F\in\Per(\bX/X)$.
The last assertion follows by applying
$\bR\Hom(\shfO_C,\bullet)$ to the given exact sequence, and observing
that
$U'\cong U'\otimes\Hom(\shfO_C,\shfO_C)\to \Ext^1(\shfO_C,F)$ is the
natural inclusion.
\end{NB}
Noticing $\shfO_C\in\Per(\bX/X)$, we consider the extension in the
category $\Per(\bX/X)$ instead of $\Coh(\bX)$. Then $H^{-1}(F) = 0$ as
$H^{-1}(F') = 0 = H^{-1}(\shfO_C\otimes U)$. Therefore
$F\in\Coh(\bX)$. So the extension is also an exact sequence in $\Coh(\bX)$.
\end{proof}

\begin{Lemma}\label{lem:keystable}
In the following \textup{($a$)} \textup($a=1,2,3,4$\textup) 
we suppose $E$, $E'$, $F$, $F'$ are as in the corresponding
\lemref{lem:key(-2)}\textup{($a$)}.

\textup{(1)} If $E$ is stable, then so is $E'$.

\textup{(2)} If $E'$ is stable, then so is $E$.

\textup{(3)} If $F$ is stable, then so is $F'$.

\textup{(4)} If $F'$ is stable, then so is $F$.
\end{Lemma}

\begin{proof}
(1)  As $p_*(E) \cong p_*(E')$ from the
exact sequence \eqref{eq:E'} and $\bR p_*(\shfO_C(-1)) = 0$, the
assertion is clear.

(2) The same argument as in (1).

(3) From the exact sequence
\(
  0 \to F' := \Ker(\operatorname{ev}) 
  \to F \to U^\vee\otimes \shfO_C \to 0
\)
we get an exact sequence
\begin{equation}\label{eq:directimage}
   0 \to p_*(F') \to p_*(F) \to U^\vee\otimes \C_0 \to 0.
\end{equation}
As $p_*(F)$ is torsion free, so is $p_*(F')$.
\begin{NB}
It is clear that $c_1(p_*(F')) = c_1(p_*(F))$ by
$H^2(X)\cong H^2(X\setminus \{0\})$. 

We also have $\rank p_*(F) = \rank p_*(F')$.
\end{NB}%
Note that $p_*(F')$ and $p_*(F)$ have the same $\mu$, as 
they differ only at $0$. Therefore $p_*(F')$ is also $\mu$-stable, and
hence $F'$ is stable.

(4) Since $F'$ is stable, $p_*(F')$ is torsion free, so we have
$\Hom(\shfO_C,F') = 0$ by \lemref{lem:torsionfree}. Therefore
$\Hom(\shfO_C,F) = 0$ by the last assertion in
\lemref{lem:key(-2)}(4), so $p_*(F)$ is also torsion free.
We have the exact sequence \eqref{eq:directimage}. Then by the same
argument as in \lemref{lem:othermor}, the $\mu$-stability of $p_*(F')$
and the torsion freeness of $p_*(F)$ implies the $\mu$-stability of
$p_*(F)$.
\end{proof}

\subsection{Stable sheaves becoming unstable after the wall-crossing}

Note that
\begin{alignat*}{4}
    & E\in \bMm{0}(c) & \quad & \Longrightarrow & \quad
    & \Hom(E, \shfO_C(-1)) = 0, & \quad & \Hom(\shfO_C, E) = 0,
\\
    & E\in \bMm{1}(c) & \quad & \Longrightarrow & \quad
    & \Hom(E, \shfO_C(-2)) = 0, & \quad & \Hom(\shfO_C(-1), E) = 0.
\end{alignat*}
Note also that
\begin{alignat*}{3}
  & \Hom(E,\shfO_C(-1)) = 0 & \quad & \Longrightarrow & \quad
  & \Hom(E,\shfO_C(-2)) = 0,
\\
  & \Hom(\shfO_C(-1),E) = 0 & \quad & \Longrightarrow & \quad
  & \Hom(\shfO_C,E) = 0.
\end{alignat*}

The following proposition says that these are the only conditions
which are altered under the wall-crossing.

\begin{Proposition}\label{prop:extension}
\textup{(1)} Suppose $E^-$ is $0$-stable, but not $1$-stable,
  i.e.\ $E^-$ is stable perverse coherent, but
  $E^-(-C)$ is not.
Then $V := \Hom(\shfO_C(-1),E^-) \neq 0$ and the evaluation
homomorphism gives rise an exact sequence
\begin{equation*}
    0 \to V\otimes \shfO_C(-1) \to E^- \to E'  \to 0
\end{equation*}
such that $E'$ is both $0$-stable and $1$-stable.
Moreover the induced homomorphism
\(
   V \to \linebreak[2]\Ext^1(E',\shfO_C(-1))
\)
is injective.

\begin{NB}
I comment out the `uniqueness' as its meaning is not clear.

Moreover, such an exact sequence is unique.

The `uniqueness' may be explained as follows:
$\shfO_C(-m-1)^{\oplus i} \subset E^-$ is the Harder-Narasimhan
filtration of $E^-$ with respect to the $(m+1)$-stability condition.

I need to explain why $\shfO_C(-m-1)$ is $(m+1)$-stable, as the
definition must be extended.

We need to show that
(i) $\shfO_C$, $E^-(-(m+1)C)/\shfO_C = E'(-(m+1)C)$ are stable, and
(ii) $p_*(\shfO_C) = \C_0\subset p_*(E^-(-(m+1)C))$ is the
Harder-Narashimhan filtration of $p_*(E^-(-(m+1)C))$.
\end{NB}

Conversely if $E'$ is both $0$ and $1$-stable and $E^-$ is the
extension corresponding to a nonzero subspace $V$ of
$\Ext^1(E',\shfO_C(-1))$ as above. Then $E^-$ is $0$-stable, but not
$1$-stable, and $V$ is naturally identified with
$\Hom(\shfO_C(-1),E^-)$.

These give a bijection
\begin{multline*}
   \{ E^-\in\bMm{0}(c)\setminus \bMm{1}(c) \mid
    \dim \Hom(\shfO_C(-1),E^-) = i\}
\\
    \longleftrightarrow
    \{ (E',V) \mid
    E' \in\bMm{0}(c-ie)\cap \bMm{1}(c-ie),
    V \in \operatorname{Gr}(i,\Ext^1(E',\shfO_C(-1))\}.
\end{multline*}

\textup{(2)} Suppose $E^+\in \bMm{1}(c)\setminus \bMm{0}(c)$. Then
$U := \Hom(E^+,\shfO_C(-1)) \neq 0$ and the evaluation homomorphism
gives rise an exact sequence
\begin{equation*}
    0 \to E' \to E^+ \to U^\vee \otimes \shfO_C(-1)\to 0
\end{equation*}
such that $E'$ is both $0$-stable and $1$-stable.
Moreover it induces an injection
\(
   U^\vee \to \linebreak[2]\Ext^1(\shfO_C(-1),E').
\)
\begin{NB}
Comment out:

Moreover, such an exact sequence is unique.%

  We may explain this also by the Harder-Narashimhan filtration, but
  it is in the sense of the Bridgeland stability condition, as
  $E^+(-mC), \shfO_C(-1)\notin\Per(\bX/X)$.
\end{NB}

Conversely if $E'$ is both $0$ and $1$-stable and $E^+$ is the
extension corresponding to a nonzero subspace $U^\vee$ of
$\Ext^1(\shfO_C(-1),E')$. Then $E^+$ is $1$-stable, but not
$0$-stable, and $U^\vee$ is naturally identified with
$\Hom(E^+,\shfO_C(-1))^\vee$.

These give a bijection
\begin{multline*}
   \{ E^+\in\bMm{1}(c)\setminus \bMm{0}(c) \mid
    \dim \Hom(E^+,\shfO_C(-1)) = i\}
\\
    \longleftrightarrow
    \{ (E',U^\vee) \mid
    E' \in\bMm{0}(c-ie)\cap \bMm{1}(c-ie),
    U^\vee \in \operatorname{Gr}(i,\Ext^1(\shfO_C(-1),E')\}.
\end{multline*}
\end{Proposition}

\begin{proof}
(1) By \lemref{lem:key(-2)}(1) we can consruct the exact sequence 
as in the statement. Then $E'$ is stable by \lemref{lem:keystable}(1).
Note also that we have $\Hom(\shfO_C(-1), E') = 0$ from the exact
sequence and our choice of $V$. Therefore $p_*(E'(-C))$ is torsion
free by \lemref{lem:torsionfree}.

Next consider
\(
   0 \to V\otimes \shfO_C \to E^-(-C) \to E'(-C) \to 0.
\)
We have an injective homomorphism
\begin{equation*}
   0 \to \Hom(E'(-C),\shfO_C(-1))
     \to \Hom(E^-(-C),\shfO_C(-1)) \cong \Hom(E^-,\shfO_C(-2)).
\end{equation*}
But $\Hom(E^-,\shfO_C(-2)) = 0$ as $\Hom(E^-,\shfO_C(-1)) = 0$ from
the assumption. Therefore $E'(-C)\in\Per(\bX/X)$.
Since $p_*(E')$ is $\mu$-stable and $p_*(E')/p_*(E'(-C))$ is
$0$-dimensional,  $p_*(E'(-C))$ is also $\mu$-stable.
\begin{NB}
  The original redundant argument:

Let $K$ be the kernel of $E'(-C)\to E'$ in the category
$\Coh(\bX)$. Then $p_*(K)$ is a subsheaf of $p_*(E'(-C))$ supported at
$0$. But as $p_*(E'(-C))$ is torsion free, $p_*(K) = 0$, and hence
$p_*(E'(-C))$ is a subsheaf of $p_*(E')$. As the value of $\mu$ are the
same for $p_*(E')$ and $p_*(E'(-C))$, $p_*(E'(-C))$ is also
$\mu$-stable.
\end{NB}%
This shows $E'(-C)$ is stable, and hence $E'$ is both $0$ and
$1$-stable. As $E(-C)$ is not stable by the assumption, we have $V\neq
0$. The injection $V\to \Ext^1(E',\shfO_C(-1))$ comes from the long
exact sequence associated with the given exact sequence, together with
$\Hom(E^-,\shfO_C(-1)) = 0$.

\begin{NB}
Uniqueness:

If we have an exact sequence of this type, we have a natural isomorphism
\begin{equation*}
  \Hom(\shfO_C,\shfO_C^{\oplus i}) \cong \Hom(\shfO_C(-1),E^-)
\end{equation*}
because $\Hom(\shfO_C(-1),E') = 0$ follows from the $1$-stability of
$E'$. Therefore the uniqueness is clear.
\end{NB}

Let us show the converse. From \lemref{lem:keystable}(2), $E^-$ is
stable. As $\Hom(\shfO_C(-1),E^-) \neq 0$, $E^-$ is not $1$-stable.
Moreover we have a natural isomorphism $V\cong \Hom(\shfO_C(-1),E^-)$
induced from the given exact sequence together with
$\Hom(\shfO_C(-1),E') = 0$.

It is also clear that these constructions give a bijection.

(2) As $E^+(-C)$ is stable by the assumption, $E^+(-C)\in
\Per(\bX/X)$. Therefore we can apply \lemref{lem:keystable}(3) for $F =
E^+(-C)$ with $U = \Hom(F,\shfO_C) \cong \Hom(E^+, \shfO_C(-1))$.
Then the corresponding exact sequence
\(
  0\to E' \to E^+ \to U^\vee\otimes \shfO_C(-1) \to 0
\)
defines $E'$ such that $E'(-C)$ is stable, i.e.\ $E'$ is $1$-stable.

Note $\Hom(E',\shfO_C(-1)) = 0$ from the exact sequence and our
choice of $U$.
Then $E'\in\Per(\bX/X)$ by \propref{prop:pervblowup}(1).
By \lemref{lem:othermor} we have $p_*(E')$ is $\mu$-stable, as
$E'(-C)$ is stable. Therefore $E'$ is also $0$-stable.
The injection $U^\vee\to \Ext^1(\shfO_C(-1),E')$ is induced from the
given exact sequence and $\Hom(\shfO_C(-1),E^+) = 0$.

Let us show the converse. From \lemref{lem:keystable}(4) applied to $F'
:= E'(-C)$ with $U' := \Ext^1(\shfO_C, E'(-C))$, $E^+(-C)$ is stable,
i.e.\ $E^+$ is $1$-stable.
A natural isomorphism
$U^\vee\cong \Hom(E^+,\shfO_C(-1))^\vee$ is induced from the given exact
sequence and $\Hom(E',\shfO_C(-1)) = 0$.
\end{proof}

\begin{NB}
Apr. 22:
Since the diagram \eqref{eq:flip2} is not correct, I comment out this part.

\begin{Corollary}
In the diagram \eqref{eq:flip2} we have
\(
  \xi(\bMm{0}(c)) \subset \eta(\bMm{1}(c)).
\)
\begin{NB2}
This is a set-theoretical sense. Do we have a scheme theoretical
statement ?  We defined $\bMm{0,1}(c) := \eta(\bMm{1}(c))$. Do we put the
induced reduced structure ? Then do $\xi$ map into $\bMm{0,1}(c)$ ? In
general, the answer is NO, as $\C[x]/(x^2)$ v.s.\ $\C[x]/x$. Then what
do we need ?
\end{NB2}
\end{Corollary}

\begin{NB2}
Since $\bMm{0}(c)$ may be empty while $\bMm{1}(c)$ is not, we cannot
have $\xi(\bMm{0}(c)) = \eta(\bMm{1}(c))$ in general.

If we have the smoothness assumption, then
$\bMm{0}(c)\neq \emptyset$ may be enough to conclude
$\xi(\bMm{0}(c)) = \eta(\bMm{1}(c))$.
\end{NB2}

\begin{proof}
In \propref{prop:extension} we have $\xi(E^-) = p_*(E') =
\eta(E^+)$. As we have
\begin{equation*}
  \dim \Ext^1(E',\shfO_C(-1)) = (c_1(E'),[C])
  \le (c_1(E'),[C]) + \rk E' =
  \dim \Ext^1(\shfO_C(-1),E')
\end{equation*}
from \lemref{lem:ext1}, if there exists an $i$-dimensional subspace
in $\Ext^1(E',\shfO_C(-1))$, then we can find an $i$-dimensional
subspace in $\Ext^1(\shfO_C(-1),E')$. Therefore the assertion follows.
\end{proof}
\end{NB}

\begin{Proposition}\label{prop:mult}
Let $E^-\in \bMm{m}(c)$ \textup(resp.\ $E^+\in\bMm{m+1}(c)$\textup)
and suppose that its image under $\widehat\pi$ in \eqref{eq:hatpi} has
the multiplicity $N$ at $0$ in its symmetric product part. If $m > N$,
then $E^-$ \textup(resp.\ $E^+$\textup) is $(m+1)$-stable \textup(resp.\
$m$-stable\textup).
\end{Proposition}

\begin{proof}
Suppose that $E^-$ is $m$-stable, but not $(m+1)$-stable for some $m
\ge 0$.  From \propref{prop:extension}(1) $i :=
\dim\Hom(\shfO_C(-m-1),E^-) > 0$ and we have an exact sequence
\begin{equation*}
   0 \to p_*(E^-) \to p_*(E') \to \C_0^{\oplus im} \to R^1 p_*(E^-)
   \to R^1 p_*(E') \to 0.
\end{equation*}
As $E^-$, $E'$ are $m$-stable, and hence $\Hom(\shfO_C(-m-1), E')$,
$\Hom(\shfO_C(-m-1), E^-)$ are $0$. Therefore we have $\Hom(\shfO_C,
E^-) = 0 = \Hom(\shfO_C, E')$. By \lemref{lem:torsionfree}
$p_*(E^-)$, $p_*(E')$ are torsion free. Therefore $p_*(E^-)\to
p_*(E^-)^{\vee\vee}$, $p_*(E')\to p_*(E')^{\vee\vee}$ are injective.
From the above exact sequence, we have $p_*(E^-)^{\vee\vee} \cong
p_*(E')^{\vee\vee}$. Therefore we have an exact sequence
\begin{equation*}
   0 \to p_*(E')/p_*(E^-) \to p_*(E^-)^{\vee\vee}/p_*(E^-)
     \to p_*(E')^{\vee\vee}/p_*(E') \to 0.
\end{equation*}
We get
\begin{equation*}
\begin{split}
   \operatorname{len}_0(p_*(E^-)^{\vee\vee}/p_*(E^-)) 
   \ge\; &
   \operatorname{len}_0(p_*(E')^{\vee\vee}/p_*(E')) 
   + i m - \operatorname{len}_0(R^1 p_*(E^-))
\\
   \ge\; &  m - \operatorname{len}_0(R^1 p_*(E^-)),
\end{split}
\end{equation*}
where $\operatorname{len}_0$ is the length of the stalk at $0$.
This inequality is impossible if $m > \allowbreak
\operatorname{len}_0(p_*(E^-)^{\vee\vee}/p_*(E^-)) +\allowbreak
\operatorname{len}_0(R^1 p_*(E^-))$. From the definition of
$\widehat\pi$, we get the assertion. The proof for $E^+$ is the same.
\end{proof}

\subsection{Brill-Noether locus and moduli of coherent systems}

Motivated by \propref{prop:extension}, we introduce the Brill-Noether
locus:
\begin{Definition}[Brill-Noether locus]
We set
\begin{gather*}
\bMm{m}(c)_i := \{ E^- \in \bMm{m}(c)\mid
\dim \Hom({\cal O}_C(-m-1),E^-)=i \}, 
%
\\
\bMm{m+1}(c)^i :=\{E^+ \in \bMm{m+1}(c)\mid 
\dim \Hom(E^+,{\cal O}_C(-m-1))=i \}, 
\end{gather*}
\end{Definition}

When we replace '$=i$' by '$\ge i$' in the right hand side, the
corresponding moduli spaces are denoted by the left hand side with
'$i$' replaced by '$\ge i$'. 

The scheme structures on $\bMm{m}(c)_i$, $\bMm{m+1}(c)^i$ are defined
as in \cite[5.5]{Mar} (cf.\ \cite[Ch.~IV]{ACGH}). Let us briefly
explain an essential point. Let $\mathcal E^-$ be a universal family
over $\bX\times \bMm{m}(c)$ and let $f$ be the projection to
$\bMm{m}(c)$. Then we construct an exact sequence
\begin{equation}\label{eq:lfmodel}
   0 \to \Hom_f(\shfO_C(-m-1),\mathcal E^-)
     \to \mathcal F_0 \xrightarrow{\rho} \mathcal F_1 
     \to \Ext_f^1(\shfO_C(-m-1),\mathcal E^-) \to 0
\end{equation}
such that $\mathcal F_0$, $\mathcal F_1$ are vector bundles. Then we
define $\bMm{m}(c)_{\ge i}$ to be the zero locus of $\Wedge^{\rk
  \mathcal F_0 + 1 - i} \rho$.
Moreover $\bMm{m}(c)_{i}$ is the open subscheme $\bMm{m}(c)_{\ge
  i}\setminus \bMm{m}(c)_{\ge i+1}$ of $\bMm{m}(c)_{\ge i}$.

In \propref{prop:extension} we have $E^-\in\bMm{0}(c)_i$,
$E^+\in\bMm{1}(c)^i$. Therefore \propref{prop:extension} says that
when we change the stability condition from $0$ to $1$,
$\bMm{0}(c)_{\ge 1}$ is replaced by $\bMm{1}(c)^{\ge 1}$, and
$\bMm{0}(c)_0 \cong \bMm{0}(c)\cap\bMm{1}(c) \cong \bMm{1}(c)^0$ is
preserved. We have a set-theoretical diagram
\begin{equation}\label{eq:set-diagram}
\begin{aligned}[m]
\xymatrix@R=.5pc @M=0pc
{
\bMm{0}(c)\; \ar[rd]!UL && \;\bMm{1}(c) \ar[ld]!UR
\\
& 
\; \bigsqcup_i \bMm{0}(c-ie)_0
\cong 
\bigsqcup_i \bMm{1}(c-ie)^0 \;
&
}
\end{aligned}
\end{equation}
and fibers over $E'\in \bMm{0}(c-ie)_0$ of the left and right arrows
are Grassmann
\linebreak[2]$\operatorname{Gr}(i,\Ext^1(E',\shfO_C(-1)))$ and
$\operatorname{Gr}(i,\Ext^1(\shfO_C(-1),E'))$ respectively. This is
similar to \eqref{eq:flip}, but we need to endow the target
$\bigsqcup_i \bMm{0}(c-ie)_0$ with a scheme structure.

\begin{NB}
Note that $\Ext^2(\shfO_C(-1), E) = \Hom(E,\shfO_C(-2))^\vee = 0$
for $E\in\bMm{0}(c)$. Therefore
\begin{equation*}
   \dim \Hom(\shfO_C(-1),E) = i
   \Longleftrightarrow
   \dim \Ext^1(\shfO_C(-1),E) = i - \chi(\shfO_C(-1),E)
   = i + (c_1,[C]) - r.
\end{equation*}
On the other hand, we have
$\Ext^2(E, \shfO_C) = \Hom(\shfO_C(1), E)^\vee = 0$ for
$E\in\bMm{0}(c)$. Therefore
\begin{equation*}
   \dim \Hom(E, \shfO_C) = i
   \Longleftrightarrow
   \dim \Ext^1(E, \shfO_C) = i - \chi(E, \shfO_C)
   = i + (c_1,[C]) + r.
\end{equation*}
We also have $\Ext^1(\shfO_C(-1),E)\cong
\Ext^1(E,\shfO_C(-2))^\vee$. Therefore
\begin{equation*}
   \dim \Hom(\shfO_C(-1),E) = i
   \Longleftrightarrow
   \dim \Ext^1(E,\shfO_C(-2)) = i + (c_1,[C]) - r.
\end{equation*}
We have an exact sequence
\begin{equation*}
   0 \to \Hom(E, \C_p \oplus \C_q) \to
         \Ext^1(E,\shfO_C(-2)) \to \Ext^1(E,\shfO_C)
         \to \Ext^1(E,\C_p\oplus \C_q) \to 0,
\end{equation*}
where $p$, $q$ are distict points on $C$.
\end{NB}


Let us introduce moduli spaces of coherent
systems in order to study Brill-Noether loci more closely.

\begin{Definition}
Let $\bM(c,n)$ be the moduli space of coherent systems
$(E,\linebreak[4] V \subset\Hom({\cal O}_C(-1),E))$ such that
$E \in \bMm{0}(c)$ and $\dim V=n$.
\end{Definition}
The construction is standard:
$\bM(c,n)$ is constructed as a closed subscheme of a suitable
Grassmannian bundle over $\bMm{0}(c)$.
We have a natural morphism $q_1\colon \bM(c,n)\to \bMm{0}(c)$. We have a
universal family $\mathcal V$ which is a rank $n$ vector subbundle of
$q_1^*(\mathcal F_0)$ contained in $\Ker(q_1^*\rho)$, where $\mathcal
F_1$ and $\rho$ are as in \eqref{eq:lfmodel}.

For $(E,V) \in \bM(c,n)$, we set $E':=\Coker(\operatorname{ev}\colon V
\otimes {\cal O}_C(-1) \to E)$. By \lemref{lem:keystable}(1), we have 
$E'\in \bMm{0}(c-ne)$.
Thus we get a morphism $q_2\colon \bM(c,n) \to \bMm{0}(c-ne)$.

Therefore we have the following diagram:
\begin{equation}\label{eq:diagram}
\begin{aligned}[m]
\xymatrix@R=.5pc{
& \bM(c,n) \ar[ld]_{q_1} \ar[rd]^{q_2} &
\\
\bMm{0}(c) & & \bMm{0}(c-ne)
}
\end{aligned}
\end{equation} 

Conversely suppose that $E'\in \bMm{0}(c-ne)$ and an $n$-dimensional
subspace $V^\vee\subset\linebreak[3] \Ext^1(E',\shfO_C(-1))$ are
given. Then we can consider the corresponding extension \eqref{eq:E'}.
By \lemref{lem:keystable}(2), we
have $E\in \bMm{0}(c)$.
Moreover, the exact sequence \eqref{eq:E'} induces an injection $V\to
\Hom(\shfO_C(-1),E)$. Thus $(E,V)\in \bM(c,n)$. This gives an
isomorphism from $\bM(c,n)$ to the moduli space of `dual' coherent systems
$(E',V^\vee\subset\Ext^1(E',\shfO_C(-1)))$ such that $E'\in
\bMm{0}(c-ne)$, $\dim V^\vee = n$.

Note $\Hom(E',\shfO_C(-1)) = 0 = \Ext^2(E',\shfO_C(-1))$ and
$\dim \Ext^1(E',{\cal O}_C(-1)) 
\begin{NB}
= (c-ne,[C])
\end{NB}%
= (c_1,[C]) + n$ by \lemref{lem:ext1}. This is a constant independent of
$E'\in \bMm{0}(c-ne)$.
Let $\mathcal E'$ be an universal family over $\bX\times
\bMm{0}(c-ne)$ and let $f$ be the projection to $\bMm{0}(c-ne)$. By the
above observation $\Ext^1_f(\mathcal E',\shfO_C(-1))$ is a vector
bundle of rank $(c_1,[C])+n$ over $\bMm{0}(c-ne)$.

Therefore
\begin{Lemma}\label{lem:diagram}
The projection $q_2$ identifies $\bM(c,n)$ with the Grassmann bundle
\linebreak[4]
$\operatorname{Gr}(n,\Ext^1_f(\mathcal E',\shfO_C(-1)))$ of
$n$-dimensional subspaces in $\Ext^1_f(\mathcal E',\shfO_C(-1))$
over $\bMm{0}(c-ne)$.
\begin{NB}
The following needs the smoothness assumption.  

In particular, 
$\bM(c,n)$ is smooth and
\begin{equation}\label{eq:system}
\dim \bM(c,n)=\dim \bMm{0}(c)+n\chi({\cal O}_C(-1),E)-n^2=
\dim \bMm{0}(c)-n((c_1(E),C)+\rk(E)+n),
\end{equation}
provided $\bM(c,n) \ne \emptyset$.
\end{NB}%
In particular, we have
\begin{equation*}
\begin{split}
   & \dim \bM(c,n) = \dim \bMz(c-ne) + n(c_1,[C]), 
\\
   & \exp\dim \bM(c,n) = \exp\dim \bMz(c) 
   - n(n + r + (c_1,[C])).
  \end{split}
\end{equation*}
If $(\shfO_X(1),K_X) < 0$, then $\bM(c,n)$ is smooth and of expected
dimension, provided it is nonempty.
\end{Lemma}

\begin{NB}
We have
\begin{equation*}
\begin{split}
   & \Delta(c - ne) = \int_{\bX} \left[-\ch_2(c-ne) 
     + \frac1{2r} c_1(c-ne)^2\right]
\\
   =\; &  \int_{\bX} \left[- \ch_2(c) - \frac{n}2 \pt 
     + \frac1{2r} \{ c_1(c)^2 - 2n (c_1(c),[C]) - n^2 \pt\}\right]
\\
   =\; &  \Delta(c) - \frac{n}{2r} (n + r + 2 (c_1,[C])).
\end{split}
\end{equation*}
Therefore
\begin{equation*}
   \exp\dim \bM(c,n) = \exp\dim \bMz(c-ne) + n(c_1,[C])
   = \exp\dim \bMz(c) - n(n + r + (c_1,[C])).
\end{equation*}
\end{NB}

\begin{NB}
These can be shown by using the deformation theory:
We note that the obstruction space is 
$\Ext^2(E',E)=\Hom(E,E'(K_X))^{\vee}$.
For a non-zero homomorphism $E \to E'(K_X)$,
we have a non-zero homomorphism
$\pi_*(E) \to \pi_*(E'(K_X))$.
By our assumption on $K_Y$ and the $\mu$-stability of
$\pi_*(E)$,
this is impossible.
Thus $\Ext^2(E',E)=0$.  
Since
$\dim \bM(c,n)=\dim\Ext^1(E',E)-\dim PGL(n)=-\chi(E',E)-n^2+1$, 
we get \eqref{eq:system}.
\end{NB}

\begin{NB}
\begin{Remark}
Instead of assuming the ampleness of $-K_Y$,
we may assume that 
$\Hom(\pi_*(E),\pi_*(E)^{\vee \vee}  \otimes K_Y)_0=0$
to show the smoothness of $\bM(c,n)$.
\end{Remark}
\end{NB}

\begin{Proposition}\label{prop:BN}
Let us consider the diagram in \eqref{eq:diagram}.

\begin{NB}
I have exchanged (1) and (2). I think that the second statement is
easier to understand after the reader knows $q_1^{-1}(\bMm{0}(c)_n)$
is an open subscheme. Then we need to know the image (1) before.

Is this OK ? I am not sure that the proof of (1) requires (2).
\end{NB}
\textup{(1)}
The image of $q_1\colon \bM(c,n) \to \bMm{0}(c)$ is
the Brill-Noether locus $\bMm{0}(c)_{\geq n}$.

\textup{(2)} The morphism $q_1\colon \bM(c,n) \to \bMm{0}(c)_{\geq n}$
becomes an isomorphism if we restrict it to the open subscheme
$q_1^{-1}(\bMm{0}(c)_n)$.
\begin{NB}
Kota's Original:

We set $\bM(c,n)_i:=\{(E,V) \in \bM(c,n)| E \in \Mp(c)_i \}$.
Then
$q_1\colon \bM(c,n)_n \to \Mp(c)$ is an immersion and the image is
$\Mp(c)_n$.
\end{NB}

\textup{(3)}
$\bMm{0}(c)_n$ is a $\operatorname{Gr}(n, n+(c_1,[C]))$-bundle over
$\bMm{0}(c-ne)_0$ via the restriction of $q_2$.

\begin{NB}
May 7 : Please check the followings are OK.  
\end{NB}

\textup{(4)}
\begin{NB}
  I remove the condition that $\bMz(c-ne)$ is of expected dimension.
\end{NB}%
Suppose $\bMz(c-ne)$ is irreducible. Then
$\overline{\bMm{0}(c)_n}=\bMm{0}(c)_{\ge n}$.

\textup{(5)}
Suppose that $\bMz(c)$ and $\bMz(c-ne)$ are irreducible and of
expected dimension. Suppose further that $\bMz(c-ne)$ is normal.
Then the Brill-Noether locus $\bMm{0}(c)_{\geq n}=
\overline{\bMm{0}(c)_n}$ is Cohen-Macauley and normal.
\end{Proposition}

\begin{proof}
(1) is clear.
\begin{NB}
  As we have a rank $n$ vector subbundle $\mathcal V\subset \mathcal
  F_0$ contained in the kernel of $\rho$, the image is contained in
  the Brill-Noether locus $\bMm{0}(c)_{\ge n}$. 

  To check the surjectivity, we consider set-theoretically.
\end{NB}

(2) We use the following facts: 
(i) $\bM(c,n) \to \bMm{0}(c)$ is projective,
(ii) 
we have an exact sequence
\begin{equation*}
{\Bbb C} \to V^{\vee} \otimes V \overset{g}{\to}
\Ext^1(E',E) {\to} \Ext^1(E,E),
\end{equation*}
with $V = \Hom(\shfO_C(-1),E)$,
\begin{NB}
We have
  \begin{equation*}
    \xymatrix@R=.5pc{    
      & \Hom(E,E) \ar[r] & V^{\vee}\otimes\Hom(\shfO_C(-1),E) \ar[lld]
\\
      \Ext^1(E',E) \ar[r] & \Ext^1(E,E)
      }
  \end{equation*}
\end{NB}%
(iii)
the Zariski tangent space of 
$\bM(c,n)$ at $(E,V)$ is $\coker g=\Ext^1(E',E)/(V^{\vee} \otimes V)$.
(See \cite{He}.)
\begin{NB}
We have an exact sequence
$0 \to (V\otimes \shfO_C(-1), V) \to (E,V) \to (E',0) \to 0$
in the category of coherent systems (cf.\ \cite[Proof of Lem.~1.7]{He}),
where $V \subset \Hom(\shfO_C(-1), V\otimes\shfO_C(-1))\cong V$ is the
canonical homomorphism. Then we have
\begin{equation*}
    \xymatrix@R=.5pc{    
      & \Hom((E,V),(E,V)) \ar[r] 
      & \Hom((V\otimes \shfO_C(-1),V), (E, V)) \ar[lld]
\\
      \Ext^1((E',0),(E,V)) \ar[r] & \Ext^1((E,V),(E,V)) \ar[r]
      & \Ext^1((V\otimes \shfO_C(-1),V), (E, V)).
      }
\end{equation*}
Then $\Ext^1((V\otimes \shfO_C(-1),V), (E, V)) = 0$
from \cite[Prop.~1.5]{He}.

Note that the assumption $\Hom(V,V)\to \Hom((E,V),(E,V))$ is
surjective is not necessary satisfied.
\end{NB}

(3) follows from \lemref{lem:diagram} and (1).

(4) From the assumption and \lemref{lem:diagram} $\bM(c,n)$ is
irreducible. Then the assertion follows from (1).

(5) From the assumption $\bMz(c)$ is a local complete intersection
(\cite[Th.~4.5.8]{HL}), and hence Cohen-Macauley.
Since the determinantal subvariety $\bMz(c)_{\ge n}$ has the correct
codimension $\codim_{\bMz(c)}(\bMz(c)_{\ge n}) =
n(n+r+(c_1,C))$, it is also Cohen-Macauley.
\begin{NB}
  See e.g., Fulton : Intersection Theory, Th.~14.3(c).
\end{NB}%
From the assumption and \lemref{lem:diagram} $\bM(c,n)$ is normal.
Therefore $\bMz(c)_{\ge n}$ is also normal.
\end{proof}

\begin{Remark}
  If we take $\Delta(c-ne) \ge \Delta_0$, where $\Delta_0$ is as
  in \propref{prop:genericsmooth}, $\bMz(c)$, $\bMz(c-ne)$,
  are irreducible, normal and of expected dimension.
\end{Remark}

\begin{Lemma}\label{lem:BN-empty}
Suppose $(\shfO_X(1),K_X) < 0$.
If
$2(c_1,[C])> 2r\Delta-r-1-(r^2-1)\chi(\shfO_X) + h^1(\shfO_X)$, then
$\bMz(c)_{\ge 1} = \emptyset$.
\end{Lemma}

\begin{proof}
We may asume that $(c_1,[C]) \geq 0$.

Suppose that $\bMz(c-ne)\neq\emptyset$ for $n > 0$. Then
\begin{equation*}
\begin{split}
  0 \le\; & \dim \bMz(c - ne)%
\\
  =\; &  2r \Delta(c) - n(n+r+2(c_1,[C])) - (r^2 - 1)\chi(\shfO_X) +
  h^1(\shfO_X)
\\
 \le \; & 2r \Delta(c) - (1+r+2(c_1,[C])) - (r^2 - 1)\chi(\shfO_X) +
  h^1(\shfO_X).
\end{split}
\end{equation*}
The result follows.
\end{proof}

Next we consider the corresponding study for another Brill-Noether
locus $\bMm{1}(c)^i$ appearing in the other side of the wall.

\begin{NB}
Apr. 16 : I changed the stability condition.
\end{NB}

\begin{Definition}\label{def:cosystem}
Let $\bN(c,n)$ be the moduli of coherent systems
$(E,U \subset \Hom(E,{\cal O}_C(-1)))$ such that
$E \in \bMm{1}(c)$ and $\dim U=n$.
\end{Definition}

We have a natural morphism $q_1'\colon \bN(c,n)\to \bMm{1}(c)$.

For $(E,U)\in \bN(c,n)$, we set $E':=\Ker(E \to U^{\vee} \otimes {\cal
  O}_C(-1))$.  By \lemref{lem:keystable}(3), we have $E' \in \bMm{1}(c-ne)$.
We thus have the diagram:
\begin{equation}\label{eq:diagram'}
\begin{aligned}[m]
\xymatrix@R=.5pc{
& \bN(c,n) \ar[ld]_{q'_1} \ar[rd]^{q'_2} &
\\
\bMm{1}(c) & & \bMm{1}(c-ne).
}
\end{aligned}
\end{equation} 

Conversely suppose that $E' \in \bMm{1}(c-ne)$ and an $n$-dimensional
subspace $U^\vee\subset \Ext^1(\shfO_C(-1),E')$ are given. Then we can
consider the associated exact sequence 
\begin{equation}
  \label{eq:E'2}
  0 \to E' \to E \to U^\vee\otimes\shfO_C(-1) \to 0
\end{equation}
by \lemref{lem:keystable}(4), we have $E\in\bMm{1}(c)$.
Moreover \eqref{eq:E'2} induces an injection $U\subset
\Hom(E,\shfO_C(-1))$. Therefore $(E,U)\in \bN(c,n)$.

Note
\(
   \Hom(\shfO_C(-1),E') = 0 = \Ext^2(\shfO_C(-1),E')
\)
and
\(
  \dim \Ext^1({\cal O}_C(-1),E')
   = (c_1(E'),[C]) + \rk E'
\)
by \lemref{lem:ext1}. If $\mathcal E'$ denotes an universal sheaf
over $\bMm{1}(c-ne)$, then $\Ext^1_f(\shfO_C(-1),\mathcal E')$ is a vector
bundle of rank $(c_1,[C]) + n + r$ over $\bMm{1}(c-ne)$.

We have
\begin{Lemma}\label{lem:cosystem}
The projection $q_2'$ identifies $\bN(c,n)$ with the Grassmann bundle
\linebreak[4]
$\operatorname{Gr}(n,\Ext^1_f(\shfO_C(-1),\mathcal E'))$ of
$n$-dimensional subspaces in $\Ext^1_f(\shfO_C(-1),\mathcal E')$
over $\bMm{1}(c-ne)$.
\begin{NB}
Again under the smoothness assumption.  I do not change the dimension
formula yet.
\end{NB}%
In particular, we have
\begin{equation*}
\begin{split}
   & \dim \bN(c,n) = \dim \bMm{1}(c-ne) + n(r + (c_1,[C])), 
\\
   & \exp\dim \bN(c,n) = \exp\dim \bMm{1}(c) 
   - n(n + (c_1,[C])).
  \end{split}
\end{equation*}
If $((\shfO_X(1),K_X) < 0$, then
$\bN(c,n)$ is smooth and of expected dimension, provided it is nonempty.
\begin{NB}
Kota's original definition was, in our notation, 
$\bN(c e^{[C]},n)$. Then $(c_1(c e^{[C]}), [C]) = (c_1, [C]) - r$.
Therefore the formula
\begin{equation}\label{eq:cosystem}
\dim \bN(c,n)=\dim M_H^p(c)-n((c_1(E),C)-\rk(E)+n)
\end{equation}
is the same with ours.
\end{NB}
\end{Lemma}

We have the statements corresponding to \propref{prop:BN}. Since they
are very similar, we omit them.

\begin{Remark}\label{rem:Grass}
As already mentioned in the introduction,
a similar Grassmann bundle structure has been observed in the
contexts of quiver varieties \cite{Na:1994} and an exceptional
bundle on K3 \cite{Yos,Mar} (see also \cite{Na:Missouri} for an
exposition).
The moduli spaces of coherent systems in \cite{Yos} and Hecke
correspondences \cite{Na:1994} play the same role connecting two
moduli spaces with different Chern classes.
However, there is a sharp distinction between the above blowup case
and the other cases, which can be considered as the $(-2)$-curve.
Namely the Grassmann bundle is defined only on a Brill-Noether locus,
as both $\Hom$ and $\Ext$ survive in general for the other cases.
\end{Remark}

\subsection{Contraction of the Brill-Noether locus}

Consider $\bMm{0}(c)$ and set $n := (c_1,[C])$, $e := \ch(\shfO_C(-1))$,
$c_\perp := c + n e$. Then we have $(c_\perp, [C]) = 0$.
Therefore we have $\bMm{0}(c_\perp)\cong M^X(p_*(c_\perp)) = M^X(p_*(c))$ by
\propref{prop:blowdown}. Therefore $\xi$ in \eqref{eq:xi} can be
considered as $\xi \colon \bMm{0}(c)\to \bMm{0}(c_\perp)$. Explicitly it is
given by $\xi(E) = p^*(p_*(E))$.

\begin{NB}
We need to assume $n\ge 0$.
\end{NB}

\begin{Proposition}\label{prop:open}
Suppose $n := (c_1,[C])\ge 0$. Let $\xi$ be as in \eqref{eq:xi}.

\textup{(1)} $\xi(\bMm{0}(c))$ is identified with the Brill-Noether
locus $\bMm{0}(c_\perp)_{\ge n}$ via the above isomorphism. 
In particular, $\xi(\bMm{0}(c))$ is a Cohen-Macauley and normal
subscheme of $M^X(p_*(c))$, provided $\bMz(c_\perp)$, $\bMz(c)$ are
irreducible and of expected dimension, and $\bMz(c)$ is normal.
\begin{NB}
  I put the generic smoothness assumption.
\end{NB}

\textup{(2)} $\xi$ is an immersion on $\bMm{0}(c)_0$.

\textup{(3)} Each Brill-Noether stratum $\bMz(c_\perp)_{n+i}$ is
isomorphic to $\bMz(c-ie)_0$, so $\bMz(c_\perp)_{\ge n}$ can be
considered as a scheme structure on $\bigsqcup_{i}\bMz(c-ie)_0$
requested in \eqref{eq:set-diagram}.

\textup{(4)} $\xi$ maps $\bMz(c)_i$ to
$\bMz(c_\perp)_{n+i}$, and it can be identified with the Grassmann
bundle $\bMz(c)_i\to \bMz(c-ie)_0$ in \eqref{eq:set-diagram} under
the isomorphism in \textup{(3)}.

\begin{NB}
Here is Kota's original statement.  

\textup{(3)} By $\xi$, the Grassmannian structure of $\bMm{0}(c)_i\to
\bMm{0}(c-ie)_0$ is contracted. Thus set-theoretically,
 we have
\begin{equation*}
\begin{split}
\xi(\bMm{0}(c))&=\bigcup_i \bMm{0}(c_\perp)_{i+n}\\
& \leftrightarrow\bigcup_i \bMm{0}(c-ie)_0.
\end{split}
\end{equation*}
\end{NB}
\end{Proposition}

\begin{NB}
I do not digest this propositon well yet.

I think that the statements are a little confusing. 

In (1) we should say $\xi(\bMm{0}(c)_i) = \bMm{0}(c^\perp)_{n+i}$.
(This is stronger than above.)

Also `In particular' is an application of \propref{prop:BN}.

In (3) why do not we just say
$\xi\colon \bMm{0}(c)_i \to \bMm{0}(c^\perp)_{n+i}$ is a Grassmann
bundle ? Do we really need $\bMm{0}(c-ie)$ in the statement ? In fact,
the Grassmann bundle structure is clear from the diagram below, as it
comes from $\bM(c_\perp,n)\xrightarrow{q_1} \bMm{0}(c_\perp)$.

And I do not understand why (3) follows from (1),(2) and the universal
extension.
\end{NB}

\begin{proof}%
\begin{NB}
I do not understand Kota's claim properly. I think all statements are
simple consequences of the diagram. So I comment out them.

We prove (1) and (2).  (3) follows from (1), (2) and the universal
extension in \propref{prop:pervblowup}(2).
\end{NB}
We consider the following diagram:
\begin{equation*}
\xymatrix@R=.5pc{
& \bM(c_\perp,n) \ar[ld]_{q_1} \ar[rd]^{q_2} &
\\
\bMm{0}(c_\perp) \ar[rd]_\cong & & \bMm{0}(c) \ar[ld]^\xi
\\
& M^X(p_*(c)) &
}
\end{equation*} 
By \lemref{lem:diagram} $q_2$ is the Grassmann bundle of $n$-planes in
a vector bundle of rank $(c_1,[C]) = n$. Therefore $q_2$ is an
isomorphism. Therefore the image of $\xi$ is identified with the image
of $q_1$. Hence (1) follows from \propref{prop:BN}(1).

Moreover $q_1$ is an immersion over $q_1^{-1}(\bMm{0}(c_\perp)_n)$ by
\propref{prop:BN}(1). Via the isomorphism $q_2$ it is identified with
$\bMm{0}(c)_0$. Hence we get (2).

(3) is also proved in a similar way. We consider $\bM(c_\perp,n+i)$
and the diagram
\begin{equation*}
  \xymatrix@R=.5pc{
& \bM(c_\perp,n+i) \ar[ld]_{q_1} \ar[rd]^{q_2} &
\\
\bMm{0}(c_\perp) & & \bMm{0}(c-ie)
}
\end{equation*}
Then $q_2$ is again an isomorphism in this case also, and we have 
$\bMz(c_\perp)_{n+i}\cong \bM(c_\perp,n+i)_0 \cong \bMm{0}(c-ie)_0$.

\begin{NB}
$\Ext^1(E'',\shfO_C(-1))$ for $E''\in
\bMz(c-ie)$ is of dimension $(c_1,[C]) + i = n+i$. Therefore $q_2$ is
an isomorphism. Also $q_1$ is an isomprhism on 
$q_1^{-1}(\bMz(c_\perp)_{n+i})$ by \propref{prop:BN}(2).

Explicitly the correspondence is given $E_\perp\leftrightarrow E''$
related by
\begin{equation*}
   0 \to W\otimes \shfO_C(-1) \to E_\perp \to E'' \to 0,
\end{equation*}
where $W$ is the $(n+i)$-dimensional vector space naturally isomorphic
to $\Hom(\shfO_C(-1),E_\perp)$ and $\Ext^1(E'', \shfO_C(-1))^\vee$. We also
have natural isomorphisms
\begin{equation*}
   \operatorname{Gr}(i,\Ext^1(E'',\shfO_C(-1)))
   \cong \operatorname{Gr}(i,\Hom(\shfO_C(-1),E_\perp)^\vee)
   \cong \operatorname{Gr}(n,\Hom(\shfO_C(-1),E_\perp)),
\end{equation*}
where the last one is $V\subset\Hom(\shfO_C(-1),E_\perp)^\vee
\leftrightarrow V^\perp\subset\Hom(\shfO_C(-1),E_\perp)$.
\end{NB}
\end{proof}

Let us constract the contraction in the other side of the wall.
We consider the diagram with a yet undefined morphism $\xi^+$:
\begin{equation*}
\xymatrix@R=.5pc{
& \bN(c_\perp e^{[C]},n') \ar[ld]!UR_(.7){q_1'} \ar[rd]^{q_2'} &
\\
\bMm{0}(c_\perp) \cong \bMm{1}(c_\perp e^{[C]}) &
& \ar@{.>}[ld]^{\xi^+}  \bMm{1}(c) 
\\
& \ar@{<-}[lu]!DR^(.7)\cong M^X(p_*(c)) &
}
\end{equation*} 
where $n' = (c_1,[C]) + r$, which is equal to the rank of the vector
bundle $\Ext^1_f(\shfO_C(-1),\mathcal E')$ over $\bMm{1}(c)$.%
\begin{NB}
We have
\(
  e^{[C]} - 1 = \ch (\shfO_{\bX}(C)) - \ch(\shfO_{\bX}) = e,
\)
hence
\begin{equation*}
  c_\perp e^{[C]} - n'e = (c + (c_1,[C]) e)\cdot (1 + e) - n'e
  = c + (c_1,[C]) e + r e + (c_1,[C])[\mathrm{pt}] 
  - (c_1,[C])[\mathrm{pt}] - n' e = c.
\end{equation*}
\end{NB}
Therefore $q_2'$ is an isomorphism. Hence we can define $\xi^+$ so
that the diagram commutes.

\begin{Proposition}\label{prop:open'}
Suppose $n' := (c_1,[C]) + r \ge 0$.

\textup{(1)} $\xi^+(\bMm{1}(c))$ is identified with the Brill-Noether
locus $\bMm{0}(c_\perp)^{\ge n'}$ via the above isomorphism. 
In particular, $\xi^+(\bMm{1}(c))$ is a Cohen-Macauley and normal
subscheme of $M^X(p_*(c))$, provided $\bMz(c_\perp)$, $\bMm{1}(c)$ are
irreducible and of expected dimension, and $\bMm{1}(c)$ is normal.
\begin{NB}
  I put the generic smoothness assumption.
\end{NB}

\textup{(2)} $\xi^+$ is an immersion on $\bMm{1}(c)^0$.

\textup{(3)} Each Brill-Noether stratum $\bMz(c_\perp)^{n'+i}$ is
isomorphic to $\bMm{1}(c-ie)^0$, so $\bMz(c_\perp)^{\ge n'}$ can be
considered as a scheme structure on $\bigsqcup_{i}\bMm{1}(c-ie)^0$
requested in \eqref{eq:set-diagram}.

\textup{(4)} $\xi^+$ maps $\bMm{1}(c)^i$ to
$\bMz(c_\perp)^{n'+i}$, and it can be identified with the Grassmann
bundle $\bMm{1}(c)^i\to \bMm{1}(c-ie)^0$ in \eqref{eq:set-diagram} under
the isomorphism in \textup{(3)}.
\end{Proposition}

\begin{NB}
  If $0 > n := (c_1,[C]) \ge -r$ ($n' := (c_1,[C]) + r \ge 0$), we
  cannot use \propref{prop:open}, but can use
  \propref{prop:open'}. This is precisely the case when
  $\bMz(c) = \emptyset$, but $\bMm{1}(c)\neq \emptyset$.
\end{NB}

The proof is the same as one for \propref{prop:open}, as we have the
commutative diagram. We just describe how $E\in\bMm{1}(c)$ is mapped
under the diagram:
\begin{equation*}
\xymatrix@R=.5pc{
& (0\to E \to E'\to U^\vee\otimes\shfO_C(-1)\to 0)
\ar@<-.5ex>@{|->}[ld]!R \ar@{|->}[rd]_\cong &
\\
E'(-C) \longleftrightarrow E' &
& \ar@{|->}[ld]!R E
\\
& \ar@<.5ex>@{<-|}[ul]!R^\cong p_*(E'(-C))&
}
\end{equation*} 

\begin{NB}
Note $p_*(E'(-C)) \neq p_*(E)$, which I have realised on Apr. 21.
\end{NB}

\begin{NB}
We consider $\bN(c_\perp e^{[C]},n'+i)$
and the diagram
\begin{equation*}
  \xymatrix@R=.5pc{
& \bN(c_\perp e^{[C]},n'+i) \ar[ld]_{q_1} \ar[rd]^{q_2} &
\\
\bMm{1}(c_\perp e^{[C]}) & & \bMm{1}(c-ie)
}
\end{equation*}
Then $q_2$ is again an isomorphism in this case also, and we have 
$\bMm{1}(c_\perp e^{[C]})^{n'+i}\cong 
\bN(c_\perp e^{[C]},n'+i)^0 \cong \bMm{1}(c-ie)^0$.

In fact, $\Ext^1(\shfO_C(-1),E'')$ for $E''\in
\bMm{1}(c-ie)$ is of dimension $(c_1,[C]) + i + r = n'+i$. Therefore $q_2$ is
an isomorphism. Also $q_1$ is an isomprhism on 
$q_1^{-1}(\bMm{1}(c_\perp e^{[C]})^{n+i})$ by the analog of
\propref{prop:BN}(2).

Explicitly the correspondence is given $E_\perp\leftrightarrow E''$
related by
\begin{equation*}
   0 \to E'' \to E_\perp \to W\otimes \shfO_C(-1) \to 0,
\end{equation*}
where $W$ is the $(n'+i)$-dimensional vector space naturally isomorphic
to $\Hom(E_\perp, \shfO_C(-1))^\vee$ and
$\Ext^1(\shfO_C(-1),E'')$. Then we have natural isomorphisms
\begin{equation*}
   \operatorname{Gr}(i,\Ext^1(\shfO_C(-1), E''))
   \cong \operatorname{Gr}(i,\Hom(E_\perp, \shfO_C(-1))^\vee)
   \cong \operatorname{Gr}(n',\Hom(E_\perp, \shfO_C(-1))),
\end{equation*}
where the last one is $V\subset\Hom(E_\perp, \shfO_C(-1))^\vee
\leftrightarrow V^\perp\subset\Hom(E_\perp,\shfO_C(-1))$.
\end{NB}

We finally need to show that the targets of $\xi$ and $\xi^+$ are the
same.

\begin{Proposition}
  Suppose that $(c_1(c_\perp),[C]) = 0$. Then 
\(
  \bMz(c_\perp)^{\ge n+r} = \bMz(c_\perp)_{\ge n}.
\)
In fact, the both Brill-Noether loci are identified with
\[
   \{ F\in M^X(p_*(c_\perp)) \mid
   \dim \Hom(F, \C_0) \ge n + r \}
\]
under the isomorphism $\bMz(c_\perp)\cong M^X(p_*(c_\perp))$.
\end{Proposition}

\begin{proof}
Let $\mathcal F$ be a universal family over $X\times
M^X(p_*(c_\perp))$. Let $0\to \mathcal V\to \mathcal W\to \mathcal
F\to 0$ be a locally free resolution.

Then $\bMz(c_\perp)_{\ge n}$ is defined by the zero locus
of $\Wedge^{-\chi(\shfO_C(-1), p^*(\mathcal V)) + 1 - n}\rho$, where
\begin{equation*}
\xymatrix@R=.8pc{
   & 0 \ar[r] & \Hom_f(\shfO_C(-1), p^*(\mathcal F)) \ar[lld] &
\\
   \Ext^1_f(\shfO_C(-1), p^*(\mathcal V)) \ar[r]_\rho & 
   \Ext^1_f(\shfO_C(-1), p^*(\mathcal W)) \ar[r] &
   \Ext^1_f(\shfO_C(-1), p^*(\mathcal F)) \ar[r] & 0.
}
\end{equation*}
On the other hand, $\bMz(c_\perp)^{\ge n+r}$ is defined by the zero
locus of $\Wedge^{-\chi(p^*(\mathcal W), \shfO_C) + 1 - n - r} \rho'$, where
\begin{equation*}
\xymatrix@R=.8pc{
   0 \ar[r] & \Hom_f(p^*(\mathcal F), \shfO_C) \ar[r] & 
     \Hom_f(p^*(\mathcal W),\shfO_C) \ar[r]^{\rho'} &
     \Hom_f(p^*(\mathcal V),\shfO_C) \ar[lld] &
\\
  & \Ext^1_f(p^*(\mathcal F),\shfO_C) \ar[r] & 0 &&
}
\end{equation*}
The transpose of $\rho$ is given by
\begin{equation*}
   \Ext^1_f(p^*(\mathcal W),\shfO_C(-2))
   \to \Ext^1_f(p^*(\mathcal V),\shfO_C(-2)),
\end{equation*}
which is naturally isomorphic to $\rho'$.
\begin{NB}
Consider the exact sequence
$0\to \shfO_C(-2) \to \shfO_C(-1)^{\oplus 2} \to \shfO_C \to 0$.
We apply $\bR\Hom(E, \bullet)$ to get
\begin{equation*}
\xymatrix@R=.8pc{
  0 \ar[r] & \Hom(E, \shfO_C(-2)) \ar[r] & \Hom(E, \shfO_C(-1))^{\oplus 2}
  \ar[r]   & \Hom(E, \shfO_C) \ar[lld] &
\\
           & \Ext^1(E, \shfO_C(-2)) \ar[r] & \Ext^1(E,
           \shfO_C(-1))^{\oplus 2} \ar[r] & \Ext^1(E, \shfO_C)
            \ar[lld] &
 \\
            & \Ext^2(E, \shfO_C(-2)) \ar[r] & \Ext^2(E,
            \shfO_C(-1))^{\oplus 2} \ar[r] & \Ext^2(E, \shfO_C)
            \ar[r] & 0.
}
\end{equation*}
If $E = p^*(F)$ for $F\in\Coh(X)$, then
$\Hom(p^*F, \shfO_C(-1)) = \Hom(F, p_*(\shfO_C(-1))) = 0$,
and $\Ext^1(p^*F, \shfO_C(-1)) = 0$ by
\lemref{lem:Bridgeland}(4). And $\Ext^2(p^*F, \shfO_C(-1)) = 0$.
Therefore 
$\Ext^i(p^*F, \shfO_C) \cong \Ext^{i+1}(p^*F, \shfO_C(-2))$ for $i=0,1$.
\end{NB}%
Moreover, the projection formula shows that $\rho'$ is equal to
\begin{equation*}
  \Hom_f(\mathcal W, \C_0) \to \Hom_f(\mathcal V, \C_0),
\end{equation*}
which implies the isomorphisms among Brill-Noether loci as in the
assertion.
\begin{NB}
\begin{equation*}
\begin{split}
    \dim \Hom(p^*F, \shfO_C) \ge n + r
    & \Longleftrightarrow
    \dim \Ext^1(p^*F, \shfO_C(-2)) \ge n + r
\\
    & \Longleftrightarrow
    \dim \Ext^1(\shfO_C(-1), p^*F) \ge n + r
\\
    & \Longleftrightarrow
    \dim \Hom(\shfO_C(-1), p^*F) \ge n.
\end{split}
\end{equation*}
\end{NB}%
\end{proof}

\subsection{Ample line bundles on moduli spaces}

If both $\bMm{m}(c)$ and $\bMm{m+1}(c)$ would be GIT quotients of a
{\it common\/} variety for the stability conditions separated by a
single wall, they are flip provided $\xi_m\colon \bMm{m}(c)\to
\bMm{m,m+1}(c)$ would be a small contraction (\cite{Th}). As we do not
know how to construct this picture in our setting, we prove this
statement {\it directly}. Moreover the smallness condition is related
to the dimension of the moduli spaces, and hence we do not expect such
a result unless we assume $(\shfO_X(1), K_X) < 0$ or $c_2$ is
sufficiently large. Instead of assuming these kinds of conditions, we
produce a line bundle which is relatively ample on $\bMm{m}(c)$, but
not on $\bMm{m+1}(c)$, where we consider the spaces relative to the
Uhlenbeck compactification $M^X_0(p_*(c))$.

We continue to assume
$\operatorname{gcd}(r,\linebreak[3](c_1,\linebreak[3]p^*\shfO_X(1)))=1$.
For $d \in K(X)$ with $\rk(d)=r$ and $c_1(d)=c_1(p_*(c))$, there is a
class $\alpha_d \in K(X)$ such that $\rk \alpha_d=0$ and $\chi(d
\otimes \alpha_d)=1$.

Let $p_X$, $p_{M}$ be the projections from
$X\times M^X(d)$ to the first and second factors respectively.
If we twist a universal bundle $\mathcal E$ by a line
bundle $L$ over the moduli space $M^X(d)$, we have
\(
   \det p_{M!}(\mathcal E\otimes p_{M}^*L \otimes p_X^*\alpha)
   = \det p_{M!}(\mathcal E\otimes p_X^*\alpha)
   \otimes L^{\otimes \chi(d\otimes \alpha)}
\)
for $\alpha\in K(X)$.
Therefore for $\alpha = \alpha_d$ we can normalize a universal family
${\cal E}_d$ on $X \times M^X(d)$ 
so that $\det p_{M!}({\cal E}_d \otimes p_X^*\alpha)={\cal O}_{M^X(d)}$.

\begin{Lemma}
Let $\beta \in K(X)$ be a class with $\rk \beta=-1$.
Then $\det p_{M!}({\cal E}_d \otimes p_X^*\beta)$ is relatively
ample over $M_0^X(d)$.
\end{Lemma}

\begin{Remark}
For $\beta':=\beta-\chi(d \otimes \beta)\alpha_d$, we have
$\chi(d\otimes \beta')=0$, which means that
$\det p_{M!}({\cal E} \otimes \beta')$ does not depend on the
choice of the universal family ${\cal E}$.
\end{Remark}

\begin{proof}
By Simpson's construction of the moduli space,
$N_{\cal E}:=\det p_{M!}({\cal E}(n+m))^{\chi(d(n))} \otimes
 \det p_{M!}({\cal E}(m))^{-\chi(d(n+m))}$ is ample
for $n \gg m \gg 0$, where ${\cal E}$ is a universal family.
Since $N$ does not depend on the choice of the universal family,
we may assume that $N_{\cal E}=N_{{\cal E}_d}$.
We set $\gamma:=\chi(d(n)){\cal O}_X(n+m)-\chi(d(n+m)){\cal O}_X(m)$.
Then $\rk \gamma<0$ and
$\beta \in {\Bbb Q}_{>0} \gamma+{\Bbb Q}h+{\Bbb Q}\alpha_d$,
where $h \in K(X)$ is a class such that
$\det p_{M!}({\cal E} \otimes h)$
descends to a determinant line bundle on $M^X_0(d)$.
\begin{NB}
i.e.\ it is the pull-back of $\mu(c_1(h))$.
\end{NB}%
(See \cite[\S8.2]{HL}.)
Therefore $\det p_{M!}({\cal E}_d \otimes \beta)$ is relatively
ample over $M_0^X(d)$.
\end{proof}

Suppose $d = p_*(c)$ and take $\beta$ with $\rk \beta=-1$ as above,
and we normalize the universal family as above.
\begin{Proposition}\label{prop:ample}
We set
$L_t:=
\det p_{\bMm{}!}\left({\cal E} \otimes p_X^* (\beta 
+ t \shfO_C(-1))\right)
\begin{NB}
=
\det p_{\widehat{M}^m(c)!}({\cal E} \otimes \beta)
\otimes {\cal O}_{\widehat{M}^m(c)}(-t \mu(C))
\end{NB}
.\)

\textup{(1)} If $m-1 < t < m$, then $L_t$ is relatively ample over
$M^X_0(d)$.

\textup{(2)}
Assume that $\widehat{M}^m(c) \ne \widehat{M}^{m+1}(c)$.
Then $L_t$ is not relatively ample over $M_0^X(d)$ for $t\ge m$.

\textup{(3)}
Assume that $\widehat{M}^m(c) \ne \widehat{M}^{m-1}(c)$.
Then $L_t$ is not relatively ample over $M_0^X(d)$ for $t\le m-1$.
\end{Proposition}

\begin{proof}
(1) 
We note that
${\cal O}_{\widehat{X}}(-mC)={\cal O}_{\widehat{X}}-m{\cal O}_C(-1)-
\frac{m(m+1)}{2}{\Bbb C}_p$, where $p$ is a point in $C$.
Hence
$c \otimes (\beta+m{\cal O}_C(-1)))=
c(-mC) \otimes (\beta-\frac{m(m+1)}{2}{\Bbb C}_p)$.
Since $p_*(E(-mC)) \in M^X(d-r\frac{m(m+1)}{2}{\Bbb C}_0)$
for $E \in \widehat{M}^m(c)$,
$L_m$ is the pull-back of a relatively ample
line bundle on $M^X(d-r\frac{m(m+1)}{2}{\Bbb C}_0)$ by the previous
lemma.
In the same way, we see that $L_{m-1}$ also the pull-back of
a relatively ample
line bundle on $M^X(d-r\frac{m(m-1)}{2}{\Bbb C}_0)$.
Since $\widehat{M}^m(c) \to M^X(d-r\frac{m(m+1)}{2}{\Bbb C}_0)
\times M^X(d-r\frac{m(m-1)}{2}{\Bbb C}_0)$ is an embedding,
$aL_m+b L_{m-1}$ is relatively ample for $a,b>0$.

(2) By $\bMm{m}(c)\to M^X(d - r\frac{m(m+1)}{2}{\Bbb C}_0)$ the
Grassmann bundle structures of the Brill-Noether loci $\bMm{m}(c)_i$
are contracted. From the assumption, the Brill-Noether locus
$\bMm{m}(c)_{\ge 1}$ is nonempty, so $L_m$ is not relatively ample.
The proof of (3) is the same.
\end{proof}

\begin{NB}
  I do not understand why you wrote this. This is just a product of
$\bMm{\infty}(c)\to \bM_0(c)$ and
$\bMm{\infty}(c)\to M_0^X(p_*(c))$ ?

\begin{Remark}
$\det p_{M !}(\mathcal E\otimes \shfO_C(-1))
\begin{NB2}
 ={\cal O}_{\widehat{M}^{\infty}(c)}(-\mu(C))
\end{NB2}%
$
gives the morphism to the Uhlenbeck
compactification $\widehat{M}_0(c)$ over $M_0^X(d)$:
thus we have a morphism
$\widehat{M}^{\infty}(c) \to \widehat{M}_0(c) \times M_0^X(d)$.
\end{Remark}
\end{NB}

This completes our construction of the diagram \eqref{eq:flip} in the
introduction.

\subsection{Another distinguished chamber -- torsion free sheaves on blow-up}

\begin{Proposition}[\protect{cf. \cite[Prop.~7.1]{perv}}]\label{prop:blowup}
  Fix $c\in H^*(\bX)$. There exists $m_0$ such that if $m \ge m_0$,
  $\bMm{m}(c)$ is the moduli space of $(p^*H-\varepsilon C)$-stable
  torsion free sheaves on $\bX$ for sufficiently small $\varepsilon >
  0$.

  If $(\shfO_X(1),K_X) < 0$, then we can take
  \begin{equation*}
    m_0 = -(c_1,[C]) + r\Delta - \frac12 (r+1 + (r^2 - 1)\chi(\shfO_X)
    - h^1(\shfO_X)) + 1.
  \end{equation*}
\end{Proposition}

\begin{proof}
Consider the projective morphism $\widehat \pi\colon
\bMm{m}(c)\to M_0^X(p_*(c))$ in \eqref{eq:hatpi}, where $M_0^X(p_*(c))$
is the Uhlenbeck 
compactification on $X$. From \propref{prop:mult}, there exists $m_0$ such
that if $m\ge m_0$ and $E\in\bMm{m}(c)$, we have $E\in\bMm{m+1}(c)$,
i.e.\ $\bMm{m}(c)\cong \bMm{m+1}(c)\cong\bMm{m+2}(c)\cong \cdots$.
If $(\shfO_X(1),K_X) < 0$, then $m_0$ can be explicitly given by
\lemref{lem:BN-empty}.

Suppose that $E$ is torsion free and $(p^* H - \ve C)$-stable. Then
$p_*(E(-mC))$ is $\mu$-stable for sufficiently large $m$.
The torsion freeness implies
\(
  \Hom(\shfO_C(-mC), E) = 0
\)
for any $m$. On the other hand, we have
\(
 \Hom(E(-mC), \shfO_C)
\)
is zero for $m\gg 0$, as $\shfO_{\bX}(-C)$ is relatively ample with
respect to $p\colon \bX\to X$.
\begin{NB}
We have $(E_{|C})/\text{torsion}=\bigoplus_i O_C(a_i)$. Therefore
$\Hom(E(-mC),O_C(-1))=0$ if $m+a_i \geq 0$ for all $i$.
\end{NB}%
Therefore $E$ is $m$-stable for sufficiently large $m$. From the above
discussion, $E$ is $m_0$-stable.

Conversely suppose that $E$ is $m_0$-stable. Then $E$ is $m$-stable
for any $m\ge m_0$. In particular, $\Hom(\shfO_C,E(-mC)) = 0$ for
$m\ge m_0$.
Suppose that $E$ is not torsion free, and let $0\ne T\subset E$ be its
torsion part. Then as $\shfO_{\bX}(-C)$ is relatively ample, we have
$p_*(T(-mC))\ne 0$ for $m\gg 0$. Since $p_*(T(-mC))$ is supported at
$0$, we have
\begin{equation*}
   0 \ne \Hom(\C_0, p_*(T(-mC))) = \Hom(\shfO_C, T(-mC))
   \subset \Hom(\shfO_C, E(-mC)).
\end{equation*}
This is a contradiction. Therefore $E$ is torsion free. Since
$p_*(E(-mC))$ is $\mu$-stable for any $m\ge m_0$, $E$ is $(p^*H - \ve
C)$-stable for sufficiently small $\ve$.
\end{proof}

\subsection{The distinguished chamber -- revisited}

In this subsection we assume $0\ge (c_1,[C]) > - r$ and study moduli
spaces $\bMm{1}(c)$ under this assumption.
We can twist sheaves by a line bundle $\shfO(C)$, and this condition
is satisfied. But it also changes the stability condition, so studying
only $\bMm{1}(c)$ means that we are choosing a certain chamber.

The case $(c_1,[C]) = 0$ was already discussed in
\subsecref{subsec:discham}. (Strictly speaking we studied $\bMm{0}(c)$.)
So we consider the case $0 > (c_1,[C]) > - r$.

\begin{Proposition}
Suppose $0 < n := -(c_1,[C]) < r$. We have a diagram
\begin{equation*}
\xymatrix@R=.5pc{
& \bN(c,n) \ar[ld]_{q'_1} \ar[rd]^{q'_2} &
\\
\bMm{1}(c) & & \bMm{1}(c-ne)
}
\end{equation*}
such that 
\textup{(i)} $\bMm{1}(c) = \bMm{1}(c)^{\ge n}$,
\textup{(ii)} $q'_1$ is surjective and isomorphism over the
open subscheme $\bMm{1}(c)^n$, 
\textup{(iii)} $q'_2$ is a $\operatorname{Gr}(n,r)$-bundle.

If $\bMm{1}(c)$, $\bMm{1}(c-ne)$ are irreducible
and of expected dimension, then $q_1'$ is birational.
\end{Proposition}

We have $(c_1(c-ne),[C]) = (c_1,[C]) + n = 0$. Therefore
$\bMm{1}(c-ne)$ becomes $M^X(p_*(c))$ after crossing a single wall.

Twisting by the line bundle $\shfO(C)$, we have an isomorphism
$\bMz(c') \cong \bMm{1}(c' e^{[C]})$. Since $(c' e^{[C]},[C]) = 0$,
$\bMz(c')$ becomes isomorphic to $M^X(p_*(c' e^{[C]}))$ after crossing
a single wall.

\begin{proof}
  Let $E\in\bMm{1}(c)$. We have $\chi(E,\shfO_C(-1)) = - (c_1,[C]) = n
  > 0$ by our assumption.  As $\Ext^2(E,\shfO_C(-1)) =
  \Hom(\shfO_C,E)^\vee = 0$ by the stability of $E$, we have $\dim
  \Hom(E,\shfO_C(-1)) \ge n$. This shows (i).

  We consider $q_1'\colon \bN(c,n)\to \bMm{1}(c)$ as in
  \eqref{eq:diagram'}. From the above observation, it is surjective.
  Moreover it is an isomorphism over $\bMm{1}(c)^n$. (ii) follows.

  We have $q_2'\colon \bN(c,n) \to \bMz(c-n\ch(\shfO_C))$. By
  \lemref{lem:cosystem} it is the Grassmann bundle
  $\operatorname{Gr}(n,\Ext^1_f(\shfO_C,\mathcal E'))$ of
  $n$-dimensional subspaces in $\Ext^1_f(\shfO_C,\mathcal E')$ over
  $\bMz(c-n\ch(\shfO_C))$, which is of rank $(c_1(E'),[C]) = (c_1,[C]) +
  n = r$. Therefore we have (iii).
\end{proof}

\section{Moduli spaces as incidence varieties}\label{sec:incidence}

Recall that we have a morphism
\begin{equation*}
\begin{matrix}
   \xi\times\eta \colon & \bMm{0}(c) & \to & 
   M^X(p_*(c))\times M^X(p_*(c)+n\pt)
\\
  & E & \mapsto & (p_*(E), p_*(E(C))),
\end{matrix}
\end{equation*}
where $n = (c_1,[C])$. (See \subsecref{subsec:discham}.)

The purpose of this section is to prove the following:

\begin{Theorem}\label{thm:incidence}
The morphism $\xi\times\eta$ identifies $\bMm{0}(c)$ with the
incidence variety $L(p_*(c)+n\pt,n)$ with $n = (c_1,[C])$, where
\begin{equation*}
   L(c',n) := 
   \{(F,U)\mid F \in M^X(c'), U \subset \Hom(F,\C_0), \dim U=n \}
\end{equation*}
for $c'\in H^*(X)$.
\end{Theorem}

\begin{Remark}\label{rem:nested}
  If $c'=1 - N\pt$, $n = 1$, then
  $L(c',1) = \allowbreak \{ (I, U) \mid \allowbreak I\in \HilbX{N},
  U\subset \Hom(I,\C_0), \allowbreak \dim U = 1\}
  \subset \HilbX{N+1}\times\HilbX{N}$
\begin{NB}
  $\cong \{ (I,I')\in \HilbX{N}\times \HilbX{N+1} \mid I'
  \subset I,\allowbreak I/I' \cong \C_0 \}$
\end{NB}%
is called the {\it nested Hilbert scheme\/}, and has been
  studied by various people. Here $\HilbX{N}$ is the Hilbert scheme
  of $N$ points in $X$.
\end{Remark}

The variety $L(c',m)$ is the quotient of the moduli of framed sheaves
$(F,F \to \C_0^{\oplus m})$ by the action of $GL(m)$.
We have a projective morphism
$\sigma\colon L(c',m) \to M^X(c')$ by sending $(F,U)$ to $F$.
For $(F,U) \in L(c',m)$,
we set $F':=\Ker(F \to U^{\vee} \otimes \C_0)$. It is easy to see that
$F\to U^{\vee}\otimes\C_0$ is surjective.
\begin{NB}
Let us consider
\begin{gather*}
  0 \to F' \to F \xrightarrow{\operatorname{ev}} 
    \Ima(\operatorname{ev})\to 0,
\\
  0 \to \Ima(\operatorname{ev}) \to U^\vee\otimes \C_0
    \to \Coker(\operatorname{ev}) \to 0.
\end{gather*}
From the second exact sequence, we have
$\Coker(\operatorname{ev})\cong W\otimes\C_0$ for some vector space $W$.
We have the induced exact sequences
\begin{gather*}
  0\to \Hom(\Ima(\operatorname{ev}),\C_0) \to \Hom(F,\C_0),
\\
  0 \to \Hom(\Coker(\operatorname{ev}),\C_0) \to U 
    \to \Hom(\Ima(\operatorname{ev}),\C_0).
\end{gather*}
Since the composition
\(
   U \to \Hom(\Ima(\operatorname{ev}),\C_0) \to \Hom(F,\C_0)
\)
is the natural inclusion, the left arrow is injective. Therefore
$\Hom(\Coker(\operatorname{ev}),\C_0) = 0$. However as
$\Coker(\operatorname{ev}) \cong W\otimes \C_0$, this means
$W = 0$, i.e.\ $\Coker(\operatorname{ev}) = 0$. Therefore
$\operatorname{ev}$ is surjective.
\end{NB}%
Moreover $F'$ is a $\mu$-stable sheaf.
Thus we also have a morphism
$\varsigma\colon L(c',n) \to M^X(c'-n \pt)$
by sending $(F,U)$ to $F'$.
By the same argument as in the case of $\bM(c,n)$,
we have an isomorphism from $L(c',n)$ to the moduli space of `dual'
coherent system $(F', U^\vee\subset \Ext^1(\C_0,F'))$ with 
$F'\in M^X(c'-n\pt)$, $\dim U^\vee = n$.

Consider
\begin{equation*}
\begin{matrix}
\sigma \times \varsigma\colon & L(c',n) & \to&
M^X(c') \times M^X(c'-n\pt)\\
& (F,U) & \mapsto & (F,F').
\end{matrix}
\end{equation*}

\begin{Lemma}\label{lem:immersion}
The morphism $\sigma \times \varsigma$ is a closed immersion.
\end{Lemma}

For this purpose, it is sufficient to prove that
\begin{enumerate}
\item[(1)]
$\sigma \times \varsigma$ is injective and
\item[(2)]
$d(\sigma \times \varsigma)_*$ is injective.
\end{enumerate}
\begin{NB}
The above criterion is in Hartshorne [II.\ 7.3].

Indeed since $\sigma \times \varsigma$ is proper, (1) implies that
$\sigma \times \varsigma$ is finite. Then (2) implies that
$$
{\cal O}_{ M^X(c') \times M^X(c'-n\pt)} \to
(\sigma \times \varsigma)_*({\cal O}_{L(c',n)})
$$
is surjective.
\end{NB}

From the $\mu$-stability of $F$, $F'$ and $\mu(F) = \mu(F')$, the
following holds from a standard argument.
\begin{Lemma}
 $\Hom(F,F) \cong \Hom(F',F') \cong \Hom(F',F) \cong {\Bbb C}$.
\end{Lemma}

\begin{NB}
\begin{proof}
Let $\xi\colon F'\to F$ be a homomorphism. We have an induced
homomorphism $\xi^{\vee\vee}\colon (F')^{\vee\vee}\to F^{\vee\vee}$.
We know that the inclusion $\iota\colon F'\to F$ induces
an isomorphism $\iota^{\vee\vee}\colon (F')^{\vee\vee}\cong F^{\vee\vee}$. 
Since $F^{\vee\vee}$ is $\mu$-stable, we have
$\xi^{\vee\vee} = \lambda \iota^{\vee\vee}$ for some $\lambda\in\C$.
Replacing $\xi$ by $\xi - \lambda\iota$, we may assume
$\xi^{\vee\vee} = 0$. We have the commutative diagram
\begin{equation*}
  \begin{CD}
    F' @>\xi>> F
\\
   @VVV @VVV
\\
   (F')^{\vee\vee} @>>\xi^{\vee\vee}> F^{\vee\vee},
  \end{CD}
\end{equation*}
where the vertical arrows are injective, as both $F$, $F'$ are torsion
free. Since $\xi^{\vee\vee} = 0$, we have $\xi = 0$.
\end{proof}
\end{NB}

\begin{proof}[Proof of \textup{(1)}]
Assume that $(F_1,U_1), (F_2,U_2) \in L(c',m)$ satisfy
$F_1 \cong F_2$ and $F'_1 \cong F'_2$,
where $F'_\alpha :=\Ker(F_\alpha \to {U_\alpha}^{\vee} \otimes \C_0)$
for $\alpha=1,2$.
Since $\Hom(F_1',F_2') \cong \Hom(F_1',F_2) \cong \Hom(F_1,F_2)$ by
the previous lemma,
we have the following diagram:
\begin{equation*}
\begin{CD}
0 @>>> F'_1 @>>> F_1 @>>> U_1^{\vee}\otimes \C_0 @>>> 0\\
@. @VVV @VVV @VVV @.\\
0 @>>> F'_2 @>>> F_2 @>>> U_2^{\vee}\otimes \C_0 @>>> 0
\end{CD}
\end{equation*}
Hence $(F_1,U_1) \cong (F_2,U_2)$.
\begin{NB}
For the moduli $M(r,c_2)$
of framed sheaves $(F,\Phi)$ on ${\Bbb P}^2$,
we also have a similar result.
Indeed if $(F_1,\Phi) \cong (F_2,\Phi')$ and
$(F_1',\Phi) \cong (F_2',\Phi')$, then
we have a unique morphism
$(F'_1,\Phi) \to (F_2,\Phi')$.
Hence
we also have a commutative diagram:
\begin{equation*}
\begin{CD}
0 @>>> (F'_1,\Phi) @>>> (F_1,\Phi) @>>> U_1^{\vee}\otimes \C_0 @>>> 0\\
@. @VVV @VVV @VVV @.\\
0 @>>> (F'_2,\Phi') @>>> (F_2,\Phi') @>>> U_2^{\vee}\otimes \C_0 @>>> 0
\end{CD}
\end{equation*}
\end{NB}%
\end{proof}

\begin{proof}[Proof of \textup{(2)}]
\begin{NB}
Could you give me the reference for the deformation theory of coherent
systems ?

Do you have an intension that you do write $F \to  U^{\vee}  \otimes
\C_0$ instead of $F'$ ? In general, the morphism is not necessarily
surjective, two are different, of course. But they are the same in our
situation.
\end{NB}

The Zariski tangent space of $L(c',n)$ at $(F,U)$ is
$$
\Ext^1(F,F')/\End(U)
\begin{NB}
\Ext^1(F,F \to  U^{\vee}  \otimes \C_0)/\End(U)
\end{NB}
$$
and the obstruction for an infinitesimal lifting belongs
to
$$
\Ext^2(F,F')   
\begin{NB}
\Ext^2(F,F \to  U^{\vee}  \otimes \C_0)
\end{NB}
\cong \Hom(F',F \otimes K_Y)^{\vee},
$$
where $\End(U) \to \Ext^1(F,F')
\begin{NB}
\Ext^1(F,F \to  U^{\vee}  \otimes \C_0)  
\end{NB}
$ is the
homomorphism
defined by the following diagram:
\begin{equation*}
\begin{CD}
\End(U) @. @. @.\\
@VVV @. @. @.\\
 U^{\vee}  \otimes \Hom(F,\C_0) @>>> \Ext^1(F,F')
@>>> \Ext^1(F,F)
@>>> \Ext^1(F,  U^{\vee}  \otimes \C_0)
\end{CD}
\end{equation*}
\begin{NB}
For the case of moduli of framed sheaves,
the Zariski tangent space at $(F,\Phi,U)$ is
$$
\Ext^1(F,F \to
(U^{\vee}  \otimes \C_0 \oplus {\cal O}_{\linf}^{\oplus r}))/\End(U)
$$
and the obstruction for an infinitesimal lifting belongs
to
$$
\Ext^2(F,F \to  (U^{\vee}  \otimes \C_0 \oplus {\cal O}_{\linf}^{\oplus r}))
\cong \Hom(E(-\linf),F \otimes K_Y)^{\vee}.
$$
\end{NB}%
Therefore the assertion follows from the following lemma.
\end{proof}

\begin{Lemma}
$$
d(\sigma \times \varsigma)_*:
\Ext^1(F,F')/\End(U) \to
\Ext^1(F',F') \oplus \Ext^1(F,F)
$$
is injective.
\end{Lemma}

\begin{proof}
We have the following exact and commutative diagram:
\begin{equation*}
\begin{CD}
0 @. @.\\
@VVV @. @.\\
 U^{\vee}  \otimes \Hom(F,\C_0)/\End(U) @>>> \Ext^1(F,F')/\End(U)
 @>>> \Ext^1(F,F)
\\
@VVV @VVV @. \\
\Hom(F', U^{\vee}  \otimes \C_0) @>{\alpha}>> \Ext^1(F',F') @.
\end{CD}
\end{equation*}
Since $\Hom(F',F') \to \Hom(F',F)$ is isomorphic,
$\alpha$ is injective, which implies the assertion. 
\end{proof}

Obviously we have an isomorphism
\begin{equation*}
\begin{matrix}
L(c',n) & \xrightarrow{\cong} & N(p^*(c')e^{[C]},n)\\
(F,U) & \mapsto & (p^*(F)(C),U),
\end{matrix}
\end{equation*}
where $N(p^*(c')e^{[C]},n)$ is as in \defref{def:cosystem} and we have
used \propref{prop:blowdown}.

We also have a morphism
\begin{equation*}
  \begin{matrix}
  N(p^*(c')e^{[C]},n) & \xrightarrow{\cong} & 
  \bMm{0}(p^*(c')-n\ch({\cal O}_C)) \\
  (E(C),U) & \mapsto & E' := \Ker(E \to U^{\vee} \otimes {\cal O}_C),
  \end{matrix}
\end{equation*}
which is essentially $q_2'$ in \eqref{eq:diagram'}. As $\rk
\Ext^1(\shfO_C, E') = n$ by \lemref{lem:ext1}, this morphism is an
isomorphism by \lemref{lem:cosystem}.

As $p^*(c') - n\ch(\shfO_C) = c$,
\begin{NB}
Use $p_*(c) = c' - n p_*(\ch(\shfO_C))$ and $p_*(\shfO_C) = \C_0$.
\end{NB}%
this completes the proof of \thmref{thm:incidence}.

\begin{NB}
  I do not understand the following:  The fiber of
$L(c,n)\to M_H(c-n\pt)$ is $\operatorname{Gr}(n,\Ext^1(\C_0,
F'))$. How this can be a single point (or empty) ?
\end{NB}

\begin{Remark}
If $n \gg \dim M_H^p(p^*(c)-n {\cal O}_C)$, then
we also have an embedding
$M_H^p(p^*(c)-n {\cal O}_C) \to L(c,n) \to M_H(c-n \pt)$.
\end{Remark}

\begin{NB}
Here is Hiraku's argument in terms of the quiver description:

Recall that we have two morphisms
\begin{equation*}
   \bMz(r,k,n) \to M(r,n + k^2/(2r) - k/2), \qquad
   \bMz(r,k,n) \to M(r,n + k^2/(2r) + k/2)
\end{equation*}
\begin{NB2}
$\dim V_0 = n + k^2/(2r) - k/2$,
$\dim V_1 = n + k^2/(2r) + k/2$.
\end{NB2}
given by
\begin{equation*}
   E \mapsto p_*(E) \text{  and  } p_*(E(C)).
\end{equation*}
(See \subsecref{subsec:discham}.)
In terms of the ADHM data, they are given by
\begin{equation*}
   [(B_1,B_2,d,i,j)]\mapsto [(B_1',B_2',i',j')] = [(B_1d,B_2d,i,jd)]
   \text{  and  }[(dB_1,dB_2, di,j)].
\end{equation*}
(See \cite[\S7.2]{perv}.)
\begin{NB2}
Strictly speaking, I do not check the latter statement.
\end{NB2}

\begin{Proposition}
  $\bMz(r,k,n)$ is isomorphic to
  \begin{equation*}
    \{ (F_1,F_2)\in M(r,n')\times M(r,n'-k) \mid
  F_1 \supset F_2, F_1/F_2 \cong \C_0^{\oplus (-k)} \}
  \end{equation*}
  \(
     E \mapsto (F_1,F_2) = (p_*(E(C)),p_*(E)),
  \)
where $n' := n+k^2/(2r) + k/2$.
\end{Proposition}

\begin{proof}
\begin{NB2}
  I am not sure that the following is unnecessary. Probably we do not
  need it...

Added on Feb. 13:

  For $k < -1$, I do not check the proof in \cite{Na:1998} can be
  modified.
\end{NB2}%
It is known that both $\bMz(r,-1,n)$ and the variety in the statement
are smooth. For the first the assertion was proved in
\cite[Th.2.5]{perv}. For the second it can be proved as in
\cite[Th.5.7]{Na:1998} or probably in other papers. Therefore it is
enough to show that the map is a set-theoretical bijection.

Let $E$ be a perverse coherent sheaf with $c_1(E) = kC$.  Then we have
$\Ext^1(E,\shfO_C(-1)) \cong \C^{\oplus (-k)}$, $\Ext^1(\shfO_C,E)
\cong \C^{\oplus (-k)}$ and $\Ext^i(E,\shfO_C(-1)) = 0 =
\Ext^i(\shfO_C,E)$ for $i\neq 1$ from \cite[Lemma~1.4]{Kota}. We
consider the universal extensions
\begin{equation*}
   0 \to E \to E_1 \to \shfO_C^{\oplus (-k)}\to 0, \qquad
   0 \to \shfO_C(-1)^{\oplus (-k)} \to E_2 \to E\to 0.
\end{equation*}
Then $E_1$, $E_2$ are perverse coherent \cite[\S1.3]{Kota} and
$c_1(E_1) = 0 = c_1(E_2)$. Therefore we can write them as
$E_1 = p^* F_1$, $E_2 = p^* F_2$ for some torsion free sheaves
$F_1$, $F_2$ on $Y$. By taking the direct image of
$E_2 = p^* F_2\to p^*F_1 = E_1$, we get an exact sequence
$0\to F_2 \to F_1 \to \C_0^{\oplus (-k)}\to 0$.

Conversely suppose that we have an exact sequence
\begin{equation*}
   0 \to F_2 \to F_1\to \C_0^{\oplus (-k)} \to 0
\end{equation*}
on $Y$, where $F_2$, $F_1$ are torsion free. Then we have an exact
sequence
\begin{equation*}
   0 \to \shfO_C(-1)^{\oplus (-k)} \to p^*(F_2) \to p^*(F_1) 
   \to \shfO_C^{\oplus (-k)} \to 0.
\end{equation*}
\begin{NB2}
Let $0\to \Wedge^2\Omega_Y \to \Omega_Y \to \shfO_Y \to \C_0\to 0$ to
be the Koszul resolution. Then ${\mathbf L}^\bullet p^*(\C_0)$ can be
computed from $p^*(\Wedge^\bullet\Omega_Y)$, i.e.,
\begin{equation*}
     0 \to p^*\Wedge^2\Omega_Y \to p^* \Omega_Y \to p^* \shfO_Y 
     \to 0.
\end{equation*}
It is enough to assume $Y = \C^2$, $X = {\widehat \C}^2$. We have
\begin{equation*}
    p^*\Omega_Y = \shfO_X^{\oplus 2} \xrightarrow{\left[
        \begin{smallmatrix}
          x & y
        \end{smallmatrix}
        \right]}
    \shfO_X = p^*\shfO_Y,
\end{equation*}
where ${\widehat \C}^2 = \{ ((x,y),[z:w]) \mid xw = yz \}$. From this
we have
\begin{equation*}
   \Coker(p^*\Omega_Y \to p^*\shfO_Y) \cong \shfO_C, \quad
   \Ker(p^*\Omega_Y \to p^*\shfO_Y) \cong \shfO_X(C)
\end{equation*}
and $\shfO_X = p^*\Wedge^2\Omega_Y \to \Ker(p^*\Omega_Y \to
p^*\shfO_Y) \to \shfO_X(C)$
is the natural section. Therefore ${\mathbf L}^1 p^* (\C_0) = \shfO_C(-1)$.
\end{NB2}%
Let $E := \Coker(\shfO_C(-1)^{\oplus (-k)} \to p^*(F_2))$. Then
$p_*(E) \cong F_2$,
\begin{NB2}
(as $p_*\shfO_C(-1) = 0$)
\end{NB2}%
and $p^*p_*(E) \to E$ is surjective. Therefore by \cite[Remark after
Prop.1.3]{Kota} $E$ is perverse coherent. We have $c_1(E) = kC$.
We also have $p_*(E(C)) \cong F_1$ from the exact sequence
$0\to E\to p^*(F_1)\to \shfO_C^{\oplus (-k)} \to 0$.
\end{proof}
\end{NB}

\section{Betti numbers}\label{sec:Betti}

In this section, we prove the formula \eqref{eq:Betti} in the
introduction and its higher rank generalization.

\subsection{Framed moduli spaces}

We consider $p\colon \bp\to \proj^2$ the blow-up of the
projective plane at $0 = [1:0:0]$.
Let $\linf = \{ [0:z_1:z_2] \}$ and denote its inverse image
$p^{-1}(\linf)$ by the same notation $\linf$ for brevity.
Following \cite{perv} we consider the framed moduli space of framed
coherent sheaves $(E,\Phi)$ on $\bp = \widehat{\C}^2\cup\linf$ with
$\ch(E) = c$, where $E$ is assumed to be locally free
along $\linf$, the framing $\Phi$ is a trivialization $\Phi\colon
E|_{\linf}\xrightarrow{\cong}\shfO_{\linf}^{\oplus r}$ over $\linf$,
and finally $E$ satisfies
\begin{equation*}
   \Hom(E, \shfO_C(-m-1)) = 0, \qquad
   \Hom(\shfO_C(-m),E) = 0.
\end{equation*}
This space was written as $\bM_{\zeta}(r,k,n)$ in \cite{perv}, where
$r(c) = r$, $(c_1(c),[C]) = -k$, $\Delta(c) := \int_{\bp} c_2(c) -
(r-1)c_1(c)^2/(2r) = n$,
and the parameter $\zeta = (\zeta_0, \zeta_1)\in\R^2$ satisfying
$0 > m\zeta_0 + (m+1)\zeta_1 \gg -1$.
But we use the same notation $\bMm{m}(c)$ as in the ordinary moduli
space for brevity. We hope this does not make any confusion.
Also we set $X = \C^2$, $\bX = \widehat{\C}^2$.
This convention applies to the other moduli spaces: $M^X(c')$ denotes
the framed moduli space of torsion free sheaves on $\proj^2 =
\C^2\cup\linf$, $M^X_0(c')$ denotes the Uhlenbeck partial
compactification, i.e.\ $M^X_0(c') = M^X_{\text{\it lf}}(c') \sqcup
M^X_{\text{\it lf}}(c'+\pt)\times \C^2 \sqcup
M^X_{\text{\it lf}}(c'+2\pt)\times S^2(\C^2)\sqcup \cdots$, where
$M^X_{\text{\it lf}}(c')$ denotes the framed moduli space of {\it locally
  free\/} sheaves on $\proj^2$.

A modification of the construction of the moduli space in
\secref{sec:moduli} to the framed moduli space is standard, and is
omitted. Otherwise, we can use the quiver description in \cite{perv}
to construct the framed moduli space.
We also have a projective morphism
$\widehat\pi\colon \bM^m(c)\to M_0(p_*(c))$, where $M_0$ denote the
Uhlenbeck partial compactification of the framed moduli space on
$\proj^2$. (See \cite[Chapters~2, 3]{Lecture} or \cite[\S3]{NY2}.)

As is mentioned in the beginning of \secref{sec:wall-crossing}, we may
assume $m = 0$ for most purposes.

\subsection{Universality of the blow-up formula}

We consider the framed moduli spaces and ordinary moduli spaces of
$m$-stable sheaves simultaneously. So $p\colon \bX\to X$ be the blowup
of either a projective surface or $\C^2$ at the point $0$. We define a
stratification of $M^X(c)$, $\bMm{m}(c)$ as in \cite[F.4]{NY2}: Let
$\iota\colon X\setminus\{0\}\to X$ be the inclusion. We define
\begin{equation*}
\begin{split}
   & M^X(c)_k :=
   \{ E\in M^X(c) \mid \Delta(\iota_*(E|_{X\setminus \{0\}})) = \Delta(c) - k
   \},
\\
   & \bMm{m}(c)_k :=
   \{ E\in \bMm{m}(c) \mid \Delta(\iota_*(E|_{\bX\setminus C})) = \Delta(c) - k
   \},
\end{split}
\end{equation*}
where we identified $\bX\setminus C$ with $X\setminus\{0\}$.
Then \cite[Cor.~F.22]{NY2} shows that we have the following equalities
in the Grothendieck group of $\C$-varieties when $m=\infty$:
\begin{equation*}
\begin{split}
  & \sum_{c'} [M^X(c')] \q^{\Delta(c')}
  = \Bigl(\sum_{c'} [M^X(c')_0] \q^{\Delta(c')}\Bigr)
    \Bigl(\sum_n [\mathfrak Q(r,n)] \q^n\Bigr),
\\
  & \sum_{c} [\bMm{m}(c)] \q^{\Delta(c)}
  = \Bigl(\sum_{c'} [M^X(c')_0] \q^{\Delta(c')}\Bigr)
    \Bigl(\sum_n [\widehat{\mathfrak Q}^{m}(r, k, n)] \q^n\Bigr),
\end{split}
\end{equation*}
where $c'$, $c$ runs over all $H^*(X)$, $c\in H^*(\bX)$ with fixed
$r(c) = r(c') = r$ and $c_1(c) = c_1$, $c_1(c') = p^* c + k[C]$.
Here $\mathfrak Q(r,n)$, $\widehat{\mathfrak Q}^{m=\infty}(r,k,n)$ are
certain quot-schemes, which are independent of surfaces. Moreover they
are the {\it same\/} for framed moduli spaces and ordinary moduli
spaces.
These equalities in the Grothendieck group of $\C$-varieties imply the
corresponding equalities for virtual Hodge polynomials. If varieties
are smooth and projective (e.g.\ the rank $1$ case or
$(\shfO_X(1),K_X) < 0$), then virtual Hodge polynomials are equal to
Hodge polynomials.

From the proof of \cite[Cor.~F.22]{NY2}, the same result holds for
{\it finite\/} $m$.
\begin{NB}
In the definition of $\widehat{\mathfrak Q}^{m=\infty}(k,n)$ in
\cite[(F.5)]{NY2} we took the limit $l\to \infty$. We remain $l$
finite for our case.
\end{NB}%
In particular, in order to prove the formula \eqref{eq:Betti} or its
higher rank generalization, it is enough to prove it for framed moduli
spaces. So we only consider framed moduli spaces in the rest of this
section.

\subsection{A combinatorial description of fixed points}

We have an $(r+2)$-dimensional torus $\hT = T^{r}\times(\C^*)^2$
action on $\bMz(c)$, $M^X(p_*(c))$. The first factor $T^{r}$ acts by
the change of framing, and the second factor $(\C^*)^2$ acts via the
action on the base space $\proj^2$ given by
\begin{equation*}
   [z_0:z_1:z_2]\mapsto [z_0:t_1 z_1:t_2 z_2],
\end{equation*}%
and the induced action on $\bp$.
\begin{NB}
In the quiver description \cite[Theorem~1.5]{perv} it is given by
\begin{equation*}
   [(B_1,B_2,d,i,j)] \mapsto [(t_1 B_1, t_2 B_2, d, i e^{-1}, t_1 t_2
   e j)],
\end{equation*}
where $e$ is an element in the first torus $T^{r}$. 
\end{NB}

The purpose of this subsection is to classify the fixed points in
$\bMz(c)$. As in the case of $\bMm{m}(c)$ for $m\gg 0$ and $M(c)$
(\cite{part1} or \cite[\S3]{NY2}), a framed sheaf $(E,\Phi)$ is fixed
by the first factor $T^{r}$ if and only if it decomposes into a direct
sum $E = E_1\oplus\cdots\oplus E_r$ into rank $1$ sheaves. So we first
assume that the rank $r$ is $1$, and $\hT = \C^*\times (\C^*)^2$, but
the first factor $\C^*$ acts trivially. By \thmref{thm:incidence} we
have
\begin{equation*}
\begin{split}
   \bMz(c) \cong &\; L(p_*(c)+n\pt,n)
\\
   = &\; \{ (F,\Phi,U) \mid (F,\Phi) \in 
   M^X(p_*(c)+n\pt), U \subset \Hom(F,\C_0), \dim U=n \}
\end{split}
\end{equation*}
with $n = (c_1,[C])$,
and it is an incidence variety in $M^X(p_*(c)+n\pt)\times
M^X(p_*(c))$. As we are assuming that the rank is $1$, it is the
product $\HilbX{N+n}\times\HilbX{N}$ of Hilbert scheme of points in
$X = \C^2$, where $p_*(c) + n\pt = 1 - N\pt$. Also recall
$\HilbX{N}$ is the set of all ideals in the polynomial ring $\C[x,y]$
such that $\dim \C[x,y]/I = N$. So we have
\begin{equation*}
\begin{split}
   \bM^0(c) & \cong
   \{ (I', I) \in \HilbX{N+n}\times\HilbX{N}
   \mid I' \subset I \subset \C[x,y] \ \text{a flag of
     ideals},\; I/I' \cong \C^n \},
\end{split}
\end{equation*}
where $\C^n$ is the $n$-dimensional vector space with the trivial 
$\C[x,y]$-module structure.

As the isomorphism is $\hT$-equivariant, a fixed point is mapped to a
fixed point. The torus fixed points in $\HilbX{N}$ are monomial ideals
$I$ in $\C[x,y]$, and are in bijection to Young diagrams with $N$
boxes as in \cite[Chap.~5]{Lecture}. Moreover, the box at the
coordinate $(a,b)$ corresponds to the $1$-dimensional weight space $\C
x^a y^b\pmod I$ of weight $t_1^{-a} t_2^{-b}$.

Therefore the fixed points in $\bMz(c)$ correspond to pairs $(I,I')$
of monomial ideals such that $I/I'\cong \C^n$. Let $Y$ be the Young
diagram corresponding to 
$I'$. Its boxes correspond to weight spaces of $\C[x,y]/I'$. Then
$I/I'\subset \C[x,y]/I'$ is a direct sum of weight spaces, so
corresponds to a subset $S$ of boxes in $Y$. Moreover, as $I/I'$ must
be the trivial $\C[x,y]$-module, so it must be contained in 
\[
\Ker \left[
     \begin{smallmatrix}
       x \\ y
     \end{smallmatrix}
\right]  \colon
   \C[x,y]/I' \to \C^2\otimes_\C \C[x,y]/I'.
\]
Therefore $S$ must be consisting of removable boxes.
Here recall a box in a Young diagram $Y$ at the coordinate $(a,b)$ is
{\it removable} if there are no boxes above and right of $(a,b)$. In
terms of a monomial ideal $I'$ corresponding to $Y$, removable boxes
correspond to weight spaces contained in
\(
   \Ker \left[
     \begin{smallmatrix}
       x \\ y
     \end{smallmatrix}\right].
\)

Conversely if $(Y,S)$ is given, then we set $I$, $I'$ be the monomial
ideals corresponding to $Y\setminus S$, $Y$ respectively. Then
$I'\subset I$, and $I/I'$ is a trivial $\C[x,y]$-module.

For an arbitrary rank case we have $r$-tuples of such pairs
$(Y_\alpha,S_\alpha)$ corresponding to each factor $E_\alpha$
($\alpha=1,\dots, r$).

\begin{Lemma}\label{lem:fixed}
The torus fixed points in $\bMm{0}(c)$ are in bijection to $r$-tuples
of pairs $(Y_\alpha,S_\alpha)$ of a Young diagram $Y_\alpha$ and a set
$S_\alpha$ consisting of removable boxes such that
\(
   \sum_\alpha \# S_\alpha = (c_1,[C]),
\)
\(
   \sum_\alpha |Y_\alpha| = -\int_{\bX} \ch_2 + \frac12 (c_1,[C]).
\)
\end{Lemma}
\begin{NB}
In the rank $1$ case we have 
\(
  |Y| = N+n = -\int_{\proj^2} p_*(c - 1)\td_{\proj^2}
  = - \int_{\bp} (c - 1) \td_{\bp}
  = - \int_{\bp} (c - 1)(1 + \frac12 c_1(\bX))
  = - \int_{\bp} \ch_2 + (c_1,[C]).
\)
In a higher rank case, the formula follows from the linearlity of
$\ch_2$, $c_1$ with respect to the direct sum.

In the quiver description:

$k = \dim V_1 - \dim V_0$, 
$n = \frac12 (\dim V_0 + \dim V_1) - \frac{k^2}2$.
Hence $|Y| = \dim V_0 = n + \frac{k^2-k}2$.
\end{NB}

We mark a box in $S_\alpha$ and call it a {\it marked box}. (See
Figure~\ref{fig:Young1}.)

As fixed points are isolated, so the class of $\bMm{m}(c)$ in the
Grothendieck group of $\C$-varieties is a polynomial in the class of
$\C$. In particular, it is determined by its Poincar\'e polynomial.
Therefore we will discuss only on Poincar\'e polynomials hereafter.

\begin{NB}
\begin{proof}[Proof of \lemref{lem:fixed}]
In terms of quiver description:    

We may assume $r=1$.
At a fixed point $X = [(B_1,B_2,d,i,j)]\in \bMz$, the vector spaces
$V_0$, $V_1$ become $\hT$-modules, and we have weight space
decompositions $V_0 = \bigoplus V_0(a,b)$, $V_1 = \bigoplus V_1(a,b)$,
and data map weight spaces as
\begin{gather*}
   B_1\colon V_1(a,b) \to V_0(a+1,b), \quad
   B_2\colon V_1(a,b) \to V_0(a,b+1), 
\\
   d\colon V_0(a,b) \to V_1(a,b), \quad
   i\colon W\to V_0(0,0), \quad j\colon V_1(-1,-1)\to W.
\end{gather*}
As fixed points are mapped to fixed points, and all weight spaces are
$1$-dimensional at $M(1,n)^{\hT}$, we have $\dim V_0(a,b)$, $\dim
V_1(a,b)$ are all $1$-dimensional, if they are not $0$. 
We put $V_0(a,b)$ at the coordinate $(a,b)$ as in the case of
$M(1,n)$. Therefore we have a Young diagram $Y$. Recall that the
${}^0\zeta$-stability says that a pair of subspaces $T_0\subset V_0$,
$T_1\subset V_1$ with $B_\alpha(T_1)\subset T_0$, $d(T_0)\subset T_1$,
$\Ima i\subset T_0$ must be $T_0 = V_0$, $T_1 = V_1$. Then $V_1(a,b)$
can be nonzero only when $V_0(a,b)$ is nonzero. Moreover, if
$V_0(a+1,b)$ or $V_0(a,b+1)$ is nonzero, then a restriction of $B_1 d$
or $B_2 d$ is nonzero, and hence $V_1(a,b)$ {\it must\/} be nonzero.
Therefore the case $V_0(a,b) \neq 0$ and $V_1(a,b) = 0$ is possible
only if $V_0(a+1,b) = 0 = V_0(a,b+1)$, i.e., the box at $(a,b)$ is
removable.
Hence $S := \{ (a,b) \mid \text{$V_0(a,b) \neq 0$ and $V_1(a,b) = 0$}\}$
is a subset of all removable boxes in the Young diagram $Y$.
Conversely if a Young diagram $Y$ and a subset $S$ of removable boxes
of $Y$ is given, we define $V_0(a,b)$ according to $Y$, $V_1(a,b)$
according $Y\setminus S$.
\end{proof}
\end{NB}

\begin{figure}[htbp]
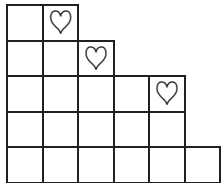

  \centering
\begin{equation*}
\Yvcentermath1
\young(\hf\hs,\hf\hf\hs,\hf\hf\hf\hf\hs,\hf\hf\hf\hf\hf,\hf\hf\hf\hf\hf\hf)
\end{equation*}
  \caption{Young diagram and marked removable boxes}
  \label{fig:Young1}
\end{figure}

\begin{NB}
Kota's manuscript:

By \cite[(1.13)]{Kota} we have a natural exact sequence
\begin{equation*}
   0 \to \shfO_C(-1)^{\oplus \langle c_1(E),C\rangle} \to
   p^* p_*(E) \to E \to 0.
\end{equation*}

\begin{Lemma}
Let $F$ be a torsion free sheaf on $Y$, and let $(x,y)$ be a local
coordinate around $0$. Then
\begin{equation*}
\begin{split}
& p_*({\cal H}om(\shfO_C(-1),p^*(F))) \cong \Tor_1^{\shfO_Y}(F,\C_0)
                               \cong \Tor_2^{\shfO_Y}(F^{**}/F,\C_0)
\\
  \cong \; & \Ker(F^{**}/F \xrightarrow{
    \left[
      \begin{smallmatrix}
        x \\ y
      \end{smallmatrix}\right]} (F^{**}/F)x \oplus (F^{**}/F)y).
\end{split}
\end{equation*}
\end{Lemma}

\begin{proof}
Take a locally free resolution of $F$:
\begin{equation*}
   0 \to V \to W \to F \to 0.
\end{equation*}
We have an exact sequence
\begin{equation*}
   0 \to p^*(V) \to p^*(W) \to p^*(F) \to 0.
\end{equation*}
\begin{NB2}
  Note ${\mathbf L}^i p^*(V) = 0 = {\mathbf L}^i p^*(W)$ for $i > 0$
  and $p^*(V)\to p^*(W)$ is injective. Therefore we have ${\mathbf
    L}^ip^*(F) = 0$ for $i > 0$ and the above exact sequence.
\end{NB2}%
Applying ${\mathbf R}{\cal H}om(\shfO_C(-1),\bullet)$ we get
\begin{equation*}
\begin{split}
{\cal H}om(\shfO_C(-1),p^*(F))=
\Ker({\cal E}xt^1(\shfO_C(-1),p^*(V)) \to {\cal E}xt^1(\shfO_C(-1),p^*(W))).
\end{split}
\end{equation*}
\begin{NB2}
This is because ${\cal H}om(\shfO_C(-1),p^*(W)) = 0$.
\end{NB2}%
As
\[
{\cal E}xt^1(\shfO_C(-1),p^*(V)) \cong p^*(V) \otimes \shfO_C,\quad
{\cal E}xt^1(\shfO_C(-1),p^*(W)) \cong p^*(W) \otimes \shfO_C,
\]
\begin{NB2}
This can be proved from \( 0 \to \shfO \to \shfO(C) \to \shfO_C(-1)\to 0\).
\end{NB2}%
we have
\begin{equation*}
  \begin{split}
    & p_*{\cal H}om(\shfO_C(-1),p^*(F))=
    \Ker(p_*(p^*(V) \otimes \shfO_C) \to p_*(p^*(W) \otimes \shfO_C) )
\\
= \; & \Ker(V \otimes \C_0 \to W \otimes \C_0 )=
   \Tor_1^{\shfO_Y}(F,\C_0).
  \end{split}
\end{equation*}
The isomorphism $\Tor_1^{\shfO_Y}(F,\C_0)\cong \Tor_2^{\shfO_Y}(F^{**}/F,\C_0)$
follows from the short exact sequence $0\to F \to F^{**} \to
F^{**}/F\to 0$. The last isomorphsim
\(
\Tor_2^{\shfO_Y}(F^{**}/F,\C_0)
  \cong \Ker(F^{**}/F \to (F^{**}/F)x \oplus (F^{**}/F)y)
\)
follows from the Koszul resolution
\(
  0\to \Wedge^2 \Omega_Y \to \Omega_Y \to \shfO_Y \to \C_0 \to 0.
\)
\begin{NB2}
And $\Tor_2^{\shfO_Y}(F^{**}/F,\C_0) \cong \Tor_2^{\shfO_Y}(\C_0,F^{**}/F)$.
\end{NB2}%
\end{proof}
\end{NB}

\begin{NB}
The following discussion can be omitted as we prove the combinatorial
bijection directly. I keep this for the record.  

It is not difficult to compute Betti numbers of $\bMz(r,k,n)$ from
this lemma, but I start to study Euler numbers for simplicity.

Let $A^n_{m,l}$ be the number of Young diagrams with $n$ boxes, $m$
removable boxes, and the bottom row of length $n-l$.
From the proof of \cite[Theorem~3.3.3(5)]{Cheah}, we have
\begin{equation}\label{eq:Cheah}
   \sum_{l,m,n} A^n_{m,l} u^m t^l \q^n
   = \prod_{d=0}^\infty
   \left[
     u\left(
       \frac1{1 - \q^{d+1} t^d} - 1
       \right)
       + 1
     \right].
\end{equation}
In fact, if we expand the right hand side as
\begin{equation*}
  \prod_{d=0}^\infty
   \left[
     u\left(
       \frac1{1 - \q^{d+1} t^d} - 1
       \right)
       + 1
     \right]
   =  \prod_{d=0}^\infty
   \left(
     1 + u\left(\q^{d+1}t^d + \q^{2(d+1)} t^{2d} + \q^{3(d+1)} t^{3d}
       + \cdots\right)
     \right),
\end{equation*}
we see that the coefficient of $u^m t^l \q^n$ is the number of
sequence $(m_0,m_1,\cdots)$ such that
\begin{aenume}
\item the number of nonzero entries is $m$,
\item $\sum (d+1)m_d = n$,
\item $\sum d m_d = l$.
\end{aenume}
And such a sequence is in bijection to a partition
$1^{m_0} 2^{m_1} 3^{m_2} \cdots$ of $n$ into $m$ parts with
$\sum m_d = n-l$. Hence this is equal to $A^n_{m,l}$.

From \eqref{eq:Cheah} we immediately get the generating function of
Euler numbers:
\begin{equation*}
\begin{split}
   & \sum_{k,n} e(\bMz(1,k,n)) u^{-k} \q^{n+\frac{k(k-1)}2}
   = \sum_{l,m,n} A^n_{m,l} (1+u)^m \q^n
\\ 
   =\; &
   \prod_{d=0}^\infty
   \left[
     (1+u)\left(
       \frac1{1 - \q^{d+1}} - 1
       \right)
       + 1
     \right]
   =
   \prod_{d=0}^\infty
   \left[
       \frac{1 + u \q^{d+1}}{1 - \q^{d+1}}
     \right].
\end{split}
\end{equation*}

Let us pick up the coefficient of $u^{-k}$. It is equal to
\begin{equation*}
  \left(\prod_{d=0}^\infty
         \frac{1}{1 - \q^{d+1}}
  \right) \times
  \sum_{0 < a_1 < a_2 < \cdots < a_{-k}} \q^{\sum a_i}
  =  \left(\prod_{d=0}^\infty
         \frac{1}{1 - \q^{d+1}}
  \right) \times
  \sum_{b_1 \ge b_2 \ge \cdots \ge b_{-k} \ge 0} \q^{k(k-1)/2 + \sum b_i},
\end{equation*}
where $b_i = a_{-k+1-i} - (-k+1-i)$.

Let ${}^m\zeta$ be a stability parameter satisfying ${}^m\zeta_0 < 0$,
$0 < - (m\cdot{}^m\zeta_0 + (m+1)\cdot{}^m\zeta_1) \ll 1$. Then 
\(
   \bM_{{}^m\zeta}(1,0,n) \cong \bMz(1,-m,n)
\)
by the morphism given by $E\mapsto E(-mC)$. Hence
\begin{equation*}
   \sum_n e(\bM_{{}^m\zeta}(1,0,n)) \q^n
   = \left(\prod_{d=0}^\infty
         \frac{1}{1 - \q^{d+1}}
  \right) \times
  \sum_{b_1 \ge b_2 \ge \cdots \ge b_{m} \ge 0} \q^{\sum b_i}
\end{equation*}
The summation is over all nonincreasing sequences of $m$ nonnegative
integers. Such a sequence corresponds to a partition of $\sum b_i$
into $m$ parts via
\begin{equation*}
  (\underbrace{p,\cdots, p}_{\text{$m_p$ times}},
  \cdots,
  \underbrace{2,\cdots, 2}_{\text{$m_2$ times}},
  \underbrace{1,\cdots ,1}_{\text{$m_1$ times}},
  \underbrace{0,\cdots ,0}_{\text{$m-\sum m_\alpha$ times}})
  \longleftrightarrow
  1^{m_1} 2^{m_2} \cdots.
\end{equation*}
It also corresponds to a Young diagram with at most $m$ columns.
By considering transpose, it corresponds to a Young diagram which has
at most $m$ rows. Then the summation over the set gives
\begin{equation*}
   \sum_n e(\bM_{{}^0\zeta}(1,-m,n)) \q^n
   = \sum_n e(\bM_{{}^m\zeta}(1,0,n)) \q^n
   = \left(\prod_{d=0}^\infty
         \frac{1}{1 - \q^{d+1}}
  \right)
  \left(\prod_{d=0}^{m-1}
         \frac{1}{1 - \q^{d+1}}
  \right).
\end{equation*}
In the right hand side the coefficient of $\q^n$ is equal to the
number of pairs of Young diagrams $(Y^1,Y^2)$ such that $n = |Y^1| +
|Y^2|$ and $Y^2$ has at most $m$ rows.

\begin{NB2}
Is it true that the fixed point set is in bijection to such a pair ?  
\end{NB2}

Let $m\to \infty$. Then we get
\begin{equation*}
   \sum_n e(\bM_{{}^\infty\zeta}(1,0,n)) \q^n
   = \left(\prod_{d=0}^\infty
         \frac{1}{1 - \q^{d+1}}
  \right)^2.
\end{equation*}
This is consistent with the formula for Euler numbers of Hilbert
schemes of points on $\widehat{\C}^2$.

This formula can be extended to higher rank cases.
Let $(\q;\q)_n$ be $(1-\q)(1-\q^2)\cdots (1-\q^n)$ for $n > 0$
and $1$ for $n=0$. Let $(\q;\q)_\infty = \prod_{d=0}^\infty
(1-q^{d+1})$. 
Then
\begin{equation*}
   \sum_n e(\bM_{{}^m\zeta}(r,k,n)) \q^n
   = \frac1{(\q;\q)_\infty^r}
   \sum_{\substack{k_1 + \cdots + k_r = -k \\
       k_\alpha \ge -m}}
   \frac{\q^{(\vec{k},\vec{k})/2}}{\prod_{\alpha=1}^r (\q;\q)_{k_\alpha
       + m}},
\end{equation*}
where $(\vec{k},\vec{k}) = \sum_{\alpha,\beta} (k_\alpha -
k_\beta)^2/(2r)$.
\end{NB}

\subsection{Tangent space -- rank $1$ case}

We first state the weight decomposition of the tangent space in the
rank $1$ case.

Let $(Y,S)$ be a pair of a Young diagram and marked removable boxes
corresponding to a torus fixed point $(E,\Phi)$ in $\bMm{0}(c)$.
We call a box in $Y$ {\it irrelevant\/} if
\begin{aenume}
  \item the upmost box in the column is marked, and
  \item the rightmost box in the row is marked.
\end{aenume}
In Figure~\ref{fig:Young2} the boxes with $\hs$ are marked removable
boxes, and the boxes with $\hs$ or $\sps$ are irrelevant boxes. We
call a box {\it relevant\/} if it is not irrelevant.
Then
\begin{Proposition}
We have
  \begin{equation*}
     \ch T_{(E,\Phi)} \bMm{0}(c)
     = \sum_{s} \left( t_1^{-l_{Y}(s)} t_2^{a_{Y\setminus S}(s)+1} 
       + t_1^{l_{Y\setminus S}(s)+1} t_2^{-a_{Y}(s)} \right),
  \end{equation*}
  where the summation runs over all relevant boxes $s$ in $Y$, and
  $Y\setminus S$ is the Young diagram obtained by removing all marked
  boxes from $Y$.
\end{Proposition}

The proof will be given in more general higher rank cases in
\propref{prop:Ext}.

\begin{figure}[htbp]
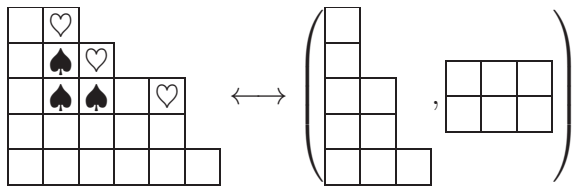

  \centering
\begin{equation*}
\Yvcentermath1
\young(\hf\hs,\hf\sps\hs,\hf\sps\sps\hf\hs,\hf\hf\hf\hf\hf,\hf\hf\hf\hf\hf\hf)
\longleftrightarrow
\left(
\young(\hf,\hf,\hf\hf,\hf\hf,\hf\hf\hf),
\young(\hf\hf\hf,\hf\hf\hf)
\right)
\end{equation*}
  \caption{Marked removable boxes $\hs$ and a pair of Young diagrams}
  \label{fig:Young2}
\end{figure}

We have $\bMz(c e^{-m[C]}) \cong \bMm{m}(c)$, so we may assume $c_1 = 0$. Then%
\begin{NB}
Let $c = 1 - N\pt$. Then we have 
\begin{gather*}
  (c_1(c e^{-m[C]}), [C]) = m,
\\
  p_*(c e^{-m[C]}) 
  \begin{aligned}[t]
    &= p_*\left(1 - m[C] - (N + \frac{m^2}2)\pt\right)
\\
    &= p_*(\left(1 - m(e + \frac12 \pt) - (N + \frac{m^2}2)\pt\right)
\\
    &= 1 - (N + \frac{m(m+1)}2)\pt
  \end{aligned}
\end{gather*}
\end{NB}

\begin{Corollary}\label{cor:rank1Betti}
Let $c_N = 1 - N\pt$.
The Poincar\'e polynomial of $\bMm{m}(c_N)$ is given by
\begin{equation*}
    \sum t^{2(N + m -l(Y))}
\end{equation*}
where the summation runs over all Young diagrams with $m$ marked
removable boxes with $|Y|=N+m(m+1)/2$, and $l(Y)$ is the number of
columns in $Y$.
\end{Corollary}

\begin{proof}
By the same argument as in \cite[Cor.~5.10]{Lecture}, it is enough
to count the dimension of sum of weight spaces which satisfy either
of the followings:
\begin{enumerate}
\item the weight of $t_2$ is negative,
\item the weight of $t_2$ is $0$ and the weight of $t_1$ is negative.
\end{enumerate}
The second possibility cannot happen.
Therefore it is number of relevant boxes with $a_Y(s) > 0$. This is
equal to $|Y| - m(m-1)/2 - l(Y) = N + m -l(Y)$.
\end{proof}

\subsection{A combinatorial bijection}\label{subsec:comb}
In \cite[\S3]{part1} we parametrized torus fixed points in the Hilbert
schemes of points on the blowup $\widehat{\C}^2$ via a pair of
partitions. The parametrization in the previous subsection must be
related to this parametrization in the limit $m\to \infty$. This will
be done in this subsection.

Let us consider two sets $A$, $B$ consisting of
\begin{enumerate}
\item pairs of Young diagrams $Y$ and sets $S$ of $m$ marked removable
  boxes such that $|Y| - m(m+1)/2 = N$,
\item pairs of Young diagrams $(Y^1,Y^2)$ such that $Y^2$ has at most
$m$ columns and $|Y^1| + |Y^2| = N$
\end{enumerate}
respectively. Note that $m$ is fixed here, so it must be included in
the set $B$ if we move it.
We construct a bijection between $A$ and $B$.
\begin{NB}
  If $Y$ has $m$ removable boxes, then $|Y| \ge m(m+1)/2$. Therefore
  $N\ge 0$. The equalitity holds if and only if $Y$ corresponds to
  the partition $(m,m-1,\cdots,1)$.
\end{NB}

Take a Young diagram with marked boxes from $A$. We define a Young
diagram $Y^1$ by removing all columns containing marked boxes from
$Y$. (And we shift columns to the left to fill out empty columns.)
We define another Young diagram $Y^2$ as follows. We first define
a Young diagram $Y'$ consisting of columns removed from $Y$ when we got
$Y^1$. Then we remove all the irrelevant boxes from $Y'$. (And we move
boxes to down to fill out empty spots.) Call the resulted Young
diagram $Y^2$.
See Figure~\ref{fig:Young2} where the boxes with $\hs$ are marked
removable boxes, and the boxes with $\hs$ or $\sps$ are irrelevant
boxes.
This $Y^2$ is a Young diagram
which has at most $m$ columns and $|Y^1|+|Y^2| = N$. 
Thus we have a map from $A$ to $B$.

Conversely from $(Y^1,Y^2)\in B$ we can construct a Young diagram $Y$
with marked removable boxes by the reverse procedure. Namely we add
$m$ boxes to the first (=leftmost) column of $Y^2$, $m-1$ boxes to the
second column, ... Put markings on the top box in each column of
$Y^2$. Then merge two Young diagrams to $Y$ keeping columns.

\begin{NB}
Then $l(Y) - m = l(Y^1)$.  
\end{NB}

\begin{Corollary}\label{cor:rank1Betti2}
Let $C_N = 1 - N\pt$.
The Poincar\'e polynomial of $\bMm{m}(c_N)$ is given by
\begin{equation*}
   P_t(\bMm{m}(c_N)) = \sum t^{2(|Y^1|+|Y^2|-l(Y^1))},
\end{equation*}
where the summation runs over all pairs of Young diagrams
$(Y^1,Y^2)\in B$. Therefore its generating function is
\begin{equation*}
   \sum_{N=0}^\infty P_t(\bMm{m}(c_N)) \q^N
   = \left(\prod_{d=1}^\infty
         \frac{1}{1 - t^{2d-2}\q^{d}}
  \right)
  \left(\prod_{d=1}^{m}
         \frac{1}{1 - t^{2d} \q^{d}}
  \right).
\end{equation*}
\end{Corollary}


\begin{NB}
The left hand side is equal to
\[
   \sum_{(Y^1,Y^2)} (t^2\q)^{|Y^2|} \q^{|Y^1|}t^{2|Y^1|-2l(Y^1))}.
\]
This is equal to the right hand side.
\end{NB}

Let $m\to \infty$. Then $\bMm{m}(c_N)$ becomes the Hilbert schemes
$(\widehat{\C}^2)^{[N]}$ of points on $\widehat{\C}^2$ by \propref{prop:blowup} for $m\gg 0$.
From the above formula we get
\begin{equation*}
   \sum_N P_t((\widehat{\C}^2)^{[N]}) \q^N
   = \left(\prod_{d=1}^\infty
         \frac{1}{1 - t^{2d-2}\q^{d}}
  \right)\left(\prod_{d=1}^\infty
         \frac{1}{1 - t^{2d}\q^{d}}
  \right).
\end{equation*}
This is nothing but G\"ottsche's formula for Betti numbers of
$(\widehat{\C}^2)^{[N]}$. (See e.g., \cite{Lecture}.)

\begin{NB}
I comment out the following, since its higher rank generalization
seems to be difficult.  

Furthermore we have the following:
\begin{Proposition}
  Suppose that a pair of Young diagrams $(Y^1,Y^2)$ is given and
  construct a single Young diagram $Y$ with $m$ marked boxes $S$ as
  above. Let $E_m\in \bMs{m}(1,0,n)$ be the
  corresponding sheaf. Then $\ch T_{E_m} \bMs{m}(1,0,n)$ stabilizes at
  large $m$ to
  \begin{equation*}
     \begin{aligned}[t]
       & \sum_{s\in Y^1}
     t_1^{-l_{Y^1}(s)} \left(\frac{t_2}{t_1}\right)^{a_{Y^1}(s)+1} 
      + t_1^{l_{Y^1}(s) + 1}
      \left(\frac{t_2}{t_1}\right)^{-a_{Y^1}(s)}
\\
      & \qquad+ \sum_{s\in Y^2}
     \left(\frac{t_1}{t_2}\right)^{-l_{Y^2}(s)} t_2^{a_{Y^2}(s)+1} 
      + \left(\frac{t_1}{t_2}\right)^{l_{Y^2}(s) + 1} t_2^{-a_{Y^2}(s)}. 
     \end{aligned}
  \end{equation*}
\end{Proposition}
The character formula is the same as one of the tangent space at the
fixed point corresponding to $(Y^1,Y^2)$ in
$\bM_{{}^\infty\zeta}^{\mathrm{s}}(1,0,n) \cong (\widehat{\C}^2)^{[n]}$.

\begin{proof}
Suppose $m$ is sufficiently large. Then boxes in $Y^1$ appear in $Y$
at the far right. Boxes in $Y^2$ appear at the far up
left. Hence the leg and arm lengths of relevant boxes stabilize as
$m\to \infty$. Let us consider the contributions from the boxes $s$ in
$Y^1$. The arm length for $Y$ (or $Y\setminus S$) does not change,
i.e., $a_Y(s) = a_{Y\setminus S}(s) = a_{Y^1}(s)$. Here we use the
same notation $s$ for a box in $Y^1$, and the corresponding box in
$Y$. On the other hand, columns of lengths $1$, $2$, \dots with marked
top boxes are inserted to $Y^1$, so we have
\begin{equation*}
   l_Y(s) = l_{Y\setminus S}(s) + 1 = l_{Y^1}(s) + a_{Y^1}(s) + 1.
\end{equation*}
So
\begin{equation*}
    t_1^{-l_{Y}(s)} t_2^{a_{Y\setminus S}(s)+1} 
       + t_1^{l_{Y\setminus S}(s)+1} t_2^{-a_{Y}(s)} 
    = t_1^{-l_{Y^1}(s)} \left(\frac{t_2}{t_1}\right)^{a_{Y^1}(s)+1} 
      + t_1^{l_{Y^1}(s) + 1} \left(\frac{t_2}{t_1}\right)^{-a_{Y^1}(s)}.
\end{equation*}
\begin{NB2}
\(
    = t_1^{-l_{Y^1}(s) - a_{Y^1}(s) - 1} t_2^{a_{Y^1}(s)+1} 
      + t_1^{l_{Y^1}(s) + a_{Y^1}(s) + 1} t_2^{-a_{Y^1}(s)} .
\)
\end{NB2}
For a box $s$ in $Y^2$, we have
\begin{equation*}
   l_{Y}(s) = l_{Y\setminus S}(s) = l_{Y^2}(s), \qquad
   a_{Y}(s) = a_{Y\setminus S}(s) + 1 = a_{Y^2}(s) + l_{Y^2}(s) + 1
\end{equation*}
by a similar consideration. Therefore
\begin{equation*}
    t_1^{-l_{Y}(s)} t_2^{a_{Y\setminus S}(s)+1} 
       + t_1^{l_{Y\setminus S}(s)+1} t_2^{-a_{Y}(s)} 
    = \left(\frac{t_1}{t_2}\right)^{-l_{Y^2}(s)} t_2^{a_{Y^2}(s)+1} 
      + \left(\frac{t_1}{t_2}\right)^{l_{Y^2}(s) + 1} t_2^{-a_{Y^2}(s)}.
\end{equation*}
\end{proof}
\end{NB}

\subsection{Tangent space -- general case}

We consider general case.
Let $(Y_\alpha,S_\alpha)$, $(Y_\beta,S_\beta)$ be two pairs of Young
diagrams and marked removable boxes. Let $(E_\alpha,\Phi_\alpha)$,
$(E_\beta,\Phi_\beta)$ be the corresponding framed perverse coherent
sheaves of rank $1$.

For a given pair $(s,s')\in S_\alpha\times S_\beta$, we consider the
box $u$ (resp.\ $u'$) which is the same row as in $s$ (resp.\ $s'$)
and the same column as in $s'$ (resp.\ $s$).
We have
\begin{enumerate}
\item If $a'(s) \le a'(s')$ (i.e., $s'$ sits higher than or equal to
  $s$), then $u\in Y_\beta$, $u'\notin Y_\alpha\setminus S_\alpha$.
\item If $a'(s) > a'(s')$ (i.e., $s'$ sits lower than $s$), then
  $u\notin Y_\beta$, $u'\in Y_\alpha\setminus S_\alpha$.
\end{enumerate}
We say $u$ or $u'$ is {\it irrelevant\/} accordingly. We say a box 
(in $Y_\alpha\setminus S_\alpha$ or $Y_\beta$) is {\it relevant \/}
otherwise.
\begin{figure}[htbp]
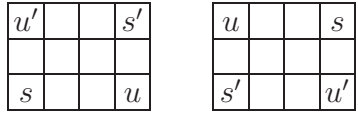

  \centering
\begin{equation*}
\newcommand{\sprime}{
{s'}}
\newcommand{\uprime}{
{u'}}
\Yvcentermath1
\young(\uprime\hf\hf\sprime,\hf\hf\hf\hf,s\hf\hf u)
\qquad
\young(u\hf\hf s,\hf\hf\hf\hf,\sprime\hf\hf \uprime)
\end{equation*}
  \caption{The irrelevant box is $u$ in the first case, and
  $u'$ in the second case.}
  \label{fig:Young3}
\end{figure}

\begin{Proposition}\label{prop:Ext}
We have
  \begin{equation*}
     \ch \Ext^1(E_\alpha,E_\beta(-\linf))
     = \sideset{}{'}\sum_{s\in Y_\alpha\setminus S_\alpha} 
     t_1^{-l_{Y_\beta}(s)} t_2^{a_{Y_\alpha\setminus S_\alpha}(s)+1} 
       + \sideset{}{'}\sum_{t\in Y_\beta} 
       t_1^{l_{Y_\alpha\setminus S_\alpha}(t)+1} t_2^{-a_{Y_\beta}(t)},
  \end{equation*}
where the summation runs over all relevant boxes $s\in
Y_\alpha\setminus S_\alpha$, $t\in Y_\beta$.
\end{Proposition}

\begin{proof}
\begin{NB}
First I give a proof using the quiver description.

Recall that the tangent space of $\bMz$ is given by middle cohomology
of the complex $C^\bullet$ whose cohomology groups vanishes in left
and right (see \cite[(2.6)]{perv}). By the same argument as in
\cite[\S2]{part1}, the complex becomes that of $\hT$-modules after
putting certain $(\C^*)^2$-modules (see below). In particular, the
decomposition $E \cong E_1\oplus \cdots\oplus E_r$ induces
decompositions $W = \bigoplus_\alpha W_\alpha$, $V_0 =
\bigoplus_\alpha V_{0,\alpha}$, $V_1 = \bigoplus_\alpha V_{1,\alpha}$
($\alpha=1,\dots, r$). We have the decomposition of the tangent space
\(
   T_{(E,\Phi)} \bMz\cong \Ext^1(E,E(-\linf))
   \cong \bigoplus_{\alpha,\beta} \Ext^1(E_\alpha,E_\beta(-\linf))
\)
as a $T^{r-1}$-module, and also that of the complex
$C^\bullet = \bigoplus_{\alpha,\beta} C_{\alpha,\beta}^\bullet$ where
\begin{equation*}
   C_{\alpha,\beta}^\bullet:\quad
\begin{matrix}
   \Hom(V_{0,\alpha},V_{0,\beta})
\\
   \oplus
\\ \Hom(V_{1,\alpha},V_{1,\beta})
\end{matrix}
\to
\begin{matrix}
   \Hom(V_{0,\alpha},V_{1,\beta})\\ \oplus \\ Q\otimes \Hom(V_{1,\alpha}, V_{0,\beta}) \\ 
   \oplus \\ \Hom(W_\alpha,V_{0,\beta}) \\ \oplus \\ \Wedge^2 Q\otimes
   \Hom(V_{1,\alpha},W_\beta)
\end{matrix}
\to
\begin{matrix}
  \Wedge^2 Q\otimes \Hom(V_{1,\alpha},V_{0,\beta})
\end{matrix},
\end{equation*}
where $Q$ is the $(\C^*)^2$-module with character $t_1+t_2$.
As each summand has a different weight with respect to $T^{r-1}$, the
middle cohomology of $C_{\alpha,\beta}^\bullet$ must be equal to
$\Ext^1(E_\alpha,E_\beta(-\linf))$.
\begin{NB2}
  I believe that there is a direct proof of the isomorphism of the
  cohomology of the complex and $\Ext^1$.
\end{NB2}

Recall that we computed the cohomology of a similar complex in
\cite[Th.~2.11]{part1}, namely
\begin{equation*}
   C_{\alpha,\beta}^{\prime\bullet}:\quad
\begin{matrix}
   \Hom(V_{1,\alpha},V_{0,\beta})
\end{matrix}
\to
\begin{matrix}
   Q\otimes \Hom(V_{1,\alpha}, V_{0,\beta}) \\ 
   \oplus \\ \Hom(W_\alpha,V_{0,\beta}) \\ \oplus 
   \\ \Wedge^2 Q\otimes \Hom(V_{1,\alpha},W_\beta)
\end{matrix}
\to
\begin{matrix}
  \Wedge^2 Q\otimes \Hom(V_{1,\alpha},V_{0,\beta})
\end{matrix}.
\end{equation*}
We have
\begin{equation*}
   \ch H^1(C^{\prime\bullet}_{\alpha,\beta})
   = \sum_{s\in Y_\alpha\setminus S_\alpha}
   t_1^{-l_{Y_\beta}(s)} t_2^{a_{Y_\alpha\setminus S_\alpha}(s)+1} 
       + \sum_{t\in Y_\beta} t_1^{l_{Y_\alpha\setminus S_\alpha}(t)+1} t_2^{-a_{Y_\beta}(t)},
\end{equation*}
where $Y_\alpha\setminus S_\alpha$ (resp.\ $Y_\beta$) corresponds to
$V_{1,\alpha}$ (resp.\ $V_{0,\beta}$). We have
\begin{equation*}
   \ch H^1(C^\bullet_{\alpha,\beta})
   = \ch H^1(C^{\prime\bullet}_{\alpha,\beta})
     - \ch(\Hom(V_{0,\alpha} - V_{1,\alpha}, V_{0,\beta} - V_{1,\beta})).
\end{equation*}
Moreover $V_{0,\alpha} - V_{1,\alpha}$ (resp.\ $V_{0,\beta} -
V_{1,\beta}$) corresponds to marked boxes $S_\alpha$ (resp.\
$S_\beta$), and we have
\begin{equation*}
  \ch(\Hom(V_{0,\alpha} - V_{1,\alpha}, V_{0,\beta} - V_{1,\beta}))
  = \sum_{\substack{s\in S_\alpha,\  s'\in S_\beta}}
  t_1^{l'(s) - l'(s')} t_2^{a'(s) - a'(s')} .
\end{equation*}
\end{NB}%
The space $\Ext^1(E_\alpha,E_\beta(-\linf))$ is a weight space of the
tangent space of $\bMz(c)$ at a $\hT$-fixed point $(E,\Phi) \cong
(E_1,\Phi_1)\oplus \cdots \oplus (E_r,\Phi_r)$.
Since $\bMz(c)$ and $L(c',n)$ are isomorphic by
\thmref{thm:incidence}, the tangent space $\Ext^1(E, E(-\infty))$ of
$\bMz(c)$ at $(E,\Phi)$ is isomorphic to the tangent space of $L(c',n)$
at $((F,\Phi),U)$ corresponding to $(E,\Phi)$.
In the genuine moduli space of sheaves case, the latter was given 
\(
   \Ext^1(F, F')/\End(U)
\)
where $F' := \Ker(F\to U^\vee\otimes\C_0)$. (See the proof of
\lemref{lem:immersion}(1).) In the framed case, it is modified as
\(
   \Ext^1(F, F'(-\linf))/\End(U).
\)
Since the isomorphism $\bMz(c)\cong L(c',n)$ is $\hT$-equivariant, the
weight spaces at fixed points must be respected, so
\(
   \Ext^1(E_\alpha,E_\beta(-\linf))
\)
is isomorphic to
\(
  \Ext^1(F_\alpha, F_\beta'(-\linf))/\Hom(U_\alpha,U_\beta),
\)
where $(F_\alpha,U_\alpha)$ corresponds to the summand $E_\alpha$.

If $(F_\alpha,U_\alpha)$ corresponds to a marked Young diagram
$(Y_\alpha,S_\alpha)$, then the $T^2$-character of
\(
  \Ext^1(F_\alpha, F_\beta'(-\linf))
\)
was computed in \cite[\S2]{part1}:
\begin{equation*}
  \ch \Ext^1(F_\alpha, F_\beta'(-\linf))
  = \sum_{s\in Y_\alpha\setminus S_\alpha}
   t_1^{-l_{Y_\beta}(s)} t_2^{a_{Y_\alpha\setminus S_\alpha}(s)+1} 
       + \sum_{t\in Y_\beta} t_1^{l_{Y_\alpha\setminus S_\alpha}(t)+1} t_2^{-a_{Y_\beta}(t)},
\end{equation*}
where we should notice that $F_\alpha$ corresponds to the Young
diagram $Y_\alpha\setminus S_\alpha$, while $F_\beta'$ corresponds to
$Y_\beta$.

On the other hand, we have
\begin{equation*}
  \ch \Hom(S_\alpha,S_\beta)
  = \sum_{\substack{s\in S_\alpha,\  s'\in S_\beta}}
  t_1^{l'(s) - l'(s')} t_2^{a'(s) - a'(s')} .
\end{equation*}

For a given pair $(s,s')\in S_\alpha\times S_\beta$, we consider the
boxes $u$ and $u'$ explained as above. Then we have
\begin{equation*}
  \begin{gathered}
   l_{Y_\alpha\setminus S_\alpha}(u) + 1 = l'(s) - l'(s'), \qquad
   - a_{Y_\beta}(u) = a'(s) - a'(s'),
\\
   -l_{Y_\beta}(u') = l'(s) - l'(s'), \qquad
   a_{Y_\alpha\setminus S_\alpha}(u') + 1 = a'(s) - a'(s').
  \end{gathered}
\end{equation*}
Therefore we substract the box $u$ from $Y_\beta$, or
$u'$ from $Y_\alpha\setminus S_\alpha$ according to
$u\in Y_\beta$ or $u'\in Y_\alpha\setminus S_\alpha$ to get
the assertion.
\end{proof}

\begin{Corollary}
  The Poincar\'e polynomial of $\bMz(c)$ is given by
\begin{equation*}
  P_t(\bMz(c)) = 
    \sum_{(\vec{m},\vec{Y^1},\vec{Y^2})} \prod_{\alpha=1}^r
      t^{2(r|Y_\alpha^1| + r|Y_\alpha^2| - \alpha l(Y_\alpha^1))}
    \prod_{\alpha < \beta} t^{(m_\alpha - m_\beta)(m_\alpha - m_\beta-1)},
\end{equation*}
where the summation runs over $r$-tuples
\(
   (\vec{m},\vec{Y^1},\vec{Y^2}) = 
   ((m_1,Y^1_1,Y^2_1),\dots, (m_r,Y^1_r,Y^2_r))
\)
of triples of nonnegative integers and two Young diagrams
such that
\(
   \sum_\alpha m_\alpha = (c_1,[C]),
\)
the number of columns of $Y^2_\alpha$ is at most $m_\alpha$
\rom($\alpha=1,\dots,r$\rom),
\begin{NB}
$m_\alpha = |S_\alpha|$  
\end{NB}%
and
\(
    \sum_\alpha |Y^1_\alpha| + |Y^2_\alpha|
    \begin{NB}
      = \sum_\alpha |Y_\alpha| - m_\alpha(m_\alpha+1)/2
      = - \int_{\bX} \ch_2 + \frac12 (c_1,[C])
         - \sum m_\alpha^2/2 - \sum m_\alpha/2
      = \Delta + 1/(2r)\, (c_1,[C])^2 - \sum m_\alpha^2/2
      = \Delta
        + 1/(2r) \left[ \sum_{\alpha,\beta} m_\alpha m_\beta
        - r\sum m_\alpha^2 \right]
    \end{NB}
      = \Delta(c) - 1/(4r) \sum_{\alpha, \beta} (m_\alpha - m_\beta)^2.
\)
Here $\Delta(c) = \int_{\bX} [ - \ch_2 + 1/(2r)\, c_1^2]$.

And their generating function \textup(for fixed
$r$, $c_1$\textup) is given by
\begin{multline*}
    \sum_{c} P_t(\bMz(c)) \q^{\Delta(c)}
\\
    = \sum_{\substack{m_\alpha \ge 0 \\ \sum m_\alpha = (c_1,[C])}}
    \prod_{\alpha=1}^r \left( \prod_{d=1}^\infty
              \frac1{1-t^{2(rd-\alpha)}\q^d}
    \times\! 
    \prod_{d=1}^{m_\alpha} \frac1{1-t^{2rd}\q^d} \right)
    \times
    t^{-2\langle \vec{m},\rho\rangle}
    (t^{2r} \q)^{(\vec{m},\vec{m})/2},
\end{multline*}
where
\(
\langle \vec{m},\rho\rangle = \sum_{\alpha < \beta} (m_\alpha - m_\beta)/2,
\)
\(
  (\vec{m},\vec{m}) = 1/(2r) \sum_{\alpha,\beta} (m_\alpha - m_\beta)^2.
\)
\end{Corollary}

\begin{proof}
The torus fixed points in $\bMz(c)$ is parametrized by
$r$-tuples $((Y_1,S_1),\dots,(Y_r,S_r))$ of pairs of Young diagrams
with marked removable boxes with 
\(
   \sum_\alpha |S_\alpha| = (c_1,[C]),
\)
\(
    \sum_\alpha |Y_\alpha|  = - \int_{\bX} \ch_2 + \frac12 (c_1,[C]).
\)
Moreover such $r$-tuples correspond to $r$-tuples of
triples of nonnegative intergers and two Young diagrams 
\(
   ((m_1,Y^1_1,Y^2_1),\dots, (m_r,Y^1_r,Y^2_r))
\)
as above by \subsecref{subsec:comb}, where
$m_\alpha = |S_\alpha|$.

As in \cite[Th.~3.8]{NY2} we take a one parameter subgroup
\(
   \lambda\colon \C^*\to \hT
\)
with
\[
   \lambda(t) = (t^{N_1},t^{N_2},t^{n_1},\dots,t^{n_r})
\]
and
\[
   N_2 \gg n_1 > n_2 > \cdots > n_r \gg N_1 > 0.
\]
Then we compute the dimension of negative weight spaces of the tangent
space at each fixed point. Thus we count those weight spaces such that
\renewcommand{\descriptionlabel}[1]{\hspace\labelsep \upshape #1}
\begin{description}
\item[$(1)$] weight of $t_2$ is negative,
\item[$(2)$] weight of $t_2$ is zero and weight of $e_1$ is negative,
\item[$(3)$] weight of $t_2$, $e_1$ are zero and weight of $e_2$ is
negative,
\item[$(4)$] weight of $t_2$, $e_1$, $e_2$ are zero and weight of $e_3$ is
negative,
\item[$\cdots$]
\item[$(r+1)$] weight of $t_2$, $e_1$, $e_2$, \dots, $e_{r-1}$ are
zero and weight of $e_r$ is negative,
\item[$(r+2)$] weight of $t_2$, $e_1$, $e_2$, \dots, $e_{r}$ are
zero and weight of $t_1$ is negative.
\end{description}

We have decomposition of the tangent space
$T_{(E,\Phi)} \bMz = \bigoplus_{\alpha,\beta}
\Ext^1(E_\alpha,E_\beta(-\linf))$. The $T^{r}$-weight of the summand
$\Ext^1(E_\alpha,E_\beta(-\linf))$ is given by $e_\beta
e_\alpha^{-1}$. Therefore in the summand $\alpha = \beta$, the
total dimension of negative weight spaces is
\(
    2(|Y^1_\alpha| + |Y^2_\alpha| - l(Y^1_\alpha))
\)
as in the rank $1$ case (see \corref{cor:rank1Betti2}).
In the summand $\alpha < \beta$, we compute the total dimension of
weight spaces whose $t_2$-weight is nonpositive. It is given by
\[
   2\left[|Y_\beta|
   - \# \{ (s,s')\in S_\alpha\times S_\beta \mid a'(s) \le a'(s') \}
   \right]
\]
by the same argument as in \corref{cor:rank1Betti}. 
In the summand $\alpha > \beta$, we get
\[
   2\left[|Y_\beta| - l(Y_\beta)
   - \# \{ (s,s')\in S_\alpha\times S_\beta \mid a'(s) < a'(s') \}
   \right].
\]
We combine the last term for $\alpha < \beta$ and the corresponding
term for $\alpha \leftrightarrow \beta$ to have
\[
  \# \{ (s,s')\in S_\alpha\times S_\beta \mid a'(s) \le a'(s') \}
  + \# \{ (s',s)\in S_\beta \times S_\alpha \mid a'(s') < a'(s) \}
  = m_\alpha m_\beta.
\]
We also note
\begin{equation*}
  \begin{split}
  & 
  |Y_\beta| = |Y^1_\beta| + |Y^2_\beta| + \frac12 m_\beta(m_\beta+1), 
\qquad
  l(Y_\beta) = l(Y^1_\beta) + m_\beta.  
  \end{split}
\end{equation*}
So in total we have
\begin{equation*}
\begin{split}
   & 2 \sum_{\alpha=1}^r \left[ r (|Y^1_\alpha| + |Y^2_\alpha|) 
     - \alpha l(Y^1_\alpha) + \frac{r-1}2 m_\alpha(m_\alpha+1)\right]
   - 2\sum_{\alpha < \beta} \left( m_\alpha + m_\alpha m_\beta \right)
\\
   = \; & 2 \sum_{\alpha=1}^r \left[ r (|Y^1_\alpha| + |Y^2_\alpha|) 
     - \alpha l(Y^1_\alpha) \right]
   + \sum_{\alpha < \beta} ( m_\alpha - m_\beta )( m_\alpha - m_\beta - 1)
   .
\end{split}
\end{equation*}
\begin{NB}
\begin{equation*}
\begin{split}
  & \sum_{\alpha < \beta} ( m_\alpha - m_\beta )( m_\alpha - m_\beta - 1)
  =   \sum_{\alpha < \beta} (m_\alpha^2 - 2 m_\alpha m_\beta +
  m_\beta^2 - m_\alpha + m_\beta)
\\
  =\; & (r-1)\sum_\alpha m_\alpha^2 
  - 2 \sum_{\alpha< \beta} m_\alpha m_\beta 
  + \left[ -(r-1) m_1 - (r-3) m_2 - \cdots + (r-3) m_{r-1} + (r-1) m_r
  \right]
\\
  =\; & (r-1)\sum_\alpha m_\alpha^2 
  - 2 \sum_{\alpha< \beta} m_\alpha m_\beta 
  + (r-1) \sum_\alpha m_\alpha - 2 \sum_{\alpha<\beta} m_\alpha.
\end{split}
\end{equation*}
\end{NB}%
From this we get the formula.
\end{proof}

For general $\bMm{m}(c)$, we just need to apply the formula for
$\bMm{0}(c e^{-m[C]})$ and replace $m_\alpha$ by $m+k_\alpha$:

\begin{Corollary}\label{cor:higherrk}
\begin{multline*}
    \sum_{\text{\rm $c$\,:\, $r$, $(c_1,[C])$ fixed}} P_t(\bMm{m}(c)) \q^{\Delta(c)}
\\
    = \sum_{\substack{k_\alpha \ge -m \\ k_1 + \cdots + k_r = (c_1,[C])}}
    \prod_{\alpha=1}^r \left( \prod_{d=1}^\infty
              \frac1{1-t^{2(rd-\alpha)}\q^d}
    \times\! 
    \prod_{d=1}^{m+k_\alpha} \frac1{1-t^{2rd}\q^d} \right)
    \times
    t^{-2\langle \vec{k},\rho\rangle}
    (t^{2r} \q)^{(\vec{k},\vec{k})/2}.
\end{multline*}
\end{Corollary}

In the limit $m\to\infty$, we recover the formula \cite[Cor.~3.10]{NY2}.

\begin{NB}
Let $(a;\q)_n$ be $(1-a)(1-a\q)(1-a\q^2)\cdots (1-a\q^{n-1})$ for $n > 0$
and $1$ for $n=0$. Let $(a;\q)_\infty = \prod_{d=0}^\infty
(1-a\q^{d})$. 
Then
\begin{equation*}
   \sum_n P_t(\bM_{{}^m\zeta}(r,k,n)) \q^n
   = \sum_{\substack{k_\alpha \ge -m \\ k_1 + \cdots + k_r = -{k}}}
    \frac{t^{-2\langle \vec{k},\rho\rangle}
    (t^{2r} \q)^{(\vec{k},\vec{k})/2}}
    {\prod_{\alpha=1}^r \left\{
              {(t^{2(r-\alpha)}\q; t^{2r}\q)_\infty}
    \times\! 
    {(t^{2r}\q;t^{2r}\q)_{m+k_\alpha}} \right\}}.
\end{equation*}
\end{NB}

\end{document}